\documentclass[11pt, a4paper]{article}
\usepackage{t1enc}
\usepackage[latin1]{inputenc}
\usepackage[english]{babel}
\usepackage{amsmath,amsthm}
\numberwithin{equation}{section}
\usepackage{amsfonts}
\usepackage{latexsym}
\usepackage{graphicx}
\usepackage{float}
\usepackage[natural]{xcolor}
\usepackage{algorithm}
\usepackage{algorithmic}
\usepackage{enumerate,enumitem}
\usepackage{multirow}
\usepackage[colorlinks,linkcolor=blue]{hyperref}
\usepackage{lineno}
\usepackage{setspace}
\usepackage[title]{appendix}
\usepackage{todonotes}
\usepackage{amssymb}
\usepackage{listings}

\usepackage[
  a4paper,
  left=2.2cm,
  right=2.2cm,
  top=1.5cm,
  bottom=1.9cm,
  headheight=12pt,
  headsep=7pt,
  includehead
]{geometry}
\usepackage{titlesec}
\usepackage{fancyhdr}
\usepackage{verbatim}

\counterwithin{figure}{section}
\newcounter{shared}
\setlist{
  itemsep=0pt,
  topsep=3pt,
  parsep=0pt,
  partopsep=0pt
}

\titleformat{\section}
  {\centering\normalfont\scshape}
  {\thesection.}{0.5em}{}

\titleformat{\subsection}[runin]
  {\normalfont\bfseries}
  {\thesubsection.}{0.5em}{}[.]

\titleformat{\subsubsection}[runin]
  {\normalfont\bfseries}
  {\thesubsubsection.}{0.5em}{}[.]

\newcommand{\compactsectionspacing}{%
  \titlespacing*{\section}
    {0pt}{11pt plus 2pt minus 2pt}{5pt}%
  \titlespacing*{\subsection}
    {0pt}{8pt plus 2pt minus 1pt}{0.6em}%
  \titlespacing*{\subsubsection}
    {0pt}{6pt plus 1pt minus 1pt}{0.6em}%
}
\compactsectionspacing

\newtheorem{conjecture}[shared]{Conjecture}

\newtheorem{theorem}{Theorem}[section]
\newtheorem*{thm-non}{Theorem}

\newtheorem{proposition}[theorem]{Proposition}
\newtheorem{lemma}[theorem]{Lemma}
\newtheorem{corollary}[theorem]{Corollary}

\newtheorem{claim}[theorem]{Claim}

\newtheorem{problem}[theorem]{Problem}
\theoremstyle{definition}
\newtheorem{definition}[theorem]{Definition}
\newtheorem*{defn-non}{Definition}

\newcounter{propcounter}
\newenvironment{poc}
  {\begin{proof}[Proof of claim]}
  {\end{proof}}

\newcommand{\I}[1]{\mathbb{#1}}

\newcommand{\eps}{\varepsilon}

\title{\LARGE The Erd\H{o}s--Gallai bound for consecutive even cycle lengths}

\hypersetup{pdftitle={The Erdos--Gallai bound for consecutive even cycle lengths},
  pdfauthor={Yaobin Chen, Hong Liu, Xia Wang, Xin Wei, and Fan Yang},
  pdfsubject={Extremal graph theory},
  pdfkeywords={consecutive even cycle lengths, sublinear expanders, clique subdivisions, adjusters}
}

\author{
Yaobin Chen \quad Hong Liu \quad Xia Wang \quad Xin Wei \quad Fan Yang
\thanks{All authors are supported by the Institute for Basic Science
(IBS-R029-C4). Xia Wang is also affiliated with the School of Mathematics,
    Shandong University, Jinan, China, and supported by the China Scholarship Council. 
Fan Yang is also supported by the National Natural Science Foundation
of China (12301447), the Natural Science Foundation of Shandong Province
(ZR2024QA056), and the China Scholarship Council. Emails: \texttt{\{ybchen,hongliu,weixinma\}@ibs.re.kr}, \texttt{xiawang@mail.sdu.edu.cn},
\texttt{fyang@sdu.edu.cn}}\\[0.6em]
\small Extremal Combinatorics and Probability Group (ECOPRO),\\[-0.1em]
\small Institute for Basic Science (IBS), Daejeon, South Korea\\[0.4em]}
\date{}

\begin{document}
\setlength{\abovedisplayskip}{7pt plus 2pt minus 3pt}
\setlength{\belowdisplayskip}{7pt plus 2pt minus 3pt}
\setlength{\abovedisplayshortskip}{4pt plus 2pt}
\setlength{\belowdisplayshortskip}{4pt plus 2pt minus 2pt}
\maketitle

\begin{abstract}
Erd\H{o}s and Gallai in 1959 proved the seminal result that every $n$-vertex
graph with no cycle of length at least $2t+2$ has at most
$\tfrac{2t+1}{2}(n-1)$ edges. We prove the extension that, for every sufficiently
large $t$, the same quantity is also the sharp extremal bound for
graphs with no $t$ consecutive even cycle lengths, resolving a
conjecture of Verstra\"ete. Thus, at the Erd\H{o}s--Gallai threshold,
forcing an entire interval of even cycle lengths costs no more than
forcing its longest member. More precisely, every $n$-vertex graph
$G$ with
  $e(G)\ge \tfrac{(2t+1)(n-1)}2$
\begin{itemize}
    \item either contains $t$ consecutive even cycle lengths,
    \item or equality holds
and $G$ is connected with every block isomorphic to $K_{2t+1}$.
\end{itemize}
 
As consequences, for every sufficiently large even $k$ we
determine the sharp edge thresholds forcing a cycle of length
$0\pmod k$ or $2\pmod k$, answering questions of Bai, Grzesik, Li, and Prorok and of Gao, Li, Ma and Xie, respectively, for sufficiently large even $k$. The proof develops a stability-enhanced sublinear expander method. Its
main new ingredient is a dense-case decomposition that recovers the
lengths lost in the expander extraction by combining a flexible dense
core with rooted cycle families in the vertices outside the core.
\end{abstract}

\section{Introduction}\label{sec:intro}

The classical Erd\H{o}s--Gallai theorem~\cite{Erdos-cycle}, proved in 1959,
provides a sharp bound on the number of edges guaranteeing a long cycle. It
states that every \(n\)-vertex graph with more than
$\tfrac{2t+1}{2}(n-1)$ edges contains a cycle of length at least $2t+2$. The set of cycle lengths of a graph records substantially more information
than the existence of one long cycle. A central theme in extremal graph
theory is to understand which additive or arithmetic patterns must occur in
this set: nearby lengths, long arithmetic progressions, prescribed residue
classes, or sparse multiplicative sequences. This viewpoint connects
classical extremal problems for long cycles with questions of a markedly
number-theoretic character; see, for example,~\cite{Fan2002,GaoMa2020,GyarfasKomlosSzemeredi1984,LiuMa2018,Ma2016,SudakovVerstraete2008,
Verstraete2016Survey}.

The problem studied here lies exactly at this interface. We say that a graph
contains $t$ \emph{consecutive even cycle lengths} if, for some integer
$a\ge2$, it contains cycles of lengths
\[
  2a,2a+2,\ldots,2a+2(t-1).
\]
The striking assertion of a conjecture of Verstra\"ete is
that, at precisely the same threshold, one can force not merely one long cycle but an entire interval of $t$ even cycle lengths. Thus the cost of
forcing this highly structured family should be no greater than the cost of
forcing its longest member.

The study of clustered cycle lengths goes back to a question of Erd\H{o}s~\cite{Erdos-conj}
asking whether every graph with minimum degree at least three contains two cycles whose lengths differ
by one or two. Bondy and Vince~\cite{BondyVince1998} proved a
stronger form of this assertion. H\"aggkvist and Scott~\cite{HaggkvistScott1998}
then initiated the systematic study of arithmetic progressions of cycle
lengths, and Verstra\"ete~\cite{Verstraete2000} showed that every bipartite
graph with average degree at least $4k$ and girth $g$ contains
$(g/2-1)k$ consecutive even cycle lengths. In general graphs, a theorem
recorded in~\cite{Verstraete2016Survey} gives $t$ consecutive even cycle
lengths from an edge bound of order $3tn$. The aforementioned conjecture asks for the
sharp edge bound together with the exact equality
structure.

The use of an average-degree hypothesis is essential and is also the main
source of difficulty. Under a minimum-degree condition, admissible-path
methods provide strong local control and can be rooted at prescribed
vertices; see~\cite{ChibaOtaYamashita2023,GaoHuoLiuMa2022}. At the present
edge threshold, however, the average degree is only about $2t+1$, and the
standard deletion argument produces a subgraph with minimum degree only about
$t+1$. This factor-two loss is fatal: at that scale, previous  minimum-degree
theorems yield only about half of the required interval. Moreover, average
degree permits the edges to be distributed very unevenly among dense blocks,
sparse separators, and vertices outside the dense core. A proof of the exact
bound must therefore exploit density that cannot be localized in one
high-minimum-degree subgraph.

Consecutive even cycle lengths also provide a natural mechanism for modular cycle
problems: an interval of sufficiently many even integers contains every even
residue class modulo an even modulus. This perspective was developed by
Sudakov and Verstra\"ete~\cite{SudakovVerstraete2017}. It is complementary to
the multiplicative cycle-spectrum problem of Erd\H{o}s, solved by Liu and
Montgomery~\cite{Liu-Mon-JAMS}, concerning cycles whose lengths are powers of
two. Recent work has resolved several divisibility problems, including the
odd-modulus extremal problem and Dean's conjecture~\cite{BaiGrzesikLiProrok2025,LuoMaZhao2026}, while exact results are known
for several small moduli~\cite{ChenSaito1994,DeanKanekoOtaToft1991,DeanLesniakSaito1993,GyoriEtAl2026,Saito1992}.

\subsection{Main result}

Let $G$ be connected and suppose that every block of $G$ is a copy of
$K_{2t+1}$. Every cycle is contained in one block, so the even cycle lengths
in $G$ are precisely $4,6,\ldots,2t$; in particular, $G$ has no $t$
consecutive even cycle lengths. Moreover, $e(G)=\tfrac{(2t+1)(|G|-1)}2$.
This is also the extremal construction attaining equality in the
Erd\H{o}s--Gallai theorem.
Motivated by it, Verstra\"ete proposed the following conjecture.

\begin{conjecture}[\cite{Verstraete2016Survey}]\label{conj:verstraete}
Let $t$ be a positive integer. If an $n$-vertex graph $G$ does not contain
$t$ consecutive even cycle lengths, then
\[
  e(G)\le \tfrac{(2t+1)(n-1)}2.
\]
Equality is attained by connected graphs whose blocks are copies of
$K_{2t+1}$.
\end{conjecture}

The case $t=1$ is the classical extremal theorem for graphs without even
cycles; see~\cite{DeanLesniakSaito1993}. Gao, Li, Ma and
Xie~\cite{GaoLiMaXie2024} proved the case $t=2$, and Li, Pan and
Shi~\cite{LiPanShi2025} subsequently determined the extremal graphs for all
congruence classes of $n$ in the stronger problem. No result covered
all orders $n$ for an unbounded family of values of $t$.

Our main theorem proves Conjecture~\ref{conj:verstraete} for sufficiently large $t$ and,
in fact, gives a slightly stronger interval statement. To formulate it, we
formally regard an edge as a cycle of length two.

\begin{theorem}\label{thm:main}
For every sufficiently large positive integer $t$, every $n$-vertex graph
$G$ with
\[
  e(G)\ge \tfrac{(2t+1)(n-1)}2
\]
satisfies one of the following:
\begin{enumerate}[label=\rm(\arabic*)]
  \item\label{maintheorem1} $G$ contains $t+1$ consecutive even cycle lengths when an edge is
  regarded as a cycle of length two;
  \item\label{maintheorem2} equality holds, $G$ is connected, and every block of $G$ is
  isomorphic to $K_{2t+1}$.
\end{enumerate}
\end{theorem}

Consequently, every graph with no $t$ consecutive even cycle lengths has at
most $\tfrac{(2t+1)(n-1)}2$ edges, with equality precisely for the graphs in
\ref{maintheorem2}. The strengthening in \ref{maintheorem1} by one term is useful in the modular applications below.

\subsection{Applications to cycles modulo an even integer}

We next apply the strengthened form of Theorem~\ref{thm:main} to modular
cycle problems. For odd moduli, Bai, Grzesik, Li and
Prorok~\cite{BaiGrzesikLiProrok2025} determined the sharp extremal bound and
showed that the extremal behavior is bipartite. They conjectured a different
picture for even moduli, with clique-block extremal graphs.

\begin{conjecture}[\cite{BaiGrzesikLiProrok2025}]\label{conj:k-div}
Let $k\ge6$ be even and $n\ge k-1$. If
\[
  e(G)>\tfrac{k-1}{2}(n-1),
\]
then $G$ contains a cycle whose length is divisible by $k$. Moreover, when
$k-2$ divides $n-1$, the connected graphs whose blocks are copies of
$K_{k-1}$ are the unique extremal graphs.
\end{conjecture}

Applying Theorem~\ref{thm:main} with $t=(k-2)/2$, we obtain $k/2$
consecutive even cycle lengths in the extended convention. If the interval starts
with two, it contains the actual cycle length $k$; otherwise all its terms
are cycle lengths and they represent every even residue modulo $k$. 

\begin{corollary}\label{cor:c0even}
Conjecture~\ref{conj:k-div} is true for sufficiently large $k$.
\end{corollary}

The same argument (but with $t=k/2$) gives a sharp bound for the residue class $2$ and
answers a question of Gao, Li, Ma and Xie~\cite{GaoLiMaXie2024} for sufficiently large even moduli.

\begin{corollary}\label{cor:c2even}
For sufficiently large even integer $k$, every $n$-vertex graph $G$
with
\[
  e(G)>\tfrac{k+1}{2}(n-1)
\]
contains a cycle whose length is congruent to $2\pmod k$. Moreover, when
$k$ divides $n-1$, the connected graphs whose blocks are copies of
$K_{k+1}$ are the unique extremal graphs.
\end{corollary}

\subsection{Stability-enhanced sublinear expanders}
Roughly speaking, writing $d=2t+2$, our goal is to find cycles of $d/2$ consecutive even lengths in a graph with average degree $\Omega(d)$.
Our proof uses sublinear expanders, introduced by Koml\'os and
Szemer\'edi~\cite{KS1,KS2}. A key feature of this method is that one can
extract a sublinear expander from any graph while retaining almost the same average degree. This makes the method particularly useful in embedding  problems of specified structures, like minors,  subdivisions, and cycles of prescribed lengths. See, for example,
\cite{Liu-balanced-sub,Yang1,LiuC4,Liu-Mon-JAMS,balanced1,
Montgomery2015,Wang,Yang2}. We highlight here the main new ingredients that we develop
beyond the existing sublinear expander theory, and refer the reader to
Section~\ref{subsec:proof-overview} for a more detailed proof overview.

At the sharp Erd\H{o}s--Gallai bound, we first extract a sublinear expander $H$ with average degree $(1-o(1))d$ from the original graph. Note that $|H|$ may range from $O(d)$ to values much larger than $d$, and the different regimes for $|H|$ relative to $d$ pose rather different challenges.

The main challenges, and the bulk of the proof, arise in the dense expander regime~$|H|=O(d)$ (see Section~\ref{sec:dense-expander}). The first obstruction is that we have an $o(d)$ loss in degree when passing to $H$. It could be the case that $H$ is a clique-like
graph on $(1-o(1))d$ vertices,  providing almost all of the desired interval but still leaves $o(d)$ missing lengths, 
which can only be recovered by using vertices not in the expander. 
The second obstruction comes from  a competing bipartite-type structure. Apart from the clique-block extremal configuration, the expander may be close to a complete bipartite graph with one part of order $d/2$. Since the two structures are indistinguishable at the $o(dn)$ scale 
and can not be separated by the expansion alone, 
a stability analysis is therefore needed essentially.

Throughout the argument, we work within a minimal counterexample. With the aid of the regularity lemma, we develop a cycle-augmenting process and prove a rooted cycle lemma to overcome the difficulties illustrated above. At the beginning, the regularity lemma gives a \emph{dense core} in the expander which already provides all but $o(d)$ of the required cycle lengths. Our novelty lies in showing that, by running the augmenting process to increase the length of cycles, either the desired cycles have already been found, or some vertices outside the expander $H$ can induce a subgraph with minimum degree $\Omega(d)$. The other ingredient, the rooted cycle lemma, may be of independent interest:  every vertex in a $2$-connected graph with minimum degree at least $d$ lies in $d/2-O(1)$ consecutive even cycles. By contracting the dense core and resolving the parity issue, it is possible to join the cycle family arising from the vertices outside the expander to the flexible path family in the core, thereby obtaining the whole interval of $d/2$ consecutive even cycle lengths.

When the expander is sparse $|H|\gg d$, we can find much more than $d/2$ consecutive even cycle lengths by using several tools, such as adjusters, units, webs, and clique subdivisions. 
Adjusters, introduced by the second author and Montgomery~\cite{Liu-Mon-JAMS}, provide paths with a controlled range of lengths under two requirements: graphs should be bipartite and free of large clique $1$-subdivisions. If we directly use the previous approaches by passing to a bipartite subgraph, the expander may lose half of the average degree, while for the second requirement, the clique $1$-subdivision can not provide enough controlled path lengths. Instead, we develop an \emph{odd adjuster} to balance the parity of the path lengths, which can change path lengths with pace $1$ when the graph is non-bipartite.  On the other hand, when adjusters fail to work, there must exist a large clique $1$-subdivision. Observe that such a  clique $1$-subdivision itself already gives most of the required lengths. For the remaining missing lengths, inspired by our argument in the dense expander case, we show that there exists a dense subgraph (with average degree
$\Omega(d)$) after deleting the branch vertices of the clique $1$-subdivision. We then extract another expander from this subgraph to supply the
missing lengths (see
Section~\ref{section-spar}).

\medskip
\noindent
\textbf{Organization.}
Section~\ref{sec:pre} introduces the notation and preliminary tools.
Section~\ref{subsec:proof-overview} reduces the main theorem to the dense,
intermediate, and sparse expander lemmas and outlines their proofs. These
three regimes are treated in Sections~\ref{sec:dense-expander},
\ref{section-med}, and~\ref{section-spar}, respectively. We conclude in
Section~\ref{sec:con}.
\section{Preliminaries}\label{sec:pre}
\subsection{Notation}
Given a graph $G=(V,E)$, write $|G|:=|V(G)|$, and denote by $\Delta(G),\delta(G),d(G)$ the maximum degree, minimum degree, and average degree of $G$, respectively. For $x\in \mathbb{N}$, an \emph{$x$-star} is a star with $x$ leaves. Let $S$ be a star. We call its center the \emph{core} of $S$, write $\mathsf{Int}(S)$ for the singleton consisting of its core, and write $\mathsf{Ext}(S)$ for its set of leaves. Given a set $W\subseteq V(G)$, denote its \emph{external neighborhood} by $N_G(W):=\{u\notin W: uv\in E(G) \ \text{for} \ \text{some} \ v\in W\}$. Furthermore, set $N_G^{0}(W):=W$ and $N_G^1(W):=N_G(W)$ and, for each $i\geq 1$, define $N_G^{i+1}(W):=N_G(N_G^{i}(W))\setminus N_G^{i-1}(W)$. Denote by $B_G^r(W)$ the ball of radius $r$ around $W$, that is, $B_G^r(W)=\bigcup_{0\leq i\leq r}N_G^{i}(W)$. For simplicity, write $B_G^r(v)$ for $B_G^r(\{v\})$. Moreover, the subgraph of $G$ induced by $W$, denoted by $G[W]$, is the graph with vertex set $W$ and edge set $\{xy\in E(G):x,y\in W\}$; we write $G-W=G[V(G)\backslash W]$. If $F\subseteq G$ is a subgraph, then $G\setminus F$ denotes the spanning subgraph obtained by deleting the edges of $F$. For $U,W\subseteq V(G)$, set $N_U(W)=N_G(W)\cap U$, and simply write $N_U(w)=N_U(\{w\})$ and $d_U(w)=|N_U(w)|$ for $w\in V(G)$. If $U$ and $W$ are disjoint, let $e_G(U,W)$ denote the number of edges between $U$ and $W$.

For a path $P$, its length is the number of edges in $P$, denoted by $\ell(P)$ or $|P|$, and $\mathsf{Int}(P)$ denotes its set of internal vertices. For a family $\mathcal P$ of paths, write $V(\mathcal P):=\bigcup_{P\in\mathcal P}V(P)$ and $\mathsf{Int}(\mathcal P):=\bigcup_{P\in\mathcal P}\mathsf{Int}(P)$. We say that $P$ is an $x,y$-path if $x$ and $y$ are the endvertices of $P$, and $x$ and $y$ are called the \emph{starting vertex} and the \emph{terminal vertex} of $P$, respectively. For two vertex sets $A$ and $B$, we say that $P$ is a path from $A$ to $B$ if the starting vertex of $P$ lies in $A$, the terminal vertex of $P$ lies in $B$, and the internal vertices of $P$ are disjoint from $A\cup B$. We say that two paths \(P_1,P_2\) are vertex-disjoint if and only if \(V(P_1)\cap V(P_2)=\varnothing\), which means that all their vertices, including their endvertices, are distinct. If only internal vertex-disjointness is required, we state this explicitly.
For a graph $H$, a \emph{subdivision} of $H$, denoted by $TH$, is a graph obtained by replacing edges of $H$ by internally vertex-disjoint paths. The vertices corresponding to $V(H)$ are called the \emph{branch vertices} of $TH$, and the other vertices of $TH$ are its \emph{subdividing vertices}. For simplicity, we write $\mathsf{Br}(TH)$ for the set of branch vertices of $TH$ and set $\mathsf{Sub}(TH):=V(TH)\backslash \mathsf{Br}(TH)$. A subdivision of $H$ is \emph{balanced} if every edge of $H$ is subdivided the same number of times. For $\ell\in \mathbb{N}$, denote by $TH^{(\ell)}$ a balanced subdivision of $H$ in which every edge of $H$ is subdivided $\ell$ times. Furthermore, if $H$ is a clique, we call $TH^{(\ell)}$ a \emph{clique $\ell$-subdivision}. In particular, $TK_r^{(1)\ast q}$ consists of $r$ common branch vertices and, for every pair of branch vertices, $q$ internally vertex-disjoint paths of length two, with all subdividing vertices distinct. Subdivisions are a key tool for constructing cycles or paths with well-controlled lengths.

Throughout the paper, we use standard hierarchy notation: we write $a\ll b$ to mean that, given $b$, one can choose $a_0$ such that the subsequent arguments hold for all $0<a\le a_0$. For any positive integer $r$, we write $[r]$ for the set $\{1, \ldots, r\}$. We follow Bondy and Murty~\cite{Bon} for terminology and notation not defined here.

\subsection{Sublinear expander}
Koml\'{o}s and Szemer\'{e}di~\cite{KS1,KS2} introduced the notion of a sublinear expander, in which every vertex set of reasonable size expands by a sublinear factor. We shall use a robust version due to Haslegrave, Kim, and Liu~\cite{robustref}, which retains the expansion property even after a relatively small set of edges is removed.
For $\varepsilon_1>0$ and $k>0$, define the function $\rho(x)$ by
\begin{align}\label{eq:rho}
  \begin{split}
    \rho(x)=\rho(x,\varepsilon_1,k):=\left\{
      \begin{array}{ll}
        0    & \text{if} \ x<k/5, \\
        \tfrac{\varepsilon_1}{\log^2(\tfrac{15x}{k})}   & \text{if} \ x\geq k/5.
      \end{array}
      \right.
    \end{split}
  \end{align}

  \noindent
  For simplicity, we write $\rho(x)$ for $\rho(x,\varepsilon_1,k)$ when $\varepsilon_1$ and $k$ are clear from context. Note that the function $\rho(x)$ decreases for $x\ge k/5$, whereas $x\rho(x)$ increases for $x\ge k/2$.

  \begin{definition}[\cite{robustref}]
    Let $\varepsilon_1>0$ and $k>0$. A graph $G$ is an \emph{$(\varepsilon_1,k)$-robust-expander} if, for every $X\subseteq V(G)$ with $k/2\leq |X|\leq |V(G)|/2$ and every subgraph $F\subseteq G$ with $e(F)\leq d(G)\cdot\rho(|X|)\cdot|X|$, we have
 $|N_{G\setminus F}(X)|\geq \rho(|X|)\cdot|X|$.
  \end{definition}

Throughout the paper, we refer to a robust expander simply as an expander. In arguments that require only vertex expansion, we take $F=\varnothing$.
 %\footnote{Only  Lemma \ref{lem:lar-unit} use robust version in the whole paper.}

  \begin{lemma}[\cite{robustref}]\label{lem:sublinear}
    Let $C>30$, $\eps_1\leq \tfrac{1}{10C}$, $c'<\tfrac{1}{2}$, $d>0$ and $\rho(x)=\rho(x,\eps_1,c'd)$. Then every graph $G$ with $d(G)=d$ has a subgraph $H$ such that
    \begin{itemize}
      \item[$(1)$] $H$ is an $(\eps_1,c'd)$-expander,
      \item[$(2)$] $d(H)\geq (1-\eps_0)d$, where $\eps_0:=\tfrac{C\eps_1}{\log 3}<1$,
      \item[$(3)$] $\delta(H)\geq \tfrac{d(H)}{2}$, and
      \item[$(4)$] $H$ is $\nu d$-connected, where $\nu:=\tfrac{\eps_1}{6\log^2(5/c')}$.
    \end{itemize}
  \end{lemma}

\begin{lemma}\label{lem:delete-few-vertices-expander}
Let $G$ be an $(\eps,k)$-expander, and $S\subseteq V(G)$ with $|S|\le |G|/2$. If $\rho(k/2,\eps,k)\tfrac{k}{2}\ge 2|S|$, then $G-S$ is an $(\eps/2,k)$-expander.
\end{lemma}

\begin{proof}
  Set $H:=G-S$. Let $X\subseteq V(H)$ with $\tfrac{k}{2}\le |X|\le\tfrac{|H|}{2}$, and let $F\subseteq H$ with $e(F)\le d(H)\rho(|X|,\eps/2,k)|X|$. Since $|H|\ge |G|/2$, we have $d(H)\le2d(G)$. Moreover, $\rho(x,\eps/2,k)=\tfrac12\rho(x,\eps,k)$.
  Consequently, $e(F)\le d(G)\rho(|X|,\eps,k)|X|$. As $G$ is an
  $(\eps,k)$-expander, $|N_{G\setminus F}(X)|
    \ge \rho(|X|,\eps,k)|X|$. Since $x\rho(x,\eps,k)$ is increasing for $x\ge k/2$, it follows that
  \[
    \rho(|X|,\eps,k)|X|
    \ge \rho(k/2,\eps,k)\tfrac{k}{2}
    \ge2|S|.
  \]
  Therefore,
    $|N_{H\setminus F}(X)|
    \ge |N_{G\setminus F}(X)|-|S|
    \ge \tfrac12\rho(|X|,\eps,k)|X|
    =\rho(|X|,\eps/2,k)|X|$.
\end{proof}

  A key property of expanders is that two vertex sets can be connected by a short path while avoiding a moderately sized set of vertices.

  \begin{lemma}[Small-diameter lemma~\cite{KS2}]\label{distance}
    Let $\varepsilon_1, k>0$. If $G$ is an $n$-vertex $(\varepsilon_1,k)$-expander, then for any two vertex sets $X_1, X_2$ each of size at least $x\geq k/2$ and any vertex set $W$ of size at most $\rho(x)x/4$, there exists a path in $G-W$ between $X_1$ and $X_2$ of length at most $(2/\varepsilon_1)\log^3\left(15n/k\right)$.
  \end{lemma}

  The following lemma implies that any non-bipartite sublinear expander satisfying a suitable minimum-degree condition contains a short odd cycle.
  \begin{lemma}\label{lem:oddcyclelengthinexpander}
    Let $0<\tfrac{1}{d}\ll \eps_1,\eps_2<1$, with, in addition, $\eps_2<\tfrac{1}{20}$. If $G$ is an $n$-vertex $(\eps_1,\eps_2d)$-expander with $\delta(G)\geq d/10$ and $G$ contains an odd cycle, then there is an odd cycle of length at most $2m+5$ in $G$, where $m=\tfrac{2}{\eps_1}\log^3(\tfrac{15n}{\eps_2 d})$.
  \end{lemma}
  \begin{proof}
    Let $\Gamma$ be a shortest odd cycle and write $|\Gamma|=2\ell+1$. Choose antipodal vertices $u,v\in V(\Gamma)$. Since $\delta(G)\ge d/10$, there are disjoint sets $N'(u)\subseteq N(u)$ and $N'(v)\subseteq N(v)$, each of size $d/20\ge\eps_2d$. Lemma~\ref{distance} gives a path of length at most $m$ between $N'(u)$ and $N'(v)$. Extending it by one edge at each end and removing loops if necessary gives a $u,v$-path $P$ of length at most $m+2$. One of the two $u,v$-arcs $Q$ of $\Gamma$ has parity opposite to $P$ and length at most $\ell+1$. Thus $P\cup Q$ contains an odd cycle, and the minimality of $\Gamma$ gives
 $      |\Gamma|\le |P|+|Q|
      \le m+2+\tfrac{|\Gamma|+1}{2}$.
    Hence $|\Gamma|\le2m+5$.
  \end{proof}
  \begin{proposition}\label{fact:bipartition}
    Every graph $G$ has a partition $V(G)=V_1\mathbin{\dot\cup}V_2$
    such that $d_{V_{3-i}}(v)\ge d_G(v)/2$ for every $i\in[2]$ and every $v\in V_i$.
  \end{proposition}
 % \begin{proof}
%    Choose a partition \(V(G)=V_1\mathbin{\dot\cup}V_2\) for which the number of edges between \(V_1\) and \(V_2\) is maximum. We claim that this partition has the desired property. Indeed, fix \(i\in[2]\) and \(v\in V_i\). If $d_{V_{3-i}}(v)<d_{V_i}(v)$, then moving \(v\) from \(V_i\) to \(V_{3-i}\) would increase the number of edges between the two parts by $d_{V_i}(v)-d_{V_{3-i}}(v)>0$, contradicting the maximality of the chosen partition. Hence $d_{V_{3-i}}(v)\ge d_{V_i}(v)$. Since $d_G(v)=d_{V_i}(v)+d_{V_{3-i}}(v)$, it follows that $d_{V_{3-i}}(v)\ge \tfrac{d_G(v)}{2}$. This holds for every \(i\in[2]\) and every \(v\in V_i\), as desired.
 % \end{proof}
  \begin{proof}
     Choose a bipartition maximizing the number of crossing edges. If some
     vertex had more neighbors on its own side than on the other, moving it
     would increase that number.
   \end{proof}

  \subsection{Regularity tools}
  For two disjoint vertex sets $A$ and $B$ in a graph $G$, we define the \emph{density} between $A$ and $B$ to be $d_G(A,B):=\tfrac{e_G(A,B)}{|A||B|}$.

  \begin{definition}[$\varepsilon$-regular pair]
    Let $G$ be a graph and let $X,Y\subseteq V(G)$ be disjoint. We call $(X,Y)$ an \emph{$\varepsilon$-regular pair} (in $G$) if for all $A\subseteq X$ and $B\subseteq Y$ with $|A|\geq \varepsilon|X|$ and $|B|\geq \varepsilon|Y|$, one has
    \begin{equation*}
      |d_G(A,B)-d_G(X,Y)|\leq \varepsilon.
    \end{equation*}
    We say that $(X,Y)$ is \emph{$(\varepsilon,\tau^+)$-regular} if it is $\varepsilon$-regular and $d_G(X,Y)\geq\tau$.
  \end{definition}

  \begin{lemma}[Degree form of Szemer\'{e}di's regularity lemma~\cite{komlos1996szemeredi}]\label{regularity}
    Let $0<1/n\ll1/r_1\ll\xi\ll\tau\ll\beta<1$, and let $G$ be an $n$-vertex
    graph with $d(G)\ge\beta n$. Then there exist a partition
    $V(G)=V_0\mathbin{\dot\cup}V_1\mathbin{\dot\cup}\cdots\mathbin{\dot\cup}V_r$,
    a spanning subgraph $G^\circ\subseteq G$, and a graph $R$ on $[r]$ such that the following holds.
    \begin{enumerate}[label=\rm(\arabic*)]
      \item $1/\xi\le r\le r_1$, $|V_0|\le\xi n$, and $|V_1|=\cdots=|V_r|=:m$;
      \item $d_{G^\circ}(v)\ge d_G(v)-(\tau+2\xi)n$ for every $v\in V(G)$;
      \item $e(G^\circ[V_i])=0$ for every $i\in[r]$;
      \item for $i\ne j$, either $e(G^\circ[V_i,V_j])=0$, or $(V_i,V_j)$ is an $(\xi,\tau^+)$-regular pair in $G^\circ$; and
      \item $ij\in E(R)$ precisely when $e(G^\circ[V_i,V_j])\neq 0$, and $d(R)\ge(\beta-3\tau)r$.
    \end{enumerate}
  \end{lemma}
  The spanning subgraph $G^\circ$ is called the \emph{pure graph}, and $R$ is the \emph{reduced graph} of $G$. In particular, $V_0$ is called the exceptional cluster. Throughout what follows, all regular pairs and all neighborhoods between non-exceptional clusters are taken in $G^\circ$; hence the resulting paths also lie in $G$, since $G^\circ\subseteq G$. Next, we introduce the embedding and slicing lemmas, which are essential for constructing cycles of the desired lengths.

  \begin{lemma}[Embedding lemma,~\cite{komlos1996szemeredi}]\label{embedding}
    Suppose that $0<1/n\ll\eps\ll\tau,1/\Delta$ and $1/n\ll1/r$, where
    $r,n\in\mathbb N$. Let $H$ and $G$ have vertex partitions
    $X_1\dot\cup\cdots\dot\cup X_r$ and $V_1\dot\cup\cdots\dot\cup V_r$, respectively, and let $R$ be a
    graph on $[r]$. Suppose that $\Delta(H),\Delta(R)\le\Delta$ and that:
    \begin{enumerate}[label=\rm(\arabic*)]
      \item for each $i\in [r]$, $X_i$ is an independent set in $H$ and $|X_i|\leq (1-\eps^{1/2})|V_i|$;
      \item for each $ij\in E(R)$, the graph $G[V_i,V_j]$ is $(\eps,\tau^+)$-regular; and
      \item for each $\{i,j\}\in \binom{[r]}{2}$ with $ij\notin E(R)$, the graph $H$ contains no edges from $X_i$ to $X_j$.
    \end{enumerate}
    Then there exists an embedding $\phi:V(H)\to V(G)$ such that $\phi(X_i)\subseteq V_i$ for every $i\in [r]$.
  \end{lemma}

  \begin{lemma}[Slicing lemma,~\cite{komlos1996szemeredi}]\label{lem:slicing}
    Let $(A,B)$ be $\varepsilon$-regular. Set $d:=d_G(A,B)$. If
    $\alpha>\varepsilon$, $A'\subseteq A$,
    $B'\subseteq B$, $|A'|\ge\alpha|A|$, and $|B'|\ge\alpha|B|$, then
    $(A',B')$ is $\varepsilon'$-regular for
    $\varepsilon'=\max\{\varepsilon/\alpha,2\varepsilon\}$. Moreover,
    $|d_G(A',B')-d|\le\varepsilon$.
  \end{lemma}
  \begin{lemma}[\cite{Erdos-cycle}]\label{lem:Erdos-cycle}
    Every graph with average degree at least $r\ge2$ contains a cycle of length at least $r$.
  \end{lemma}
  \subsection{Properties of minimal counterexamples to Theorem~\ref{thm:main}}\label{minicoeg}
  Let $d:=2t+2$ be an even integer. In what follows, let $G$ be a counterexample to Theorem~\ref{thm:main} with the minimum possible number $n$ of vertices; that is, $G$ satisfies
  \stepcounter{propcounter}
    \begin{enumerate}[label = ({\bfseries \Alph{propcounter}\arabic{enumi}})]
        \rm
      \item\label{couneg1} $e(G)\ge \tfrac{(n-1)(d-1)}{2}$,
      \rm
    \item\label{couneg2} $G$ is not a graph all of whose blocks are isomorphic to $K_{d-1}$, and
    \rm
    \item\label{couneg3} $G$ contains no $d/2$ consecutive even cycle lengths, where we regard $K_2$
      as a $2$-cycle.
  \end{enumerate}
  We first record two auxiliary results and some elementary properties of such a minimal counterexample.
\begin{lemma}[\cite{Chvatal1972}]\label{chvatal}
Let $G$ be an $n$-vertex graph with degree sequence
$d_1\le d_2\le\cdots\le d_n$. If, for every integer $i<n/2$,
\(
  d_i\le i\) implies \(d_{n-i}\ge n-i,
\)
then $G$ is Hamiltonian.
\end{lemma}

\begin{lemma}[\cite{Bondy1971}]\label{bondy}
Let $G$ be a Hamiltonian graph on $n$ vertices with
$e(G)\ge n^2/4$. Then $G$ is pancyclic unless $n$ is even and
$G\cong K_{n/2,n/2}$.
\end{lemma}

  \begin{proposition}\label{prop:minicoeg}
    Let $G$ be a minimal counterexample with $n$ vertices. Then $G$ satisfies~\ref{couneg1}--\ref{couneg3}, is $2$-connected, has $\delta(G)\ge d/2$, and satisfies $n\ge d+2$.
  \end{proposition}

  \begin{proof}
    Observe that $G$ is connected by the minimality of $G$. 
    Let $B_1,\ldots,B_s$ be the blocks of $G$. Note that
    \[
      e(G)=\sum_{i=1}^s e(B_i)
      \qquad\text{and}\qquad
      n-1=\sum_{i=1}^s\bigl(|B_i|-1\bigr).
    \]
    Suppose that $s\ge2$. If
    $e(B_i)>\tfrac{(|B_i|-1)(d-1)}{2}$ for some $i$, then $B_i$ is not isomorphic
    to $K_{d-1}$ and is therefore a smaller counterexample, contradicting the minimality of $G$.
    Consequently, $e(B_i)=\tfrac{(|B_i|-1)(d-1)}{2}$ for every $i\in[s]$. If some $B_j$ is not isomorphic to $K_{d-1}$, then $B_j$ is a smaller
    counterexample, a contradiction. Hence every block of $G$ is isomorphic to $K_{d-1}$, contradicting~\ref{couneg2}. It follows that $s=1$,
    and hence $G$ is $2$-connected.

    We next prove $\delta(G)\geq d/2$. Suppose that there is some vertex $v\in V(G)$ with $d_G(v)< d/2$. Then
    \(
        e(G-v)
        =e(G)-d_G(v)
        >\tfrac{(n-2)(d-1)}{2}.
    \)
    Note that the strict inequality shows that not every block of $G-v$ can be
    isomorphic to $K_{d-1}$. Thus, $G-v$ is a smaller counterexample, contradicting the minimality of $G$.

We finally prove that $n\ge d+2$. If $n\le d-2$, then
\(
  e(G)\ge \tfrac{(n-1)(d-1)}2>\binom n2,
\)
a contradiction. If $n=d-1$, equality must hold throughout, so $G=K_{d-1}$, contrary to~\ref{couneg2}. It remains to rule out $n\in\{d,d+1\}$.
Let $d_1\le\cdots\le d_n$ be the degree sequence of $G$. If $G$ were not Hamiltonian, then Lemma~\ref{chvatal} would give some $k<n/2$ such that $d_k\le k$ and $d_{n-k}\le n-k-1$. Since $G$ is $2$-connected, $k\ge2$, and hence
\[
  2e(G)\le k^2+(n-2k)(n-k-1)+k(n-1)<(n-1)(d-1),
\]
where the strict inequality follows from $n\in\{d,d+1\}$ and sufficiently large $d$. This contradicts~\ref{couneg1}, so $G$ is Hamiltonian. Moreover,
\(
  e(G)\ge \tfrac{(n-1)(d-1)}2>\tfrac{n^2}{4}
\)
for both $n=d$ and $n=d+1$. Lemma~\ref{bondy} therefore implies that $G$ is pancyclic. Thus $G$ contains a cycle of every even length from $4$ to $d$; together with any edge of $G$, regarded as a $2$-cycle, this gives the $d/2$ consecutive even cycle lengths $2,4,\ldots,d$, a contradiction.
  \end{proof}

  \section{Main lemmas and overviews}\label{subsec:proof-overview}
  %\subsection{Proof of Theorem \ref{thm:main}}
  Recall that $d:=2t+2$. It suffices to find $d/2$ consecutive even cycle lengths. In particular, we just need to find $d/2-1=t$ cycles if the consecutive even cycle lengths start with $2$. Let $G$ be an $n$-vertex minimal counterexample to Theorem~\ref{thm:main}. By Proposition~\ref{prop:minicoeg}, $G$ is
  $2$-connected and $\delta(G)\ge d/2$.
  %Unless $G=K_{d-1}$, we show that $e(G)\ge \tfrac{(d-1)(|G|-1)}2$ already forces the required interval.

  Induction on proper induced subgraphs gives a hereditary density bound. Lemma~\ref{lem:sublinear} then provides an
  $h$-vertex expander $H\subseteq G$ such that
$    d(H)\ge(1-\eps)d$,
$    \delta(H)\ge(1-\eps)d/2$,
  and $H$ has connectivity linear in $d$. We divide the proof according to the relative sizes of $d$ and $h$.

  \begin{lemma}[Dense case]\label{lem:dense-expander}
    Let
    \(
      0<1/n,1/h,1/d\ll1/K\ll\nu\ll\eps\ll1/5
    \)
    with $d\in2\mathbb N$ and $h/K\leq d\leq n$. Let $G$ be a minimal counterexample with $n$ vertices. Suppose that  $G$ contains an $h$-vertex subgraph $H$ with
    \(
      \delta(H)\ge(1-\eps)d/2,
      d(H)\ge(1-\eps)d
    \)
    and $H$ is $\nu d$-connected. 
    Then $G$ contains $d/2$ consecutive even cycle lengths, where $K_2$ is
    regarded as a cycle of length $2$.
  \end{lemma}

  \begin{lemma}[Intermediate case]\label{lem:medium-expander}
    Let
    \(
      0<1/n,1/h,1/d\ll1/K\ll\nu \ll \eps_1,\eps_2\ll\eps\ll 1/5
    \)
    with $d\in2\mathbb N$, $\tfrac13\log^{1000}h\le d\le h/K$. Let $G$ be a minimal counterexample with $n$ vertices. Suppose that  $G$ contains an $h$-vertex $(\eps_1,\eps_2d)$-expander $H$ with
    \(
      \delta(H)\ge(1-\eps)d/2,
      d(H)\ge(1-\eps)d,
    \)
    and $H$ is $\nu d$-connected. 
    Then $G$ contains $d/2$ consecutive even cycle lengths, where $K_2$ is
    regarded as a cycle of length $2$.
  \end{lemma}

  \begin{lemma}[Sparse case]\label{lem:den-exp-reduced}
    Let
    \(
      0<1/n,1/h,1/d\ll\nu \ll
     \eps_1, \eps_2\ll \eps\ll 1/5
    \)
    with $d\in2\mathbb N$ and $d\le \tfrac{1}{2}\log^{1000}h$. Let $G$ be a minimal counterexample with $n$ vertices. Suppose that  $G$ contains an $h$-vertex $(\eps_1,\eps_2d)$-expander $H$ with
    \(
      \delta(H)\ge(1-\eps)d/2,
      d(H)\ge(1-\eps)d,
    \)
    and $H$ is $\nu d$-connected. 
    Then $G$ contains $d/2$ consecutive even cycle lengths, where $K_2$ is
    regarded as a cycle of length $2$.
  \end{lemma}

  We now prove Theorem~\ref{thm:main} using the preceding three lemmas.
  \begin{proof}[Proof of Theorem~\ref{thm:main}]
    Let $d=2t+2$. Choose constants such that $0<1/n,1/d,1/h\ll 1/K\ll \eps_1,\eps_2,\eps\ll 1/5$. As $e(G)\geq \tfrac{(n-1)(d-1)}{2}$, we have 
        $d(G)\geq \tfrac{(n-1)(d-1)}{n}\geq (1-\eps/3)d$.
    Let $c':=\tfrac{\eps_2d}{d(G)}<1/2$. Choose $C_0>30$ and $\eps_1\leq \tfrac{1}{10C_0}$ such that $\eps_0=\tfrac{C_0\eps_1}{\log3}\leq \eps/3$. Applying Lemma~\ref{lem:sublinear} to $G$, we obtain a subgraph $H$ that is an \(h\)-vertex \((\eps_1,c'd(G)=\eps_2d)\)-expander with \(d(H)\geq (1-\eps_0)d(G)\geq (1-\eps)d\) and \(\delta(H)\geq (1-\eps)d/2\). Furthermore, $H$ is $\nu d$-connected for some constant $\nu=\nu(\eps_1,\eps_2,\eps)$.
    Theorem~\ref{thm:main} now follows from the dense, intermediate, and sparse cases---Lemmas~\ref{lem:dense-expander},~\ref{lem:medium-expander}, and~\ref{lem:den-exp-reduced}, respectively---in the ranges
    \[
      d\ge \tfrac{h}{K},\qquad
      \tfrac{1}{3}\log^{1000}h\le d\le \tfrac{h}{K},\qquad\text{and}\qquad
      d\le \tfrac{1}{2}\log^{1000}h. \qedhere
    \]
  \end{proof}

\noindent
\textbf{Proof overview of Lemma~\ref{lem:dense-expander}.}
The proof is based on a stability argument. Recall that the main extremal example is a block tree whose blocks are copies of $K_{d-1}$. However, when we extract a sublinear expander, we may lose a small part of both the average degree and the minimum degree. Thus, even when the original graph is close to the extremal example, the resulting expander may only look like a clique on $(1-o(1))d$ vertices. Therefore, the expander alone might not accommodate all cycle lengths that we look for, and we need to make use of the vertices outside of the expander. Another difficulty arises when the dense part is close to a complete bipartite graph with one part of size $d/2$. The difference in the number of edges between these two types of examples is only about $n/2$, which is negligible compared to the error term $o(dn)$. We shall use a structural analysis to exclude this bipartite near-extremal configuration.

We first apply the regularity lemma to the dense expander $H$. The reduced graph contains a long cycle, and the clusters on this cycle give a large vertex set $\mathcal F\subseteq V(H)$. Roughly speaking, for most suitable pairs $x,y\in\mathcal F$, there is an integer $s$ such that $\mathcal F$ contains $x,y$-paths of lengths
\(
s,s+2,\ldots,s+(1-o(1))d.
\)
We also obtain a partition of the vertices which controls parity: vertices in the same part are joined by paths of even length, while vertices in different parts are joined by paths of odd length. Moreover, $H$ contains a cycle of every even length from $4$ to about $(1-2\eps)d$. We choose an even cycle $C_0$ of length $(1-2\eps)d$.

Starting from $C_0$, 
we repeatedly apply a local augmentation that replaces several short segments of the current cycle by paths through vertices outside the cycle. Each augmentation produces a new even cycle whose length is exactly two more than that of the preceding one. We continue until either the cycle reaches length $d$ or no further augmentation is possible. In the former case, the cycles inside $H$ of lengths
$2,4,\ldots,|C_0|$
together with the successive augmented cycles of lengths
$|C_0|+2,|C_0|+4,\ldots,d$
give a cycle of every even length up to $d$, completing the proof. 
Hence, we may assume that the process stops at an even cycle $C$ with $|C|\le d-2$.
At the end of the process, let $A$ be the set of vertices outside $C$ having at least $(1/3+20\eps)d$ neighbors on $C$, and let
\(
W=V(G)\setminus(A\cup V(C)).
\)
The fact that the cycle cannot be extended further implies that every vertex of $W$ has fewer than $20\eps d$ neighbors in $A$. Indeed, otherwise we can choose two neighbors $x,z\in A$ whose neighborhoods on $C$ contain two vertices at distance two, and then replace a two-edge segment of $C$ by a path of the form
\(
u_i x y z u_{i+2}.
\)
This would increase the cycle length by two.

It follows that every vertex of $W$ has fewer than $(1/3+40\eps)d$ neighbors in $A\cup V(C)$. Since $\delta(G)\ge d/2$, we obtain
\(
\delta(G[W])\ge d/7.
\)
An edge count also shows that $W$ is non-empty, since otherwise the graph induced by $A\cup V(C)$ would not contain enough edges to reach the extremal bound.
To explain the main idea, for simplicity, we suppose that $G[W]$ is $2$-connected. A result of Chiba, Ota, and Yamashita gives about $d/14$ paths between two suitable vertices of $W$ whose lengths form an arithmetic progression with difference two. On the other hand, the set $\mathcal F$ gives about $(1-o(1))d/2$ possible path lengths. Joining these two path families would give more than $d/2$ consecutive cycle lengths. However, the main problem is parity: the cycles obtained in this way may all be odd.

There are two easier cases. The first is when $W\cap V(H)\ne\varnothing$. In this case, the graph induced by $W\cap V(H)$ inside $H$ still has minimum degree at least $d/7$. Since $H$ is highly connected, we can find many disjoint paths from this graph to $\mathcal F$. By choosing two paths with suitable parities and suitable endpoints in $\mathcal F$, we can combine the path families in the two parts to obtain more than $d/2$ consecutive even cycle lengths. Thus, in a counterexample, we must have
\(
W\cap V(H)=\varnothing.
\)
The second easier case is when the reduced graph on the clusters forming $\mathcal F$ is non-bipartite. In this case, the paths inside $\mathcal F$ can be chosen with either parity. We may therefore choose the parity which makes the resulting cycles even, and again the path families in $\mathcal F$ and $W$ together give more than $d/2$ consecutive even cycle lengths.

It remains to consider the case in which $W\cap V(H)=\varnothing$ and the reduced graph supporting $\mathcal F$ is bipartite. Denote
\(
U=A\cup V(C).
\)
Then $H\subseteq G[U]$. The bipartite structure gives a loss of order $d^2$ in the number of edges of $G[U]$; more precisely,
\(
e(G[U])\le (d-2)/2\cdot(|U|-(d-2)/2).
\)
This loss forces a matching of size $\Omega(d)$ between $U$ and $W$. By the pigeonhole principle, a linear number of the endpoints of this matching in $U$ lie in the same parity class. Any two such endpoints can then be joined inside $U$ by paths of about $(1-o(1))d/2$ consecutive even lengths.

The other main ingredient is a rooted cycle lemma; see Lemma~\ref{lem:rooted-even-one-exception}. Roughly speaking, it says that if a $2$-connected graph has minimum degree at least $p$, apart from at most one exceptional vertex, then a fixed root lies on about $p/2$ consecutive even cycles, all containing the same two edges incident with the root. We add an auxiliary root joined to the matching endpoints in $W$. The rooted cycle lemma gives many consecutive even cycles through two fixed edges at this new root. We then replace these two edges by the corresponding matching edges and join their endpoints in $U$ using the path family described above. This gives at least $d/2$ consecutive even cycle lengths.

\vspace{0.25cm}
\noindent
\textbf{Proof overview of Lemma~\ref{lem:medium-expander}.} 
Let $m$ be the smallest even integer larger than $\log^4(n/d)$, and let $L_G$ be the set of vertices of degree at least $dm^{31}$. The two cases of the proof follow the same scheme. We first construct a family of connecting gadgets: if $|L_G|$ is large, these are vertices of $L_G$ together with large stars centred at them; otherwise, robust expansion provides many pairwise internally vertex-disjoint webs, whose cores and
large exteriors play the same role.

We then reserve a small adjuster part. In the bipartite case this consists of one adjuster, while in the non-bipartite case we also use an odd simple adjuster to control parity. Outside this part, we construct long paths with endpoints at the connecting gadgets by linking them one by one using the connecting property ensured by large exteriors. In the large-degree case, we find a path of length in $[Cd,5Cdm]$ on which consecutive vertices of $L_G$ are at distance at most $49Cm$ for some absolute constant $C$. In the web case, we find a path of length in $[200d, Cdm^{30}]$ on which consecutive core vertices are at distance at most $200m$.
Every subpath between two consecutive connecting gadgets can be combined with the adjuster part to form a cycle.
Connecting these paths to the adjuster part gives
overlapping ranges of cycle lengths, since consecutive target lengths
differ by $O(Cm)$, whereas the adjuster varies the length over a range
of order $m^2$. Their union therefore contains the required $Cd/2$
consecutive even cycle lengths.

\vspace{0.25cm}
\noindent
\textbf{Proof overview of Lemma~\ref{lem:den-exp-reduced}.}
We first consider the case in which the sublinear expander contains no large clique $1$-subdivision. Our argument is inspired by the adjuster method of Liu and Montgomery. When the expander $H$ is bipartite, we reserve a short adjuster and then apply the robust adjuster lemma (see Lemma~\ref{lem:spa-max-deg}) to obtain paths whose lengths cover a long interval. Closing these paths through the reserved adjuster gives the required cycle lengths and parity. In the non-bipartite case, an odd adjuster removes the parity obstruction and yields every length in the interval; the bipartite case is similar and gives every even length.

Next, we consider the case in which the sublinear expander contains a large clique $1$-subdivision. Indeed, such a subdivision already yields a long initial interval of attainable lengths. Lemma~\ref{lem:new} shows that either this interval can be extended to the required range, or, after deleting the branch vertices of the large clique $1$-subdivision, the remaining graph contains a subgraph with average degree $\Omega(d)$. In the latter case, we extract from this subgraph a further sublinear expander $H'$ as dense as the remaining graph and apply to $H'$ the appropriate argument from the dense, intermediate, or sparse cases. If $H'$ also contains a large clique $1$-subdivision, then two disjoint links between these subdivisions allow their respective path-length ranges to be combined, producing an interval whose total span exceeds $d$.

\section{Dense expander case}\label{sec:dense-expander}

  A \emph{rooted graph} $(G,x,y)$ consists of a graph $G$ together with two distinct specified vertices $x,y\in V(G)$. We call $(G,x,y)$ \emph{$2$-connected} if $G+xy$ is $2$-connected, and write
  \[
    \delta(G,x,y):=\min\{d_G(w):w\in V(G)\setminus\{x,y\}\}.
  \]
  A family of paths is \emph{admissible} if their lengths form an arithmetic progression with common difference one or two.
  If $B$ is an end-block with cut-vertex $c$, we call $V(B)\setminus\{c\}$ the interior of~$B$. We will use the following auxiliary results.

  \begin{lemma}[\cite{ChibaOtaYamashita2023}]\label{lem:exceptional-admissible-path}
    Let $k$ be a positive integer, let $(G,x,y)$ be a $2$-connected rooted graph, and let $z$ be a vertex, possibly outside $V(G)$. Suppose that $V(G)\setminus\{x,y,z\}\ne\varnothing$ and that $d_G(w)\ge k+1$ for every $w\in V(G)\setminus\{x,y,z\}$. Then $G$ contains $k$ admissible $x,y$-paths.
  \end{lemma}

We also use the following consequence of Mader's theorem.

  \begin{lemma}[\cite{Mader1972}]\label{lem:Mader-connected}
    Every graph of average degree at least $4k$ contains a $k$-connected subgraph.
  \end{lemma}

  For a partition $V(G)=X_0\mathbin{\dot\cup}X_1$ and vertices $u,v\in V(G)$, define
  \[
    \pi(u,v):=
    \begin{cases}
      0,&\text{if $u$ and $v$ lie in the same part of the partition},\\
      1,&\text{otherwise}.
    \end{cases}
  \]

  The following lemma gives almost $d/2$ consecutive path lengths. The proof idea is simple: we construct a walk along a long cycle in the reduced graph.
  
  \begin{lemma}[Paths of consecutive lengths in the dense case]\label{lem:verydense-path}
    Suppose that $K,r_1\in\mathbb N$ and
    \(
      0<1/n\ll1/r_1\ll\xi\ll\tau\ll \alpha,1/K,\nu\le 1.
    \)
    Let $G$ be an $n$-vertex $\nu d$-connected graph with $d\in2\mathbb N$ and $d>n/K$, and suppose that $d(G)\ge d$ and $\delta(G)\ge d/2$. Then there exist a partition $V(G)=X_0\mathbin{\dot\cup}X_1$ and a constant $M=M(K,\nu,r_1)\ge 100$ such that the following holds. For every pair of distinct vertices $u,v\in V(G)$ and every set $F\subseteq V(G)\setminus\{u,v\}$ with $|F|\le M$, there is an integer $t=t(u,v,F)$ satisfying $1\le t\le M/10$ and $t\equiv\pi(u,v)\pmod2$ such that, for every integer
    \(
      0\le i\le\left\lfloor (1-2\alpha)d/2\right\rfloor,
    \)
    the graph $G-F$ contains a $u,v$-path of length $t+2i$.
  \end{lemma}

\begin{proof}
    By Lemma~\ref{regularity} with $\beta=d/n$, we obtain a partition
 $V(G)=V_0\mathbin{\dot\cup}V_1\mathbin{\dot\cup}\cdots\mathbin{\dot\cup}V_r$,
    a pure subgraph $G^\circ\subseteq G$, and a reduced graph $R$ on $[r]$. The
    clusters $V_1,\ldots,V_r$ have a common size $m$, $|V_0|\le\xi n$,
    $r\le r_1$, and $d(R)\ge(d/n-3\tau)r$. By
    Lemma~\ref{lem:Erdos-cycle}, $R$ contains a cycle $\Gamma$ of length
    $\ell\ge(d/n-3\tau)r$. Relabel the corresponding clusters as
    $V_1,\ldots,V_\ell$, with indices taken modulo $\ell$. For
    $j\in[\ell]$, call a vertex $w\in V_j$ \emph{good} if
    $|N_{G^\circ}(w)\cap V_{j+1}|\ge\sqrt{\xi}\,m$. Since
    $G^\circ[V_j,V_{j+1}]$ is an $(\xi,\tau^+)$-regular pair, all but at
    most $\xi m$ vertices of $V_j$ are good. Thus, if $B$ is the set of
    good vertices in $V_1\cup\cdots\cup V_\ell$, then
    $|B|\ge(1-2\xi)\ell m\ge(1-\alpha/4)d\ge\nu d/2+1$.

    Let $\zeta=20K/\nu$. For each $z\in V(G)$, Menger's theorem gives at
    least $\nu d/2$ paths from $z$ to distinct vertices of
    $B\setminus\{z\}$ which intersect only at $z$. Since these paths are
    internally vertex-disjoint, at most $n/(\zeta-1)\le\nu d/10$ of them
    have length greater than $\zeta$. Discard the paths of length greater than $\zeta$, and let $\mathcal{P}_z$ be the collection of the remaining paths. Label each path of $\mathcal{P}_z$ ending in
    $V_j$ by its length plus $j$ modulo $2$. By the pigeonhole principle,
    we obtain a family $\mathcal Q(z)$ of at least $\nu d/8$ paths, all
    of length at most $\zeta$, with the same label. Denote this label by
    $f(z)\in\{0,1\}$, and let $X_0:=f^{-1}(0)$ and
    $X_1:=f^{-1}(1)$. Then $\pi(u,v)\equiv f(u)+f(v)\pmod 2$ for all
    $u,v\in V(G)$. 
    
    Let
    $M=10(2\zeta+3r_1+10)$.
    Fix distinct vertices $u,v\in V(G)$ and
    $F\subseteq V(G)\setminus\{u,v\}$ with $|F|\le M$. We may choose
    $Q_u\in\mathcal Q(u)$ avoiding $F\cup\{v\}$, and then choose
    $Q_v\in\mathcal Q(v)$ avoiding $F\cup V(Q_u)$. Let
    $u'\in V_x$ and $v'\in V_y$ be their endvertices, and let
    $a:=|Q_u|$ and $b:=|Q_v|$. By the definition of the labels,
    $a+x\equiv f(u)$ and $b+y\equiv f(v)\pmod 2$, and hence
    $$
      a+b\equiv y-x+\pi(u,v)\pmod 2.
    $$
    Since $u'$ and $v'$ are good, we
    may choose disjoint sets
    $W^+\subseteq N_{G^\circ}(u')\cap V_{x+1}$ and
    $W^-\subseteq N_{G^\circ}(v')\cap V_{y+1}$, each of size
    $\sqrt{\xi}\,m/10$, avoiding $F\cup V(Q_u)\cup V(Q_v)$. For each
    $k\in[\ell]$, let
    $\widetilde V_k:=V_k\setminus
    (W^+\cup W^-\cup F\cup V(Q_u)\cup V(Q_v))$. Then
    $|\widetilde V_k|\ge(1-3\sqrt{\xi}/10)m$. By
    Lemma~\ref{lem:slicing}, with $\xi'=20\sqrt{\xi}$, the pairs
    $G^\circ[\widetilde V_k,\widetilde V_{k+1}]$,
    $G^\circ[W^+,\widetilde V_{x+2}]$ and
    $G^\circ[\widetilde V_y,W^-]$ are
    $(\xi',\tau/2^+)$-regular.

    Let $R^*$ consist of the cycle
    $\widetilde v_1\ldots\widetilde v_\ell\widetilde v_1$, together with
    two vertices $w^+,w^-$ and the edges
    $w^+\widetilde v_{x+2}$ and $\widetilde v_yw^-$. Choose a
    $\widetilde v_{x+2},\widetilde v_y$-path $Q$ along $\Gamma$ of length at
    most $\ell$ and parity $y-x$; one of the two arcs of $\Gamma$ has the
    required parity. Starting at $\widetilde v_{x+2}$, traverse $\Gamma$
    once if $\ell$ is even and twice if $\ell$ is odd, and then follow $Q$
    to $\widetilde v_y$. Together with the
    two edges incident to $w^+$ and $w^-$, this gives a
    $w^+,w^-$-walk $S_0$ such that $|S_0|\le3\ell+2$,
    $|S_0|\equiv y-x\pmod 2$, and every vertex of $\Gamma$ occurs at
    least once and at most four times.
    For any integer
    $0\le i\le\lfloor(1-2\alpha)d/2\rfloor$, insert $i$ two-edge walks
    of the form
    $\widetilde v_k\widetilde v_{k+1}\widetilde v_k$ into $S_0$, as
    evenly as possible among the vertices of $\Gamma$. The resulting
    $w^+,w^-$-walk $S_i$ has length $|S_0|+2i$, and each
    $\widetilde v_k$ occurs at most $4+2\lceil i/\ell\rceil$ times.
    Since $\ell\ge(d/n-3\tau)r$ and $rm\ge(1-\xi)n$, we have
    \[
        4+2\left\lceil i/\ell\right\rceil
        \le \tfrac{(1-2\alpha)d}{\ell}+6
        \le (1-\alpha)m
        \le (1-\sqrt{\xi'})|\widetilde V_k|.
    \]
    Let $H_i$ be a path whose cluster sequence is $S_i$, with its first
    vertex assigned to $W^+$, its last vertex assigned to $W^-$, and
    all other vertices assigned to the corresponding
    $\widetilde V_k$. Lemma~\ref{embedding} gives an embedding of $H_i$ into $G^\circ$.
    Adding the two edges from the ends of the embedded path to $u'$ and
    $v'$, together with $Q_u$ and $Q_v$, gives a $u,v$-path in $G-F$ of
    length $t+2i$, where $t:=a+b+2+|S_0|$. Since $a,b\le\zeta$ and
    $|S_0|\le3r_1+2$, we have
    $1\le t\le2\zeta+3r_1+4\le M/10$. Moreover,
    $t\equiv a+b+y-x\equiv\pi(u,v)\pmod 2$. This completes the proof.
\end{proof}
  We next need a corollary about paths between two pairs of vertices in the reduced graph.
  \begin{corollary}\label{cor:verydense}
    Under the assumptions of Lemma~\ref{lem:verydense-path}, there exist a partition $V(G)=X_0\mathbin{\dot\cup}X_1$ and a constant $M=M(K, \nu, r_1)\ge 100$ such that the following holds.
    \begin{enumerate}[label=\rm(\arabic*)]
      \item Let $u_1,v_1,u_2,v_2$ be four distinct vertices of $G$.
        There is an integer $t\le M$ satisfying
        \(
          t\equiv \pi(u_1,v_1)+\pi(u_2,v_2)\pmod2
        \)
        such that, for every integer $0\le i\le \lfloor(1-3\alpha)d/2\rfloor$, there are vertex-disjoint paths $P_1,P_2$, where $P_j$ joins $u_j$ to $v_j$ for $j\in[2]$ and $|P_1|+|P_2|=t+2i$.

      \item
        Let $u,v_1,v_2$ be three distinct vertices of $G$. There is an integer $t\le M$ satisfying
        \(
          t\equiv \pi(u,v_1)+\pi(u,v_2)\pmod2
        \)
        such that, for every integer $0\le i\le \lfloor(1-3\alpha)d/2\rfloor$, there exists a $v_1,v_2$-path $P$ of length $t+2i$ avoiding $u$.

      \item
        For any $L\in\{ 4,6,\ldots,2\left\lfloor(1-2\alpha)d/2\right\rfloor\}$, the graph $G$ contains a cycle of length $L$.

    \end{enumerate}
  \end{corollary}

  \begin{proof}
    Let $V(G)=X_0\mathbin{\dot\cup}X_1$ and $M=M(K, \nu, r_1)$ be the partition and the constant given by Lemma~\ref{lem:verydense-path}.
    We first prove (1). Apply Lemma~\ref{lem:verydense-path} to $u_1,v_1$ with forbidden set $\{u_2,v_2\}$, and take the shortest path supplied by the lemma. This gives a $u_1,v_1$-path $P_1$ of length $t_1\le M/10$ with $t_1\equiv\pi(u_1,v_1)\pmod2$, avoiding $u_2$ and $v_2$. We then apply Lemma~\ref{lem:verydense-path} again to $u_2,v_2$ with forbidden set $V(P_1)$. Thus, for every $0\le i\le\lfloor(1-2\alpha)d/2\rfloor$, there is a $u_2,v_2$-path $P_2$ in $G-V(P_1)$ of length $t_2+2i$ for some $t_2\le M/10$ satisfying $t_2\equiv\pi(u_2,v_2)\pmod2$. Set $t=t_1+t_2$; then (1) holds.

    For (2), apply Lemma~\ref{lem:verydense-path} to $v_1,v_2$ with forbidden set $\{u\}$. The parity assertion follows from $\pi(v_1,v_2)\equiv\pi(u,v_1)+\pi(u,v_2)\pmod2$.

    For (3), let $\eta_0:=d/n$ and apply Lemma~\ref{regularity} directly to $G$ with $\beta=\eta_0$.  This gives a pure graph $G^\circ$, clusters $V_1,\ldots,V_r$ of common size $m$, and a reduced graph of average degree at least $(\eta_0-3\tau)r$.  By Lemma~\ref{lem:Erdos-cycle}, the reduced graph contains a cycle $\Gamma=\widetilde v_1\cdots\widetilde v_\ell\widetilde v_1$ with $\ell\ge(\eta_0-3\tau)r$. Every pair $G^\circ[V_j,V_{j+1}]$ corresponding to an edge of $\Gamma$ is $(\xi,\tau^+)$-regular.

    Let $L$ be even with $4\le L\le2\lfloor(1-2\alpha)d/2\rfloor$. We construct a closed walk $Z_L$ of length $L$ on $\Gamma$ such that each vertex of $\Gamma$ occurs at most $L/\ell+\ell+2$ times. Write $L=q\ell+s$ with $0\le s<\ell$. If $\ell$ is even, then $s$ is even, and we traverse $\Gamma$ exactly $q$ times and then traverse one fixed edge back and forth $s/2$ times. If $\ell$ is odd, let $q'=2\lfloor q/2\rfloor$. Then $0\le L-q'\ell<2\ell$ and $L-q'\ell$ is even. Traverse $\Gamma$ exactly $q'$ times and then traverse one fixed edge back and forth $(L-q'\ell)/2$ times.
    Let $C_L$ be an abstract cycle of length $L$, and assign its vertices to the clusters according to $Z_L$. For every $j\in[\ell]$,
    \[
      |V(C_L)\cap V_j|\le \tfrac L\ell+\ell+2
      \le (1-2\alpha)\tfrac{\eta_0}{(\eta_0-3\tau)(1-\xi)}m+r_1+2
      \le (1-\alpha)m
      \le (1-\sqrt\xi)|V_j|.
    \]
    Hence Lemma~\ref{embedding} embeds $C_L$ in $G$.
  \end{proof}

  We next record the path and cycle structures supplied by the regular
  pairs along a reduced cycle $\Gamma$. Only the clusters indexed by
  $V(\Gamma)$ are used; in the non-bipartite case, additional edges of
  the reduced graph on $V(\Gamma)$ are used only to switch parity.
  The proof of the following lemma is similar to the proofs of Lemma~\ref{lem:verydense-path} and Corollary~\ref{cor:verydense}.

  \begin{lemma}
    \label{lem:paths-on-reduced-cycle}
    Let $K_0,r_1\in\mathbb N$ and
    \(
      0<1/m\ll1/r_1\ll\xi\ll\tau\ll1/K_0,\alpha.
    \)
    Let $G^\circ$ be a graph with clusters $V_1,\ldots,V_r$ of common size
    $m$, where $r\le r_1$, and let $\Gamma=v_1\cdots v_\ell v_1$ be a cycle
    in the reduced graph. Suppose that
    \(
      (1-\alpha/8)d_*\le \ell m\le K_0d_*,
    \)
    and that every pair $G^\circ[V_i,V_{i+1}]$ corresponding to an edge of
    $\Gamma$ is $(\xi,\tau^+)$-regular. Then there are equal-sized sets
    $U_i\subseteq V_i$, sets $\mathcal T_i\subseteq U_i$, and a constant
    $M_0=M_0(K_0,r_1)$ with the following properties. Let
    $\mathcal F:=\bigcup_{i=1}^{\ell}U_i$ and
    $\mathcal T:=\bigcup_{i=1}^{\ell}\mathcal T_i$.

    If the graph induced by $V(\Gamma)$ in the reduced graph is bipartite,
    fix a bipartition $I_0\mathbin{\dot\cup}I_1$ of its vertex set and define
    $\chi(x)=j$ whenever $x\in\mathcal T_i$ and $v_i\in I_j$.

    \begin{enumerate}[label=\textup{(\arabic*)}]
        \item\label{lem4.6-1} $(1-\alpha)d_*\le |\mathcal F|\le(1-\alpha/2)d_*$ and
        $|\mathcal T|\ge(1-4\sqrt\xi)|\mathcal F|$.

      \item\label{lem4.6-2} Every even integer $L$ with $4\le L\le(1-\alpha)d_*$ is the
        length of a cycle in $G^\circ[\mathcal F]$.

      \item\label{lem4.6-3} Let $q\in\{1,2,3\}$, and let
        $(x_1,y_1),\ldots,(x_q,y_q)$ be ordered pairs in $\mathcal T$ whose
        endpoints are all distinct. Let
        $F\subseteq\mathcal F\setminus\bigcup_{j=1}^q\{x_j,y_j\}$ satisfy
        $|F|\le M_0/2+1$ for $q=1$, and $|F|\le M_0/2$ for $q=2,3$. Let
        $s\ge0$ be an integer with $2s\le(1-2\alpha)d_*$. Then there is an integer $t\le M_0/2$ such
        that, for every $0\le i\le s$, the graph $G^\circ[\mathcal F]-F$
        contains pairwise vertex-disjoint $x_j,y_j$-paths $P_{i,j}$,
        $j\in[q]$, with
        \(
          \sum_{j=1}^q |P_{i,j}|=t+2i.
        \)
        In the bipartite case,
        $t\equiv\sum_{j=1}^q(\chi(x_j)+\chi(y_j))\pmod2$; in the
        non-bipartite case, either parity of $t$ may be prescribed.
    \end{enumerate}
  \end{lemma}

  \begin{proof}
    Fix $M_0:=100r_1^2$, and let
    $\mu:=\lfloor(1-3\alpha/4)d_*/\ell\rfloor$. Choose
    $U_i\subseteq V_i$ with $|U_i|=\mu$. By the slicing lemma, every pair
    $G^\circ[U_i,U_{i+1}]$ is $(\sqrt\xi,(\tau/2)^+)$-regular.

    With indices taken modulo $\ell$, let $\mathcal T_i$ consist of the
    vertices of $U_i$ having at least $(\tau/3)\mu$ neighbors in each of
    $U_{i-1}$ and $U_{i+1}$. Regularity gives
    $|\mathcal T_i|\ge(1-2\sqrt\xi)\mu$. Since
    $\ell\mu=(1-3\alpha/4)d_*+O(r_1)$, property~\ref{lem4.6-1} follows. Moreover, properties~\ref{lem4.6-2} and~\ref{lem4.6-3} follow from the same closed-walk construction on the reduced cycle used in Lemma~\ref{lem:verydense-path} and Corollary~\ref{cor:verydense}.
  \end{proof}

  \subsection{A rooted cycle lemma with one exceptional vertex}
  The following lemma shows that if $G$ is a $2$-connected graph, $v\in V(G)$ has degree at least $5$, and all but at most one vertex of $V(G)\backslash\{v\}$ have high degree, then $G$ contains cycles through $v$ that share the same two incident edges and whose lengths form a long arithmetic progression with common difference $2$.

  \begin{lemma}\label{lem:rooted-even-one-exception}
    Let $p\ge8$ be an even integer, let $F$ be a $2$-connected graph, and let $v,z\in V(F)$ be distinct. Suppose that $d_F(v)\ge5$ and that $d_F(w)\ge p$ for every $w\in V(F)\setminus\{v,z\}$. Then there are distinct vertices $x,y\in N_F(v)$ and an even integer $\ell$ such that $F$ contains consecutive even cycles $D_0,D_1,\ldots,D_{(p-8)/2}$ with $vx,vy\in E(D_i)$ and $|D_i|=\ell+2i$ for every $0\le i\le(p-8)/2$.
  \end{lemma}

  \begin{proof}
    Let $H:=F-v$ and $S:=N_F(v)$. Since $F$ is $2$-connected, $H$ is
    connected, $|S|\ge5$, and
    \begin{equation}\label{eq:rooted-H-degree}
      d_H(w)\ge p-1\qquad\text{for every }w\in V(H)\setminus\{z\}.
    \end{equation}

    We use the following parity observation. Suppose that a graph $J$ contains $k$ admissible $a,b$-paths and that there are two $a,b$-paths of different parities whose internal vertices avoid
    $J$.
    Then one of these two external paths closes at least $\lceil k/2\rceil$ admissible paths into cycles of consecutive even lengths. Indeed, if the common difference is two, all admissible paths have the same parity; if it is one, take the larger parity class.

    We first isolate the non-bipartite argument that will be used twice.

    \begin{claim}\label{claim:rooted-nonbipartite-core}
      Let $J$ be a $2$-connected non-bipartite subgraph of $H$,
      $z_0\in V(J)$, and  $A\subseteq S\cap V(J)$ satisfy
      $|A\setminus\{z_0\}|\ge3$. Suppose that
      $d_J(w)\ge p-1$ for every $w\in V(J)\setminus\{z_0\}$. Then
      $F[V(J)\cup\{v\}]$ contains cycles of $(p-6)/2$ consecutive even
      lengths, all containing the same two edges incident with $v$.
    \end{claim}

    \begin{poc}
      Since $J$ is $2$-connected and non-bipartite, there is an odd path
      between two vertices of $A$.
      Choose a shortest odd path
      $Q=q_0q_1\cdots q_r$ whose endvertices
      $a:=q_0$ and $b:=q_r$ lie in $A$. Then
      $V(Q)\cap A=\{a,b\}$. Moreover, every vertex outside $Q$ has at most
      four neighbors on $Q$. Otherwise, three of five such neighbors have
      indices of the same parity, and replacing the even subpath between
      the first and the last by a two-edge path gives a shorter odd
      $A$-path.
      Choose $s\in A\setminus(V(Q)\cup\{z_0\})$, and let $D$ be the
      component of $J-V(Q)$ containing $s$. Then
      $d_D(w)\ge p-5$ for every $w\in V(D)\setminus\{z_0\}$.

      Suppose first that $D$ is $2$-connected. Since $J$ is $2$-connected,
      there are $u\in V(D)\setminus\{s\}$ and $q_j\in V(Q)$ with
      $uq_j\in E(J)$. Apply
      Lemma~\ref{lem:exceptional-admissible-path} to $(D,u,s)$, with
      exceptional vertex $z_0$ and parameter $p-6$. The two external paths
      \[
        uq_j\cup q_jQa\cup av\cup vs
        \quad\text{and}\quad
        uq_j\cup q_jQb\cup bv\cup vs
      \]
      have different parities, so the parity observation gives the claim.

      Suppose now that $D$ is not $2$-connected. A bridge end-block of $D$
      can only have $z_0$ as its leaf, so there is at most one such block.
      If $D$ has a non-bridge end-block $B$, with cut-vertex $c$, whose
      interior does not contain $s$, then the $2$-connectivity of $J$ gives
      an edge $uq_j$ with $u\in V(B)\setminus\{c\}$ and $q_j\in V(Q)$.
      Choose a $c,s$-path $T$ in
      $D-(V(B)\setminus\{c\})$. Applying
      Lemma~\ref{lem:exceptional-admissible-path} to $(B,u,c)$, with
      exceptional vertex $z_0$ and parameter $p-6$, and closing the
      resulting paths with
      \[
        uq_j\cup q_jQa\cup av\cup vs\cup sTc\text{ or } uq_j\cup q_jQb\cup bv\cup vs\cup sTc
      \]
      gives the required cycles.
      The only remaining possibility is that $D$ has a non-bridge
      end-block $B$ containing $s$ in its interior and a bridge end-block
      with leaf $z_0$. Let $c$ be the cut-vertex of $B$. There is a
      $c,z_0$-path $T$ outside the interior of $B$, and the
      $2$-connectivity of $J$ gives a neighbor $q_j\in V(Q)$ of $z_0$.
      Apply Lemma~\ref{lem:exceptional-admissible-path} to $(B,c,s)$ with
      exceptional vertex $z_0$ and parameter $p-6$. The two fixed external paths are
      \[
        cTz_0\cup z_0q_j\cup q_jQa\cup av\cup vs
        \quad\text{and}\quad
        cTz_0\cup z_0q_j\cup q_jQb\cup bv\cup vs.
      \]
      They have different parities, so the parity observation applies.
      In every case, the chosen external path is fixed, and hence so are
      the two edges incident with $v$.
    \end{poc}

    We also need the following short bipartite version when a block
    cut-vertex and $z$ are both possible low-degree vertices.

    \begin{claim}\label{claim:rooted-bipartite-two-defects}
      Let $J$ be a $2$-connected bipartite subgraph of $H$, let $A$ be a
      set of at least three vertices in one part of $J$, and let
      $c\in V(J)\setminus A$. Suppose that
      $d_J(w)\ge p-1$ for every $w\in V(J)\setminus\{c,z\}$ and
      $A\subseteq S$. Then $F[V(J)\cup\{v\}]$ contains $p-6$ cycles of
      consecutive even lengths, all containing the same two edges incident
      with $v$.
    \end{claim}

    \begin{poc}
      We may assume that $z\notin A$, since otherwise Lemma~\ref{lem:exceptional-admissible-path}, applied to $(J,z,a)$ for any $a\in A\setminus\{z\}$ with exceptional vertex $c$ and parameter $p-6$, gives $p-6$ paths of consecutive even lengths, which together with $vz$ and $va$ give the required cycles.
      By Menger's theorem, there is a path in $J$ whose two ends lie in $A$
      and which contains $c$. Choose such a path $Q$ of minimum length,
      and let $a$ be one of its ends. Then $V(Q)\cap A$ consists precisely
      of the two ends of $Q$.

      Every vertex outside $Q$ has at most four neighbors on $Q$. Indeed,
      split $Q$ at $c$. If a vertex had three neighbors on one side, then,
      since $J$ is bipartite, the first and last of them are at even distance
      at least four on $Q$. Replacing the subpath between them by a two-edge
      path gives a shorter $A$-path containing $c$.
      Choose $s\in A\setminus V(Q)$, and let $D$ be the component of
      $J-V(Q)$ containing $s$. It follows that
        $d_D(w)\ge p-5$ for every $w\in V(D)\setminus\{z\}$.
      
      Suppose first that $D$ is $2$-connected. Since $J$ is $2$-connected,
      there are $u\in V(D)\setminus\{s\}$ and $q\in V(Q)$ with
      $uq\in E(J)$; otherwise $s$ would separate $D-s$ from $Q$.
      Apply Lemma~\ref{lem:exceptional-admissible-path} to $(D,u,s)$,
      with exceptional vertex $z$ and parameter $p-6$. The resulting
      $p-6$ admissible $u,s$-paths are closed by the fixed path
      \[
        uq\cup qQa\cup av\cup vs.
      \]
      Since $a$ and $s$ lie in the same part of $J$, this fixed path has
      the same parity as every $u,s$-path in $D$. Hence we obtain
      $p-6$ cycles of consecutive even lengths.

      Now suppose that $D$ is not $2$-connected. Every end-block of $D$
      has an edge from its interior to $Q$; otherwise its cut-vertex would
      also be a cut-vertex of $J$. Moreover, a bridge end-block can only
      have $z$ as its leaf, so there is at most one such end-block.
      If $D$ has a non-bridge end-block $B$, with cut-vertex $d$, whose
      interior does not contain $s$, choose $u\in V(B)\setminus\{d\}$ and
      $q\in V(Q)$ with $uq\in E(J)$. Also choose a $d,s$-path outside
      the interior of $B$. Apply
      Lemma~\ref{lem:exceptional-admissible-path} to $(B,u,d)$, with
      exceptional vertex $z$ and parameter $p-6$. Close the resulting paths
      using
      \[
        uq\cup qQa\cup av\cup vs
      \]
      followed by the fixed $s,d$-path. This closing path has the same
      parity as every $u,d$-path in $B$, again because $a$ and $s$ lie
      in the same part.

      Otherwise, $D$ has exactly two end-blocks: a non-bridge end-block
      $B$ containing $s$ in its interior and a bridge end-block with leaf
      $z$. Let $d$ be the cut-vertex of $B$. Choose a fixed $d,z$-path
      outside the interior of $B$ and an edge $zq$ with $q\in V(Q)$.
      Apply Lemma~\ref{lem:exceptional-admissible-path} to $(B,d,s)$ with
      exceptional vertex $z$ and parameter $p-6$, and close the resulting paths through the fixed
      $d,z$-path and
      \[
        zq\cup qQa\cup av\cup vs.
      \]
      This closing path has the same parity as every $d,s$-path in $B$.

      Thus, in every case, we obtain $p-6$ cycles of consecutive even
      lengths. The closing path is fixed and contains the edges $va$ and
      $vs$, so these two edges belong to every cycle.
    \end{poc}

    Suppose first that $H$ is $2$-connected. If $H$ is bipartite, choose
    two vertices $x,y\in S$ in the same part. Applying
    Lemma~\ref{lem:exceptional-admissible-path} to $(H,x,y)$, with
    exceptional vertex $z$ and parameter $p-2$, gives $p-2$ admissible
    $x,y$-paths. Since $H$ is bipartite, their lengths have common
    difference two. Adding the edges $vx$ and $vy$ gives the required
    cycles, all containing these two fixed edges. If $H$ is non-bipartite, apply
    Claim~\ref{claim:rooted-nonbipartite-core} with
    $J=H$, $A=S$, and $z_0=z$.

    We may therefore assume that $H$ is not $2$-connected. Let
    $B_1,\ldots,B_r$ be its end-blocks, and let $c_i$ be the cut-vertex
    of $B_i$. Every end-block contains a vertex of $S$ in its interior.
    Indeed, otherwise deleting its cut-vertex would separate that
    interior from $v$ in $F$. Moreover, a bridge end-block can only have
    $z$ as its leaf, by~\eqref{eq:rooted-H-degree}. Thus there is at most
    one bridge end-block, and if it occurs then $z\in S$.
    For every non-bridge end-block $B_i$, choose
    $x_i\in S\cap(V(B_i)\setminus\{c_i\})$. Every vertex of
    $B_i\setminus\{x_i,c_i,z\}$ retains all its $H$-neighbors in $B_i$.
    Hence Lemma~\ref{lem:exceptional-admissible-path}, applied to
    $(B_i,x_i,c_i)$ with exceptional vertex $z$ and parameter $p-6$, gives
    a family $\mathcal P_i$ of $p-6$ admissible $x_i,c_i$-paths. For
    the possible bridge end-block, let $x_i:=z$ and let $\mathcal P_i$
    consist of the single edge $zc_i$.

    If some non-bridge family $\mathcal P_i$ has common difference one,
    choose another end-block $B_j$ and fix one path from $\mathcal P_j$.
    Join the two end-blocks by a fixed $c_i,c_j$-path whose internal
    vertices avoid their interiors, and close through the two-edge path
    $x_ivx_j$. As the lengths in $\mathcal P_i$ are consecutive integers,
    one parity class gives $(p-6)/2$ cycles of consecutive even
    lengths. The two edges at $v$ are fixed.
    Hence we may assume that every non-bridge family has common
    difference two.

    \begin{claim}\label{claim:rooted-at-least-three-endblocks}
      If $r\ge3$, then the conclusion of the lemma holds.
    \end{claim}

    \begin{poc}
      Fix a spanning tree $T$ of $H$ and a root $r_0$. For each $i$, choose
      any path $P_i\in\mathcal P_i$ and label $B_i$ by
      \[
        \lambda_i:=|P_i|+\operatorname{dist}_T(r_0,c_i)\pmod2.
      \]
      This is independent of the choice of $P_i$, since every non-bridge
      family has common difference two, while a bridge family contains only
      one path.

      Among any three end-blocks, two have the same label, say $B_i$ and
      $B_j$. The path $c_iTc_j$ avoids the interiors of these two
      end-blocks. For any $P_i\in\mathcal P_i$ and
      $P_j\in\mathcal P_j$, the cycle obtained from $P_i$, $P_j$,
      $c_iTc_j$, and $x_ivx_j$ has parity
      \[
        |P_i|+|P_j|+\operatorname{dist}_T(c_i,c_j)
        \equiv\lambda_i+\lambda_j\equiv0\pmod2.
      \]
      Thus all these cycles are even. If both end-blocks are non-bridges,
      the pairwise sums of the two length sets form an arithmetic progression of common
      difference two with $2p-13$ terms. If one is the bridge end-block,
      varying the other family gives $p-6$ terms. Both bounds are stronger
      than required, and all cycles contain the fixed edges $vx_i$ and
      $vx_j$.
    \end{poc}

    It remains to treat $r=2$. Then the block-cut tree of $H$ is a path.
    \begin{claim}\label{cla:lem-complete-1}
      If $H$ is bipartite, then the conclusion of the lemma holds.
    \end{claim}
    \begin{poc}
      Let $S_i:=S\cap (V(B_i)\setminus \{c_i\})$. Suppose first that $x\in S_1$ and $y\in S_2$ lie in the same part of $H$. Since the block-cut tree of $H$ is a path and $x,y$ lie in the interiors of its two end-blocks, $(H,x,y)$ is a $2$-connected rooted graph. Lemma~\ref{lem:exceptional-admissible-path}, applied to $(H,x,y)$ with exceptional vertex $z$ and parameter $p-6$, gives $p-6$ admissible $x,y$-paths. Since $H$ is bipartite, their lengths have common difference two. Adding the edges $vx$ and $vy$ gives $p-6$ cycles of consecutive even lengths through $v$.

      Next, assume that $S_1$ and $S_2$ lie in different parts of $H$. If there exists some vertex $y\in S\setminus
      \bigl((V(B_1)\setminus\{c_1\})
      \cup(V(B_2)\setminus\{c_2\})\bigr)$, choose $x\in S_1\cup S_2$ that lies in the same part as $y$. Fix a path between $x$ and $y$. This path contains a non-bridge block. Indeed, if $x\ne z$, the
      end-block containing $x$ is already non-bridge. If $x=z$, note that the path between $x$ and $y$ has length at least $2$. There exists a vertex $w$ on the path such that $d_{H}(w)\ge p-1$. Thus one of the blocks containing $w$ is non-bridge.
      Choose such a non-bridge block $D$, taking the end-block containing
      $x$ when $x\ne z$, and otherwise the first non-bridge block met from
      $x$. Let $a,b$ be the first and last vertices of the fixed path in
      $D$. Since the block-cut tree of $H$ is a path, every vertex of
      $D\setminus\{a,b,z\}$ retains all its $H$-neighbors in $D$ and hence
      has degree at least $p-1$ in $D$. Apply
      Lemma~\ref{lem:exceptional-admissible-path} to $(D,a,b)$ with
      exceptional vertex $z$ and parameter $p-6$.
      Replacing the $a,b$-subpath of the fixed path by the resulting
      admissible paths gives $p-6$ admissible $x,y$-paths. Since $H$ is
      bipartite and $x,y$ lie in the same part, their lengths have common
      difference two. Adding $vx$ and $vy$ gives $p-6$ cycles of consecutive
      even lengths.

      Finally, suppose that every vertex of $S$ lies in the interiors of
      $B_1$ and $B_2$. One of $S_1,S_2$, say $S_i$, has size at least
      three. The set $S_i$ lies in one part of $B_i$, and every vertex of
      $B_i\setminus\{c_i,z\}$ has degree at least $p-1$ in $B_i$.
      Claim~\ref{claim:rooted-bipartite-two-defects}, applied with
      $J=B_i$, $A=S_i$, and $c=c_i$, gives the required cycles.
    \end{poc}

    \begin{claim}\label{claim:rooted-two-endblocks-nonbipartite}
      If $H$ is non-bipartite, then the conclusion of the lemma holds.
    \end{claim}

    \begin{poc}
      Suppose first that both end-blocks are non-bridges. Fix a
      $c_1,c_2$-path through the intermediate blocks. Combining this
      path with one path from each of $\mathcal P_1$ and $\mathcal P_2$, and with the
      two-edge path $x_1vx_2$ gives cycles whose lengths form an arithmetic
      progression of common difference two with $2p-13$ terms. All these
      cycles have the same parity.

      If they are even, we are done. If these cycles are odd, use a non-bipartite block on the block chain to reverse their parity: replace the fixed path through it by one of the opposite parity if it is internal, and otherwise fix a path of the opposite parity in that end-block and vary the family in the other end-block; in either case we obtain at least $p-6$ consecutive even lengths.

      We may therefore assume that one end-block is the bridge $zc_0$.
      Let $B$ be the other end-block, with cut-vertex $c$, and let
      $S_B:=S\cap(V(B)\setminus\{c\})$. Choose $x\in S_B$, and let
      $\mathcal P$ be the corresponding family of $p-6$ admissible $x,c$-paths in $B$.
      Fix the $c,z$-path along the block chain. Together with the path
      $xvz$, it closes every path in $\mathcal P$ into a cycle. These
      cycles have common difference two and the same parity. If they are
      even, we are done. If a non-bipartite block occurs outside $B$,
      replace the segment of the fixed $c,z$-path inside that block by
      a path of the opposite parity. Hence we may assume that $B$ is the
      unique non-bipartite block of $H$.

      If $|S_B|\ge3$, apply
      Claim~\ref{claim:rooted-nonbipartite-core} with
      $J=B$, $A=S_B$, and $z_0=c$. Indeed, every vertex of
      $B\setminus\{c\}$ retains all its $H$-neighbors in $B$, and hence
      has degree at least $p-1$ in $B$.

      It remains to consider $|S_B|\le2$. Since
      \(
        S\cap(B\cup\{z\})\subseteq S_B\cup\{c,z\}
      \)
      and $|S|\ge5$, some vertex of $S$ lies outside $B\cup\{z\}$.
      Consequently, the block chain between $B$ and the bridge $zc_0$
      contains a non-bridge internal block $D$. Otherwise every vertex of
      this part of the chain, except $z$, would have degree at most two in
      $H$, contrary to~\eqref{eq:rooted-H-degree}.
      Let $d_-,d_+$ be the two cut-vertices through which $D$ meets the
      neighboring blocks, with $d_-$ on the $B$-side. Every vertex of
      $D\setminus\{d_-,d_+\}$ retains all its $H$-neighbors in $D$, and
      therefore has degree at least $p-1$ in $D$. Thus
      Lemma~\ref{lem:exceptional-admissible-path}, applied to
      $(D,d_-,d_+)$ with exceptional vertex $z$ and parameter $p-6$, gives
      $p-6$ admissible $d_-,d_+$-paths. Since $B$ is the unique
      non-bipartite block, $D$ is bipartite, so these paths have common
      difference two and the same parity.
      Fix the paths from $d_-$ to $c$ and from $z$ to $d_+$ along the
      block chain. Since $B$ is $2$-connected and non-bipartite, it
      contains $x,c$-paths of both parities. Choose one so that the path
      formed by
      \[
        d_-\text{ to }c,\quad c\text{ to }x\text{ in }B,\quad
        xvz,\quad\text{and}\quad z\text{ to }d_+
      \]
      has the same parity as the admissible paths in $D$. Its internal
      vertices avoid $D$, so it closes all these paths into $p-6$ 
      cycles of consecutive even lengths. All of them contain the fixed edges
      $vx$ and $vz$.
    \end{poc}

    The two claims settle the case $r=2$. Together with the preceding cases, they cover all possibilities for $H$. Thus the lemma holds.
  \end{proof}

  The following corollary shows that the two fixed edges in Lemma~\ref{lem:rooted-even-one-exception} can be chosen to avoid the edge $vz$.

  \begin{corollary}\label{cor:avoid-root-edge}
    Let $p\ge8$ be an even integer, let $F$ be a $2$-connected graph, and let $v,z\in V(F)$ be distinct vertices with $vz\in E(F)$. Suppose that $d_F(v)\ge6$ and that $d_F(w)\ge p$ for every $w\in V(F)\setminus\{v,z\}$. Then there are distinct vertices $x,y\in N_F(v)\setminus\{z\}$ and an even integer $\ell$ such that $F$ contains cycles $D_0,\ldots,D_{(p-8)/2}$ satisfying $vx,vy\in E(D_i)$ and $|D_i|=\ell+2i$ for every $i$.
  \end{corollary}

  \begin{proof}
    Let $F':=F-vz$. The graph $F'$ is connected. If it is $2$-connected, then Lemma~\ref{lem:rooted-even-one-exception}, applied to $F'$ with distinguished vertex $v$ and exceptional vertex $z$, gives the result, since $d_{F'}(v)\ge5$.

    Suppose that $F'$ is not $2$-connected. Then every cut-vertex of $F'$ separates $v$ from $z$. Indeed, if a cut-vertex $c$ left $v$ and $z$ in the same component of $F'-c$, then adding the edge $vz$ would not reconnect any other component, and $c$ would also be a cut-vertex of $F$, which means $F$ is not $2$-connected, a contradiction. Moreover, $F'-c$ has no component containing neither $v$ nor $z$, because adding $vz$ would not reconnect that component and $c$ would again be a cut-vertex of $F$. As a consequence, the block-cut tree of $F'$ is a path, with $v$ and $z$ in the interiors of its two end-blocks. Let $B$ be the end-block containing $v$, and let $c$ be its cut-vertex. Then $d_B(v)=d_{F'}(v)\ge5$, so $B$ is a non-bridge block and hence is $2$-connected, while every vertex of $B\setminus\{v,c\}$ has all its $F'$-neighbors in $B$ and hence has degree at least $p$ in $B$. Applying Lemma~\ref{lem:rooted-even-one-exception} to $B$, with distinguished vertex $v$ and exceptional vertex $c$, we get the required family of cycles.
  \end{proof}

  \subsection{Proof of Lemma~\ref{lem:dense-expander}}\label{subsec:dense-expander-proof}

The proof has four main stages. First, the regularity lemma and a long
cycle in the reduced graph yield an even cycle of length $(1-o(1))d$,
together with flexible paths between a large set of typical vertices.
Second, we repeatedly apply four local augmentations to extend this
cycle. Third, when the procedure terminates, we obtain a stable
decomposition
$V(G)=A\mathbin{\dot\cup}V(C)\mathbin{\dot\cup}W$ in which $G[A]$ is
sparse, $G[W]$ has large minimum degree, $C$ is an even cycle of length
at least $(1-o(1))d$, and $H\subseteq G[A\cup V(C)]$. Finally, we extend
the parity labelling to $A\cup V(C)$ and combine the resulting
variable-length paths with admissible path families in $G[W]$ to obtain
consecutive even cycles.
    \begin{proof}[Proof of Lemma~\ref{lem:dense-expander}]

    Suppose, for contradiction, that $G$ contains no $d/2$ consecutive
    even cycle lengths. By~Proposition~\ref{prop:minicoeg}, $n\ge d+2$,
    $\delta(G)\ge d/2$, and $G$ is $2$-connected. By replacing $H$ with
    $G[V(H)]$, we may assume that $H$ is an induced subgraph of $G$, since
    adding edges preserves both the degree property and
    vertex-connectivity.
    Choose $r_1\in\mathbb N$ and $\xi,\tau>0$ so that
    \[
      0<1/h\ll1/r_1\ll\xi\ll\tau\ll1/K\ll\nu\ll \eps_1,\eps_2\ll \eps.
    \]
    Apply Lemma~\ref{regularity} to $H$ with $\beta=(1-\eps)d/h$, and let
    $H^\circ$, $R^\ast$, and $m$ denote the resulting pure graph, reduced
    graph, and cluster size. Lemma~\ref{lem:Erdos-cycle} gives a cycle
    $\Gamma=v_1\cdots v_\ell v_1$ in $R^\ast$ with
    $\ell\ge((1-\eps)d/h-3\tau)r$.
    Since $rm\ge(1-\xi)h$ and $\xi\ll \tau\ll \eps$, we have $(1-\eps-\tfrac{1}{100}\eps)d\le \ell m\le h\le Kd$.

    By Lemma~\ref{lem:paths-on-reduced-cycle} with $(K_0,d_{\ast}, \alpha)=(2K,(1-\eps)d, \eps)$ on $\Gamma$,
    let $U_i,\mathcal T_i,\mathcal F$, and $\mathcal T$ denote the sets, and let $M_0$ denote the constant supplied by that lemma. Then
    $(1-2\eps)d\le |\mathcal F|\le(1-\eps)d,~|\mathcal T|\ge (1-4\sqrt{\xi})|\mathcal{F}|$.
    By taking $r_1$ sufficiently large, we may assume that $M_0\ge 200$.
    Moreover, the graph $H^\circ[\mathcal F]$ contains a cycle of every even length at most $(1-2\eps)d$. Let $C_0$ be the even cycle in $H^{\circ}[\mathcal{F}]$ of length exactly $(1-2\eps)d$.
    Meanwhile, for distinct $x,y\in\mathcal T$ and
    $F\subseteq\mathcal F\setminus\{x,y\}$ with $|F|\le M_0/2+1$,
    Lemma~\ref{lem:paths-on-reduced-cycle}\ref{lem4.6-3} gives an integer
    $t\le M_0/2$ such that $H^\circ[\mathcal F]-F$ contains
    $xy$-paths of lengths
    \begin{equation}\label{eq:selected-path-span}
      t,t+2,\ldots,
      t+(1-3\eps)d.
    \end{equation}
    If $R_\Gamma:=R^\ast[V(\Gamma)]$ is bipartite, let $\chi_{\mathcal F}$ be its bipartition labelling on $\mathcal T$. If $R_\Gamma$ is non-bipartite, fix any map $\chi_{\mathcal F}:\mathcal T\to\{0,1\}$, for instance the parity of the index of the cluster containing the vertex. In the bipartite case the first length in~\eqref{eq:selected-path-span} has parity $\chi_{\mathcal F}(x)+\chi_{\mathcal F}(y)$, while in the non-bipartite case either parity may be prescribed; in particular, we may always choose this same parity.

    Since $h\le Kd$ and $\tau,\xi\ll\eps/K$, we have
    $(\tau+2\xi)h\le\eps d$. Hence every vertex loses at most $\eps d$
    incident edges when passing from $H$ to $H^\circ$.

    \medskip
    \noindent
    \textbf{The cleaning procedure.}
    Starting with the even cycle $C_0\subseteq G[\mathcal F]$ constructed
    above, define for each current even cycle $C_r$
    \[
      A_r:=\left\{x\in V(G)\setminus V(C_r):d_{C_r}(x)\ge\left(\tfrac13+20\eps\right)d-2r\right\}
    \]
    and
    \[
      B_r:=\left\{x\in V(G)\setminus\bigl(V(C_r)\cup A_r\bigr):d_{A_r}(x)\ge20\eps d\right\}.
    \]
    \begin{figure}[H]
      \begin{center}
        \includegraphics[width=0.8\textwidth]{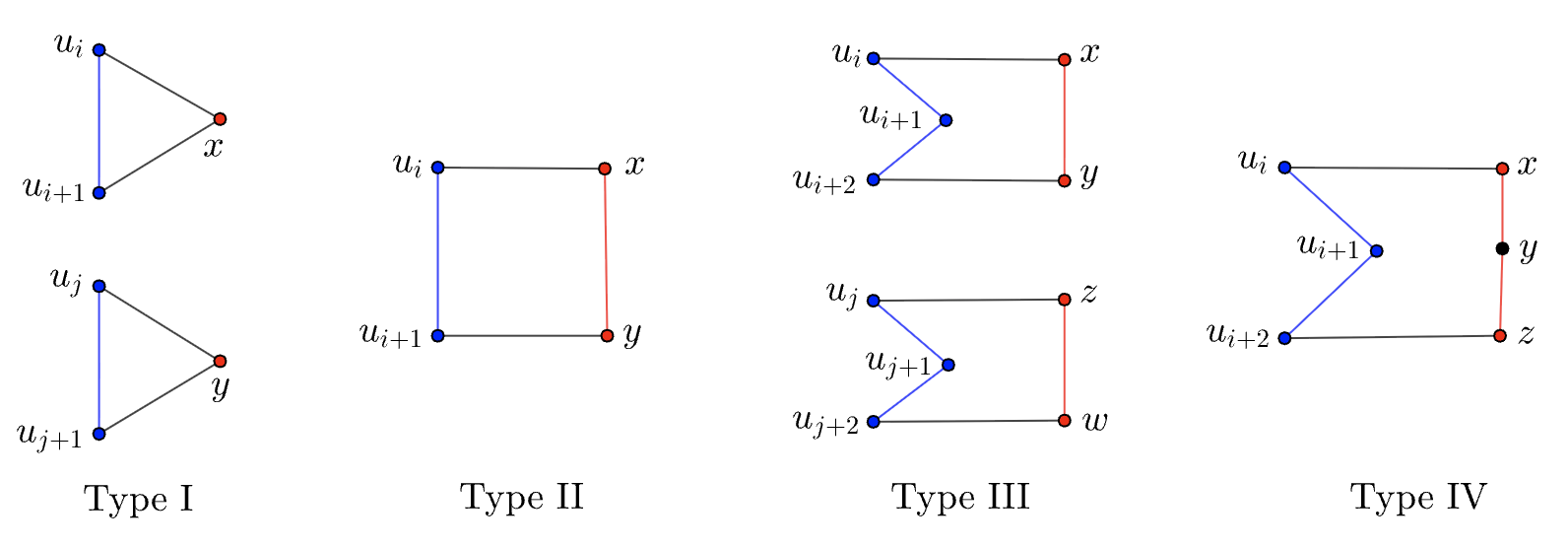}
        \caption{The four augmenting configurations. The blue vertices and edges belong to $C_r$; in Types I--III the red vertices and edges belong to $G[A_r]$, while in Type IV the two endvertices lie in $A_r$ and the middle vertex may lie in $A_r\cup B_r$.}
        \label{fig:4types}
      \end{center}
    \end{figure}
    Let $W_r:=V(G)\setminus\bigl(V(C_r)\cup A_r\cup B_r\bigr)$. If one
    of the four configurations in Figure~\ref{fig:4types} occurs for
    $(C_r,A_r,B_r)$, perform the corresponding replacement to obtain an even
    cycle $C_{r+1}$ of length $|C_r|+2$, and then update
    $A_{r+1},B_{r+1},W_{r+1}$. Because each replacement increases the cycle
    length, the procedure terminates unless it has already produced the
    required interval. Write $C_r:=u_1u_2\cdots u_{\ell_r}u_1$, with
    subscripts read modulo $\ell_r$. The four configurations are as follows.

    \begin{enumerate}[label=\textup{(\roman*)}]
      \item[\textbf{Type I.}] There are two distinct vertices $x,y\in A_r$ and two edges $u_iu_{i+1},u_ju_{j+1}\in E(C_r)$ such that $xu_i,xu_{i+1}$, $yu_j,yu_{j+1}\in  E(G)$. The two edges  $u_iu_{i+1},u_ju_{j+1}$ of $C_r$ share at most one vertex.
      \item[\textbf{Type II.}] There is an edge $xy\in E(G[A_r])$ and an edge $u_iu_{i+1}\in E(C_r)$ such that $u_ix,yu_{i+1}\in E(G)$.
      \item[\textbf{Type III.}] There are two vertex-disjoint edges
      $xy,zw\in E(G[A_r])$ and two vertex-disjoint subpaths
      $u_iu_{i+1}u_{i+2}$ and $u_ju_{j+1}u_{j+2}$ of $C_r$ such that
      $u_ix,yu_{i+2},u_jz,wu_{j+2}\in E(G)$.
      \item[\textbf{Type IV.}] There are distinct vertices
      $x,y,z\in V(G)\setminus V(C_r)$ with $x,z\in A_r$ and
      $y\in A_r\cup B_r$ such that $xyz$ is a path. Moreover, $C_r$
      contains a subpath $u_iu_{i+1}u_{i+2}$ with
      $u_ix,zu_{i+2}\in E(G)$.
    \end{enumerate}

    We perform the replacements as follows. In Type I, replace the edges
    $u_iu_{i+1}$ and $u_ju_{j+1}$ by the paths $u_ixu_{i+1}$ and
    $u_jyu_{j+1}$, respectively. In Type II, replace $u_iu_{i+1}$ by
    $u_ixyu_{i+1}$. In Type III, replace $u_iu_{i+1}u_{i+2}$ and
    $u_ju_{j+1}u_{j+2}$ by $u_ixyu_{i+2}$ and $u_jzwu_{j+2}$,
    respectively. In Type IV, replace $u_iu_{i+1}u_{i+2}$ by the simple
    path $u_ixyzu_{i+2}$, whose internal vertices are outside $C_r$.
    Each replacement produces an even cycle $C_{r+1}$
    with $|C_{r+1}|=|C_r|+2$. Suppose that the process terminates after
    $s$ replacements, and let $C:=C_s$, $A:=A_s$, $B:=B_s$, and $W:=W_s$.
    Since each replacement deletes at most two vertices of the current cycle, if $x\in A_r$, then $d_{C_{r'}}(x)\ge(\tfrac13+20\eps)d-2r'$ for every $r'\ge r$. Hence $x\in A_{r'}\cup V(C_{r'})$, and therefore every vertex belonging to some $A_r$ lies in $A\cup V(C)$. None of Types I--IV occurs for the final triple $(C,A,B)$.

    \begin{claim}\label{claim:C-small}
      The following properties hold for $A,B$, and $C$.
      \begin{enumerate}[label=\textup{(\arabic*)}]
        \item\label{c-sma1} $|C|\le d-2$.
        \item\label{c-sma2}
        $d_C(a)\ge(\tfrac13+18\eps)d$ for every $a\in A$, and all but at most one vertex
        $a\in A$ satisfy $d_C(a)\le |C|/2$.
        \item\label{c-sma3} $B=\varnothing$.
        \item\label{c-sma4} $e(G[A])\le1$.
      \end{enumerate}
    \end{claim}
    \begin{poc}
      Each replacement increases the cycle length by two, so
      $|C_r|=|C_0|+2r$ for $0\le r\le s$. Moreover,
      Lemma~\ref{lem:paths-on-reduced-cycle}\ref{lem4.6-2} gives a cycle of every even
      length from $4$ to $|C_0|$ in $H^\circ[\mathcal F]$, and the
      convention that $K_2$ has length $2$ supplies the first term.
      Together with $C_1,\ldots,C_s$, this gives every even length from
      $2$ to $|C|$. Thus $|C|\ge d$ would give cycles of lengths
      $2,4,\ldots,d$, contrary to our assumption. Since $|C|$ and $d$
      are even,~\ref{c-sma1} follows.

      Since $|C_0|\ge(1-2\eps)d-2$, it follows from $|C|\le d-2$ that
      $s=(|C|-|C_0|)/2\le\eps d$.
      Note that each replacement adds at most four vertices not in the cycle and deletes at most two vertices of the cycle. Hence
        $|V(C)\setminus V(C_0)|\le4s\le4\eps d$ and
      $|V(C_0)\setminus V(C)|\le2s\le2\eps d$.
      The definitions of the final sets give
      $d_C(a)\ge(\tfrac13+20\eps)d-2s\ge(\tfrac13+18\eps)d$ for every $a\in A$ and
      $d_A(b)\ge20\eps d$ for every $b\in B$. The first inequality proves
      the lower bound in~\ref{c-sma2}.

      Fix a clockwise orientation of $C$. For a set $S\subseteq V(C)$ and an integer $q$, let $S^{+q}$ (respectively, $S^{-q}$) denote the set obtained by shifting each vertex of $S$ by $q$ positions clockwise (respectively, counterclockwise) along the cycle. When $S=\{v\}$ is a singleton, we write $v^{+q}$ and $v^{-q}$ in place of $\{v\}^{+q}$ and $\{v\}^{-q}$, respectively. For $a\in A$, let
      $\mathcal E(a):=\{u_iu_{i+1}\in E(C):u_i,u_{i+1}\in N_C(a)\}$.
      Since Type~I is absent, one of the following holds:
      \begin{enumerate}[label=\textup{(\roman*)}]
        \item\label{case-1-i} there is at most one vertex $a^\ast\in A$
        (called the \emph{exceptional vertex}) for which
        $\mathcal E(a^\ast)\ne\varnothing$, or
        \item\label{case-1-ii} there is an edge $e^\ast\in E(C)$ such that for every $a\in A$, $\mathcal E(a)\subseteq\{e^\ast\}$.
      \end{enumerate}
      Indeed, distinct $a,b\in A$ and distinct edges
      $e\in\mathcal E(a)$ and $f\in\mathcal E(b)$ would form Type I.
      In~\ref{case-1-i}, $N_C(a)$ is independent for every
      $a\ne a^\ast$, so $d_C(a)\le |C|/2$. In~\ref{case-1-ii}, every
      $S\subseteq V(C)$ satisfies $e(C[S])\ge2|S|-|C|$, because the
      vertices outside $S$ meet at most $2(|C|-|S|)$ edges of $C$.
      Hence $|N_C(a)|>|C|/2$ would imply
      $e(C[N_C(a)])\ge2$, contradicting
      $\mathcal E(a)\subseteq\{e^\ast\}$. This proves the upper bound
      in~\ref{c-sma2}.

      We use the following cyclic observation. Let $D$ be a cycle and
      $S_1,S_2,S_3$ be independent subsets of $V(D)$ with
      $|S_i|>|D|/3$ for every $i\in[3]$. Then there are distinct
      $i,j\in[3]$ and vertices $s_i\in S_i$, $s_j\in S_j$ at distance
      $2$ on $D$. To see this, suppose, for a contradiction, that no such
      pair exists. Orient $D$ and let
      $I(v):=\{i\in[3]:v\in S_i\}$. Then $|I(v)|\ge2$ implies
      $I(v^{+2})=\varnothing$, while $|I(v)|=3$ also implies
      $I(v^{+1})=\varnothing$ as $S_1,S_2,S_3$ are independent. From each
      vertex $v$ with two labels, send one coin to $v^{+2}$; from each
      vertex with three labels, send one coin to each of $v^{+1}$ and
      $v^{+2}$. Each receiving vertex has no label, and no vertex receives
      two coins: independence makes the label sets of consecutive vertices
      disjoint, and two subsets of $[3]$ of size at least two cannot be
      disjoint. Hence
      \[
        \sum_{i=1}^3|S_i|
        =|\{v:I(v)\ne\varnothing\}|+
        \sum_{\substack{v\in V(D)\\I(v)\ne\varnothing}}(|I(v)|-1)
        \le |D|,
      \]
      a contradiction.

      Suppose that $b\in B$. Since $d_A(b)\ge20\eps d$ and
      $20\eps d\ge4$, choose distinct $a_1,a_2,a_3\in N_A(b)$, avoiding
      $a^\ast$ in~\ref{case-1-i}. In this case set $S_i:=N_C(a_i)$.
      In~\ref{case-1-ii}, fix one endvertex $r$ of $e^\ast$ and set
      $S_i:=N_C(a_i)\setminus\{r\}$. In either case the sets $S_i$ are
      independent in $C$, and
      $|S_i|\ge d_C(a_i)-1\ge(\tfrac13+18\eps)d-1>|C|/3$.
      The cyclic observation gives distinct $i,j\in[3]$ and a subpath
      $uvw$ of $C$ with $u\in S_i$ and $w\in S_j$. Replacing $uvw$ by
      $ua_i b a_jw$ is a Type IV augmentation, a contradiction.
      Thus $B=\varnothing$, proving~\ref{c-sma3}.

      To prove~\ref{c-sma4}, suppose first that $xyz$ is a path in $G[A]$.
      Since Types II and IV are absent, the sets
      $N_C(x)^{+1},N_C(y),N_C(z)^{-1}$ are pairwise disjoint.
      By~\ref{c-sma2}, each has size at least
      $(\tfrac13+18\eps)d>|C|/3$, a contradiction. Hence $G[A]$ is a
      matching.
      Suppose that this matching contains two edges $xy$ and $zw$. For
      either edge $ab\in\{xy,zw\}$, label its ends so that $b$ is not the
      exceptional vertex in~\ref{case-1-i}, and let
      $X=N_C(a)$ and $Y=N_C(b)$. In~\ref{case-1-ii}, instead fix one
      endvertex $r$ of $e^\ast$ and let
      $X=N_C(a)\setminus\{r\}$ and $Y=N_C(b)\setminus\{r\}$.
      In either case, $Y$ is independent in $C$ and
      $|X|,|Y|\ge(\tfrac13+18\eps)d-1$. Since Type II is absent,
      $X^{+1}\cap Y=\varnothing$, while independence gives
      $Y\cap Y^{-1}=\varnothing$. Hence
      \[
        |X^{+1}\cap Y^{-1}|
        \ge |X|+2|Y|-|C|
        \ge3\left(\left(\tfrac13+18\eps\right)d-1\right)-|C|>5.
      \]
      Each vertex in this intersection is the middle vertex of a
      two-edge subpath of $C$ whose ends are adjacent to $a$ and $b$,
      respectively. Thus each of $xy$ and $zw$ has more than five such
      subpaths. Choose one for $xy$. It meets at most five of those for
      $zw$, so a vertex-disjoint second subpath remains. These two subpaths
      and the edges $xy,zw$ form Type III, a contradiction. Therefore
      $e(G[A])\le1$, proving~\ref{c-sma4}.
    \end{poc}

    \begin{claim}\label{claim:A-B-bound}
      $W\neq \varnothing$.
    \end{claim}
    \begin{poc}
      Suppose that $W=\varnothing$. Since $B=\varnothing$, we have
      $n=|A|+|C|$. By Claim~\ref{claim:C-small}, $e(G[A])\le1$, and all
      but at most one vertex of $A$ have at most $|C|/2$ neighbors in $C$.
      Therefore
      \begin{equation}\label{eq:A-C-edge-bound}
        e_G(A,V(C))
        \le \max\left\{\tfrac{|C|}{2}(|A|-1)+|C|, \tfrac{|C|}{2}|A|\right\}\le
        \tfrac{(|A|+1)|C|}{2}.
      \end{equation}
      Since $e(G[V(C)])\le\binom{|C|}{2}$, it follows that
      \begin{align*}
        e(G)
        &\le1+\binom{|C|}{2}+\tfrac{(|A|+1)|C|}{2}
        =1+\tfrac{n|C|}{2}
        \le1+\tfrac{(d-2)n}{2}<\tfrac{(d-1)(n-1)}{2}.
      \end{align*}
      The final inequality holds since $n\ge d+2$. This is a contradiction. Therefore $W\ne\varnothing$.
    \end{poc}

    Since $B=\varnothing$, the definitions of $A$ and $W$ give
    $d_C(w)<(1/3+20\eps)d$ and $d_A(w)<20\eps d$ for every $w\in W$.
    Consequently, $d_G\bigl(w,A\cup V(C)\bigr)
    <\left(\tfrac13+40\eps\right)d$.
    Since $\delta(G)\ge d/2$, we see that for every $w\in W$,
    \begin{equation}\label{eq:degreeofW}
          d_{G[W]}(w)>(1/6-40\eps)d\ge d/7.
    \end{equation}
    In particular,
    $|W|\ge d/7+1$.

\begin{claim}\label{claim:R-outside-H}
  $W\cap V(H)=\varnothing$.
\end{claim}

\begin{poc}
Suppose otherwise, and let $W_H:=W\cap V(H)\neq\varnothing$. For every $v\in W_H$, we have $d_G(v,A\cup V(C))<(\tfrac13+40\eps)d$. Since $\delta(H)\ge(1-\eps)d/2$, this implies that
\begin{equation}\label{eq:RH-min-degree}
  d_{H[W_H]}(v)
  \ge d_H(v)-d_G(v,A\cup V(C))
  \ge \tfrac d7.
\end{equation}
Thus $|W_H|\ge d/7$.
Since $V(C_0)\subseteq\mathcal F$, $|\mathcal F|\le(1-\eps)d$, and $|V(C_0)|\ge(1-2\eps)d-2$, we have $|\mathcal F\setminus V(C_0)|\le\eps d+2$. Also, $|V(C_0)\setminus V(C)|\le2\eps d$. As $W$ is disjoint from $V(C)$, it follows that $|W_H\cap\mathcal F|\le4\eps d$ for sufficiently large $d$. Hence
$\delta(H[W_H\setminus\mathcal F])\ge d/8$.

Let $R$ be a spanning bipartite subgraph of $H[W_H\setminus\mathcal F]$ with a maximum number of edges. By maximality, $\delta(R)\ge d/16$. Let $R'$ be a nontrivial end-block of a component of $R$. Then every vertex of $R'$, except possibly its cut-vertex, has degree at least $d/16$ in $R'$.
Since $H$ is $\nu d$-connected, Menger's theorem gives at least $\nu d$ pairwise vertex-disjoint paths from $\mathcal F$ to $R'$. By truncating these paths, we may assume that their interiors are disjoint from $\mathcal F\cup V(R')$. Recall that $|\mathcal F\setminus\mathcal T|\le4\sqrt{\xi}|\mathcal F|\le4\sqrt{\xi}d$ and $\xi\ll\nu$. Hence at least $20$ of these paths start in $\mathcal T$.
Since $\chi_{\mathcal F}$ takes two values and $R'$ is bipartite, the pigeonhole principle gives two paths $P_1,P_2$ such that
\begin{enumerate}[label=(\roman*)]
  \item $|P_1|\equiv |P_2|\pmod2$;
  \item their initial vertices have the same $\chi_{\mathcal F}$-value and their final vertices lie in the same part of $R'$.
\end{enumerate}
Write $P_j=u_j\cdots v_j$ for $j\in[2]$, where $u_j\in\mathcal T$ and $v_j\in V(R')$. By~\eqref{eq:selected-path-span}, there are at least $(1-3\eps)d/2$ $u_1u_2$-paths whose lengths are consecutive even integers. Applying Lemma~\ref{lem:exceptional-admissible-path} to $(R',v_1,v_2)$, with the possible cut-vertex as the exceptional vertex and parameter $d/16-1$, gives $d/16-1$ $v_1v_2$-paths whose lengths are consecutive even integers.
Together with $P_1$ and $P_2$, these paths form more than $d/2$ cycles of consecutive even lengths. This is a contradiction.
\end{poc}
     In particular, $H\subseteq G[A\cup V(C)]$. Let $U=A\cup V(C)$.
    
    \medskip
    
    \noindent\textbf{Extending the parity labelling to $U$.}
First, we establish the following property.
\begin{enumerate}[label=$({\ast})$]
\item Every vertex $v\in U\setminus\mathcal T$ has at least $\nu d/4$
internally vertex-disjoint paths in $G[U]$, each of length at most $4K/\nu$,
ending at distinct vertices of $\mathcal T$.
\end{enumerate}

Suppose first that $v\in V(H)\setminus\mathcal T$. Since $H$ is $\nu d$-connected and $|\mathcal T|\ge\nu d$, Menger's theorem gives a family of $\nu d$ paths from $v$ to distinct vertices of $\mathcal T$ which intersect only at $v$. As $h\le Kd$, at least $\nu d/2$ of these paths have length at most $4K/\nu$.

Now suppose that $v\in U\setminus V(H)$. If $v\in A$, then Claim~\ref{claim:C-small}, $|V(C)\setminus V(C_0)|\le4\eps d$, and $|\mathcal F\setminus\mathcal T|\le4\sqrt\xi d$ show that $v$ has more than $d/3$ neighbors in $\mathcal T$. Suppose next that $v\in V(C)\setminus V(C_0)$, and let $r$ be the round in which $v$ was inserted into the cycle. If $v\in A_r$, then $d_{C_0}(v)\ge(1/3+20\eps)d-6r>d/3$, and again $v$ has more than $d/3$ neighbors in $\mathcal T$. If $v\in B_r$, then $v$ has at least $20\eps d$ neighbors in $A_r$. Every vertex of $A_r$ lies in $U$ by the persistence observation above and has more than $d/3$ neighbors in $\mathcal T$. We may therefore choose $\nu d/4$ distinct vertices in $N_{A_r}(v)$ and then greedily choose distinct endvertices in $\mathcal T$, giving the required paths of length two.

Let $\chi=\chi_{\mathcal F}$ on $\mathcal T$. For each $u\in U\setminus\mathcal T$, classify the paths above according to the value of $|Q|+\chi_{\mathcal F}(u')$ modulo two, where $u'$ is the endvertex of $Q$ in $\mathcal T$. Retain a collection of at least $\nu d/8$ paths with the same value and define $\chi(u)$ to be this value. Thus every retained path $Q$ from $u$ to $u'\in\mathcal T$ satisfies $|Q|+\chi_{\mathcal F}(u')\equiv\chi(u)\pmod2$.

We finally show that this extension has the following property.
\begin{enumerate}[label=$({\ast}{\ast})$]
    \item\label{propextension}
Let $q\in\{1,2,3\}$, and let
$(x_1,y_1),\ldots,(x_q,y_q)$ be pairs of vertices in $U$ whose $2q$ vertices are
distinct. For every integer $0\le i\le(1-3\eps)d/2$, there are
pairwise vertex-disjoint $x_jy_j$-paths $P_{i,j}$ in $G[U]$,
$j\in[q]$, such that
$\sum_{j=1}^q|P_{i,j}|=t+2i$ and $t\le M_0$.
If $R_\Gamma$ is bipartite, then
$t\equiv\sum_{j=1}^q(\chi(x_j)+\chi(y_j))\pmod2$; if $R_\Gamma$ is
non-bipartite, either parity of $t$ may be prescribed.
\end{enumerate}

For each endpoint among $x_1,y_1,\ldots,x_q,y_q$, take the trivial path if it belongs to $\mathcal T$, and otherwise choose one of the retained paths used to define its $\chi$-value. Since each nontrivial family has linear size and all these paths have bounded length, they may be chosen pairwise vertex-disjoint, avoiding the other endpoints, and with distinct endvertices $x'_1,y'_1,\ldots,x'_q,y'_q\in\mathcal T$. Use as a forbidden set all vertices of these paths lying in $\mathcal F$, except the endvertices in $\mathcal T$. By the choice of $r_1$, this set has size at most $M_0/2$.

Lemma~\ref{lem:paths-on-reduced-cycle}\ref{lem4.6-3}, applied with this
forbidden set to the pairs $(x'_1,y'_1),\ldots,(x'_q,y'_q)$, gives
pairwise vertex-disjoint $x'_jy'_j$-paths $P'_{i,j}$ in
$H^\circ[\mathcal F]$, $j\in[q]$, such that
$\sum_{j=1}^q|P'_{i,j}|=t'+2i$ for every
$0\le i\le(1-3\eps)d/2$, where $t'\le M_0/2$. In the bipartite case,
$t'\equiv\sum_{j=1}^q\bigl(\chi_{\mathcal F}(x'_j)+
\chi_{\mathcal F}(y'_j)\bigr)\pmod2$. In the non-bipartite case, we
prescribe the parity of $t'$ so that the resulting initial length has the
desired parity. Joining these paths to the corresponding paths from the
original endpoints gives the required $P_{i,j}$. Let $t$ be the sum of
$t'$ and the lengths of the connecting paths. In the bipartite case,
\[\sum_{j=1}^q|P_{i,j}|
  \equiv \sum_{j=1}^q\bigl(|Q_{x_j}|+|Q_{y_j}|+
      \chi_{\mathcal F}(x'_j)+\chi_{\mathcal F}(y'_j)\bigr)\equiv \sum_{j=1}^q\bigl(\chi(x_j)+\chi(y_j)\bigr)\pmod2.\]
In the non-bipartite case, the prescribed parity of $t'$ gives the
prescribed parity of $t$. By the choice of $r_1$, the total length of the
connecting paths is at most $M_0/2$, and hence $t\le M_0$.

    \medskip
    \noindent
    \textbf{Completion of the proof.}
      Recall that $R_\Gamma=R^\ast[V(\Gamma)]$ is the subgraph of the reduced graph induced by the vertices of $\Gamma$.
      Also recall that $\delta(G[W])\ge d/7$ and, by
      Claim~\ref{claim:R-outside-H}, $W\cap V(H)=\varnothing$. We divide into
      two cases according to whether $R_\Gamma$ is bipartite.

      \medskip
      
      \textbf{Case 1. The graph $R_\Gamma$ is non-bipartite.} Let $D$ be an end-block of $G[W]$. Let $b$ be its cut-vertex if it exists; otherwise, choose $b$ arbitrarily in $V(D)$. Since $\delta(G[W])\ge d/7$, every vertex of $D$ other than $b$ has degree at least $d/7$ in $D$. Since $G$ is $2$-connected, Menger's theorem gives two vertex-disjoint paths $P_i=x_i\cdots v_i$, $i\in[2]$, between $U$ and $D$. By truncating these paths, we may assume that $x_i\in U$, $v_i\in V(D)$, and their internal vertices lie outside $U\cup V(D)$.
Applying Lemma~\ref{lem:exceptional-admissible-path} to $(D,v_1,v_2)$, with $b$ as the possible exceptional vertex and parameter $d/7-1$, gives $d/7-1$ admissible $v_1v_2$-paths. Retaining one parity class if necessary, we obtain a family $\mathcal Q$ of at least $d/14-1$ paths whose lengths have common difference two. Let $j\in\{0,1\}$ be defined by $|Q|+|P_1|+|P_2|\equiv j\pmod2$ for every $Q\in\mathcal Q$.

Since $R_\Gamma$ is non-bipartite, property~\ref{propextension}, applied with $q=1$ and with the parity of the initial length prescribed to be $j$, gives $x_1x_2$-paths in $G[U]$ of lengths
$c,c+2,\ldots,c+(1-3\eps)d$, where $c\equiv j\pmod2$. Together with the paths in $\mathcal Q$ and the paths $P_1,P_2$, they yield more than $d/2$ consecutive even cycle lengths.

      \medskip
     
      \textbf{Case 2. The graph $R_\Gamma$ is bipartite.} Recall that $W\ne\varnothing$ and $\delta(G[W])\ge d/7$. Since $R_\Gamma$ is bipartite, we first give a new upper bound for $e(G[U])$ by the following claim.

      \begin{claim}\label{claim:U-edge-deficiency}
        $e(G[U])\le \tfrac{d-2}{2}|U|-\tfrac{d^2}{5}$.
      \end{claim}

      \begin{poc}
        Since $R_\Gamma$ is bipartite, so is the pure graph $H^\circ[\mathcal F]$. Hence, $e\bigl(H^\circ[V(C_0)]\bigr)\le |C_0|^2/4$.
        Since $H=G[V(H)]$ and every vertex loses at most $\eps d$ incident edges when passing to $H^\circ$, we have
        \[
          e(G[V(C_0)])\le \tfrac{|C_0|^2}{4}+\eps d^2
          \le\tfrac{d^2}{4}+\eps d^2.
        \]
        Moreover, $|V(C)\setminus V(C_0)|\le4\eps d$, and hence $e(G[V(C)])\le d^2/4+6\eps d^2$.
        By Claim~\ref{claim:C-small}, we have $e(G[A])+e_G(A,V(C))\le1+(|A|+1)|C|/2$. Therefore,
        \[
          e(G[U])=e(G[A\cup V(C)])
          \le1+\tfrac{(|A|+1)|C|}{2}+\tfrac{d^2}{4}+6\eps d^2.
        \]
        Since $|A|+|C|=|U|$, it follows that
        \begin{align*}
          e(G[U])
          &\le 1+\tfrac{(|U|-|C|+1)|C|}{2}+\tfrac{d^2}{4}+6\eps d^2\\
          &\le \tfrac{d-2}{2}|U|-\tfrac{|C|^2-|C|}{2}+\tfrac{d^2}{4}+6\eps d^2+1\le \tfrac{d-2}{2}|U|-\tfrac{d^2}{5}.
        \end{align*}
        The second inequality uses $|C|\le d-2$, while the last follows from $|C|\ge(1-2\eps)d-2$ and the hierarchy. This proves the claim.
      \end{poc}

      \begin{claim}\label{claim:large-RU-matching}
        The bipartite graph $G[W,U]$ contains a matching of size at least $d/100$.
      \end{claim}

      \begin{poc}
        Let $M$ be a maximal matching in $G[W,U]$. Suppose that $|M|<d/100$. Let $M_U=V(M)\cap U$ and $M_W=V(M)\cap W$. By the maximality of $M$, there is no edge between $W\setminus M_W$ and $U\setminus M_U$, and hence
        \begin{equation*}
          e(G)\le e(G[W\cup M_U])+e(G[U])+e_G(M_W,U\setminus M_U).
        \end{equation*}
       Since $G$ is the minimal counterexample, $e(G[W\cup M_{U}])\le \tfrac{d-1}{2}(|W|+|M_U|-1)$.
       Moreover, for each $w\in W$, $d_G(w,U)<(1/3+40\eps)d$. Hence
       $$
         e_G(M_W,U\setminus M_U)\le d/100\cdot (1/3+40\eps)d\le d^2/200.
       $$
       By Claim~\ref{claim:U-edge-deficiency}, $e(G[U])\le \tfrac{d-2}{2}|U|-\tfrac{d^2}{5}$. We obtain
       \begin{align*}
           e(G)&\le \tfrac{d-1}{2}(|W|+|M_U|-1)+d^2/200+\tfrac{d-2}{2}|U|-\tfrac{d^2}{5}\\
           &\le \tfrac{d-1}{2}(|W|+|U|-1)+|M_U|\cdot \tfrac{d-1}{2}-\tfrac{d^2}{10}\le \tfrac{d-1}{2}(|G|-1)-\tfrac{d^2}{20}.
       \end{align*}
       The final inequality holds because $V(G)=W\mathbin{\dot\cup}U$ and $|M_U|\le d/100$.
        This contradicts $e(G)\ge(d-1)(|G|-1)/2$.
      \end{poc}

      Fix a matching $M$ of size at least $d/100$ in $G[W,U]$ and the label function $\chi:U\rightarrow\{0,1\}$.

\begin{claim}\label{claim:disconnected-remainder}
  If $G[W]$ is disconnected, then $G$ contains $d/2$ consecutive even cycle lengths.
\end{claim}

\begin{poc}
  We choose two components $D_1$ and $D_2$ of $G[W]$, together with two independent edges $x_i^1y_i^1$ and $x_i^2y_i^2$ from $U$ to $D_i$ for each $i\in[2]$. Such edges exist since $G$ is $2$-connected. If $D_i$ is $2$-connected, Lemma~\ref{lem:exceptional-admissible-path}, applied to $(D_i,y_i^1,y_i^2)$ with parameter $d/7-1$, gives $d/7-1$ admissible $y_i^1y_i^2$-paths.

  Suppose that $D_i$ is not $2$-connected. Let $B_i$ be an end-block of $D_i$ with cut-vertex $c_i$. Since $G$ is $2$-connected, the edges between $U$ and $D_i$ contain a matching of size two, and the interior of $B_i$ has an edge to $U$. We may therefore choose the two independent attachment edges so that $y_i^1$ lies in the interior of $B_i$. If $y_i^2$ also lies in the interior of $B_i$, apply Lemma~\ref{lem:exceptional-admissible-path} to $(B_i,y_i^1,y_i^2)$, with $c_i$ as the possible exceptional vertex and parameter $d/7-1$. Otherwise, fix a $c_iy_i^2$-path in $D_i-(V(B_i)\setminus\{c_i\})$, apply the same lemma to $(B_i,y_i^1,c_i)$, and extend each resulting path by this fixed path. In either case, $D_i$ contains $d/7-1$ admissible $y_i^1y_i^2$-paths.

  Retaining one parity class if necessary, we obtain a family $\mathcal Q_i$ of at least $d/14-1$ $y_i^1y_i^2$-paths whose lengths form an arithmetic progression with common difference two. Let $\lambda_i\in\{0,1\}$ denote their common parity. The $U$-endvertex of any edge of $M$ can be included in such an attachment pair for its component, as can be seen by comparing the edge with a matching of size two between $U$ and that component. Choose an arbitrary attachment pair for one component. If an edge of $M$ meeting another component has its $U$-endvertex outside this pair, use an attachment pair containing that endvertex. Otherwise all but at most two edges of $M$ meet the first component, and we choose there an edge whose $U$-endvertex avoids an attachment pair of any other component. Thus the choices may be made so that $\{x_1^1,x_1^2\}\ne\{x_2^1,x_2^2\}$.

  Suppose first that $\lambda_i+\chi(x_i^1)+\chi(x_i^2)\equiv0\pmod2$ for some $i\in[2]$. By~\ref{propextension}, there are at least $(1-3\eps)d/2$ $x_i^1x_i^2$-paths in $G[U]$ whose lengths form an arithmetic progression with common difference two and have parity $\chi(x_i^1)+\chi(x_i^2)$. Together with the paths in $\mathcal Q_i$ and the edges $x_i^1y_i^1,x_i^2y_i^2$, they give more than $d/2$ consecutive even cycle lengths.

  We may therefore assume that $\lambda_i+\chi(x_i^1)+\chi(x_i^2)\equiv1\pmod2$ for both $i\in[2]$. Suppose first that the two sets $\{x_1^1,x_1^2\}$ and $\{x_2^1,x_2^2\}$ intersect. By relabelling, assume that $x_1^1=x_2^1$. Since the two sets are distinct, $x_1^2\ne x_2^2$. By the proof of~\ref{propextension}, with $x_1^1$ added to the forbidden set, there are at least $(1-3\eps)d/2$ $x_1^2x_2^2$-paths in $G[U]-x_1^1$ whose lengths form an arithmetic progression with common difference two and have parity $\chi(x_1^2)+\chi(x_2^2)$. Together with paths in $\mathcal Q_1,\mathcal Q_2$ and the four attachment edges, they form even cycles, since the sum of the two quantities $\lambda_i+\chi(x_i^1)+\chi(x_i^2)$ is even. The three path families give more than $d/2$ consecutive even cycle lengths.

  We may now assume that $x_1^1,x_1^2,x_2^1,x_2^2$ are all distinct. Applying~\ref{propextension} with $q=2$, we obtain pairwise vertex-disjoint $x_1^jx_2^j$-paths $P_{k,j}$, $j\in[2]$, such that $\sum_{j=1}^2|P_{k,j}|=t+2k$ for every $0\le k\le(1-3\eps)d/2$, where $t\equiv\sum_{j=1}^2\bigl(\chi(x_1^j)+\chi(x_2^j)\bigr)\pmod2$. Joining $P_{k,1},P_{k,2}$ with paths in $\mathcal Q_1,\mathcal Q_2$ and the four attachment edges gives even cycles, since the sum of the two quantities $\lambda_i+\chi(x_i^1)+\chi(x_i^2)$ is even. These path families again give more than $d/2$ consecutive even cycle lengths.
\end{poc}

By Claim~\ref{claim:disconnected-remainder}, we may assume that $G[W]$
is connected. By the pigeonhole principle, at least $d/200$ edges of $M$ have their $U$-endpoint in $\chi^{-1}(i)$ for some $i\in\{0,1\}$. Let $M'$ be the resulting submatching.

\begin{claim}\label{claim:two-connected-remainder}
  If $G[W]$ is $2$-connected, then $G$ contains $d/2$ consecutive even
  cycle lengths.
\end{claim}

\begin{poc}
We construct an auxiliary graph $F$ with vertex set $V(F)=W\cup\{z_0\}$. Retain all edges of $G[W]$, and add the edge $z_0w$ for every $w\in V(M')\cap W$. Since $G[W]$ is $2$-connected and $z_0$ has at least two neighbors, $F$ is also $2$-connected. Moreover, every vertex of $W$ still has degree at least $d/7$ in $F$.

Apply Lemma~\ref{lem:rooted-even-one-exception} to $F$ with root $z_0$ and any vertex of $W$ as the exceptional vertex. Let $s:=d/14-3$. We obtain even cycles $D_0,\ldots,D_{s-1}$ of consecutive lengths that contain the same two edges $z_0w,z_0w'$.
Let $aw,a'w'$ be the corresponding edges of $M'$. Since $M'$ is a matching, $a$ and $a'$ are distinct vertices of $\chi^{-1}(i)$. By~\ref{propextension}, there are at least $(1-3\eps)d/2$ consecutive even $aa'$-paths in $G[U]$. Replacing the path $wz_0w'$ in every $D_i$ by the edges $wa$ and $a'w'$ gives $d/14-3$ consecutive even $aa'$-paths whose internal vertices lie in $W$. Joining the two path families yields at least $(1-3\eps)d/2+d/14-4>d/2$ consecutive even cycle lengths.
\end{poc}

\begin{claim}\label{claim:one-connected-remainder}
  If $G[W]$ is connected but not $2$-connected, then $G$ contains
  $d/2$ consecutive even cycle lengths.
\end{claim}

\begin{poc}
Let $B_1,\ldots,B_q$ be the end-blocks of $G[W]$, with cut-vertices
$c_1,\ldots,c_q$. Since $G$ is $2$-connected, for every $i\in[q]$
there is an edge $a_ix_i$ from $U$ to the interior of $B_i$.
Apply Lemma~\ref{lem:exceptional-admissible-path} to $(B_i,x_i,c_i)$ with parameter
$d/7-1$. This gives a family of $d/7-1$ admissible
$x_ic_i$-paths with common difference
$\rho_i\in\{1,2\}$. When $\rho_i=2$, let $\sigma_i$ be the common
parity of these paths. Fix a spanning tree $T$ of $G[W]$ and a vertex
$v\in V(W)$.

We first consider two easy cases. Suppose that $\rho_i=1$ and
$a_j\ne a_i$ for some $j\ne i$. Fix an $x_jc_j$-path in $B_j$
and the $c_ic_j$-path in $T$. By~\ref{propextension}, there are at
least $(1-3\eps)d/2$ admissible $a_ia_j$-paths in $U$. Since the
paths in $B_i$ have common difference one, we may take the parity class
which makes the resulting cycles even. This class contains at least
$d/14-1$ paths. Combining the two path families gives more than
\(
  {(1-3\eps)d}/{2}+d/{14}-2>\tfrac d2
\)
consecutive even cycle lengths.
Next, suppose that $\rho_i=\rho_j=2$, $a_i\ne a_j$, and
\begin{equation}\label{eq:sumparity}
  \sigma_i+\operatorname{dist}_T(v,c_i)+\chi(a_i)
  \equiv
  \sigma_j+\operatorname{dist}_T(v,c_j)+\chi(a_j)
  \pmod2.
\end{equation}
Since
$\operatorname{dist}_T(c_i,c_j)\equiv
\operatorname{dist}_T(v,c_i)+\operatorname{dist}_T(v,c_j)\pmod2$,
the two end-block path families, the $c_ic_j$-path in $T$, and
the $a_ia_j$-paths in $U$ form even cycles. These three
families give more than $d/2$ consecutive even cycle lengths.

We may therefore assume that neither case occurs. It follows that
there are at most two distinct vertices among $a_1,\ldots,a_q$.
Indeed, if some $\rho_i=1$, then the first case implies that every
$a_j=a_i$. Otherwise all $\rho_i=2$, and among three distinct $a_i$,
two of the corresponding values in~\eqref{eq:sumparity} would be equal
by the pigeonhole principle.

By interchanging the labels $0$ and $1$ if necessary, assume that every edge of $M'$ has its $U$-endpoint in $\chi^{-1}(0)$. Delete from $M'$ at most two edges incident with vertices among $a_1,\ldots,a_q$. We can then choose two disjoint submatchings $M_z,M_w$, each of size $d/500$.
We now construct an auxiliary graph $F$. Start with $G[W]$ and add
three new vertices $z_0,z_1,w_0$. Join $z_0$ to vertices in $V(M_z)\cap W$, and join $w_0$ to vertices in $V(M_w)\cap W$. For every $B_i$ with $(V(B_i)\setminus\{c_i\})\cap (V(M_z)\cup V(M_w))=\varnothing$, add the edge
$w_0x_i$ if $a_i\in \chi^{-1}(0)$, and add the edge $z_1x_i$ if $a_i\in \chi^{-1}(1)$. Finally, add $z_0z_1$ and $z_1w_0$.
Then the following properties hold for the new graph $F$.
\begin{enumerate}[label=(\arabic*)]
    \item $F$ is $2$-connected.
    \item $d_F(x)\ge d/500$ for every vertex of $F$ except $z_1$.
\end{enumerate}
To see (1), deleting one of
$z_0,z_1,w_0$ leaves a connected graph. Now let $x\in W$. If
$G[W]-x$ is connected, then so is $F-x$. Otherwise, every component
of $G[W]-x$ contains the interior of some end-block $B_i$. If that interior contains a vertex in $M_z\cup M_w$, it is joined to $z_0$ or $w_0$.
If it contains no matching end, the edge $w_0x_i$ or $z_1x_i$ was added. Thus every component of $G[W]-x$ is joined to the path $z_0z_1w_0$, and hence $F-x$ is connected.
For (2), $d_F(z_0),d_F(w_0)\ge d/500$, while every vertex of $W$ has
degree at least $d/7$ in $F$.

Next, apply
Corollary~\ref{cor:avoid-root-edge} with root $z_0$,
exceptional vertex $z_1$, prescribed edge $z_0z_1$, and parameter
$d/500$. Let $s:=d/1000-3$. We obtain even cycles
$D_0,\ldots,D_{s-1}$ of consecutive lengths. They contain the same
two edges incident with $z_0$ and avoid $z_0z_1$. Let the two edges be $z_0y_1$ and $z_0y_2$.
It remains to replace the new vertices in each $D_i$.

Suppose first that $z_1w_0\notin D_i$. For each new vertex $z$ on $D_i$, the two edges of $D_i$ incident with $z$ both join $z$ to $W$; write them as $rz$ and $zr'$. Let $ar$ and $br'$ be the corresponding original edges. If $a=b$, replace $rzr'$ directly by $rar'$. If $a\ne b$, replace it by $ra$, $br'$ and an $ab$-path in $U$. Here $a,b\in \chi^{-1}(0)$ when $z\in\{z_0,w_0\}$, while $a,b\in \chi^{-1}(1)$ when $z=z_1$.

The other case is that $D_i$ uses the edge $z_1w_0$.
In this case it contains a path $rz_1w_0r'$. Let
$ar$ and $br'$ be the two corresponding original edges. Then
$a\in\chi^{-1}(1)$ and $b\in\chi^{-1}(0)$, and we replace this path by
$ra$, an $ab$-path in $U$, and $br'$. In addition, replace the
two-edge path of $D_i$ through $z_0$ exactly as in the preceding case.

In total, at most three disjoint pairs of vertices in $U$ need to be joined. They are disjoint because $M_z$ and $M_w$ avoid the vertices $a_i$, while the additional edges incident with $z_1$ and $w_0$ correspond to different $\chi$-classes. After omitting any pair with equal ends, property~\ref{propextension} gives pairwise vertex-disjoint paths for all remaining pairs, with their total length ranging through an arithmetic progression of common difference two and with first term at most $M_0$. Hence every replacement gives a simple cycle in $G$.
In the first case, the two ends of each replacement pair have the same $\chi$-value, so the path between them has even length. For the pair arising from $z_1w_0$, the two ends have different $\chi$-values, so the path between them has odd length. The two attachment edges are replaced one-for-one, while the additional edge $z_1w_0$ is removed. Therefore every replacement changes the cycle length by an even number.

By~\ref{propextension} with $q$ at most $3$, it follows that every $D_i$ gives an interval of at least $(1-3\eps)d/2$ consecutive even cycle lengths. The first length of
this interval differs from $|D_i|$ by a bounded even number. Since
$|D_0|,\ldots,|D_{s-1}|$ are consecutive even integers, these intervals
overlap for sufficiently large $d$. Since the first-length shifts are bounded by $M_0$, their union contains at least $(1-3\eps)d/2+d/1000-2M_0-4\ge d/2$ consecutive even cycle lengths for sufficiently large $d$.
\end{poc}

Claims~\ref{claim:two-connected-remainder} and~\ref{claim:one-connected-remainder} complete the proof.
\end{proof}

\section{Intermediate case: proof of Lemma~\ref{lem:medium-expander}}\label{section-med}

  In this section, we prove a stronger version of Lemma~\ref{lem:medium-expander}.
  \begin{lemma}\label{lem:dense}
    Let $0<\tfrac{1}{n},\tfrac{1}{d}\ll \tfrac{1}{K}\ll \eps_1,\eps_2\ll 1/C,1/5$, with $\tfrac{1}{4}\log^{1000}n\leq d\leq \tfrac{n}{K}$ and $Cd\in 2\mathbb{N}$. If $G=(V,E)$ is an $n$-vertex $(\eps_1,\eps_2d)$-expander with $\delta(G)\geq \tfrac{d}{2}$, then $G$ contains $\tfrac{Cd}{2}$ consecutive even cycle lengths.
  \end{lemma}

  Note that Lemma~\ref{lem:medium-expander} follows from Lemma~\ref{lem:dense} with %$(G,C,d):=(H,2,(1-\eps)d)$. 
  $(n,G,C)=(h,H,2)$ and $(d,\eps_2)=(2\lfloor(1-\eps)d/2\rfloor,
  \tfrac{\eps_2d}{2\lfloor(1-\eps)d/2\rfloor})$.
  Throughout this section, let $m$ be the smallest even integer larger than $\log^4\tfrac{n}{d}$. Then
\begin{equation}\label{equ:medium}
  d\ge m^{200}
  \qquad\text{and}\qquad
  n\ge dm^{200}.
\end{equation}
Every auxiliary statement below that is written with the lower bound $d\geq\log^{1000}n$ remains valid with $d\geq\tfrac14\log^{1000}n$, since this lower bound is used only to guarantee~\eqref{equ:medium}. We use this strengthened version without further comment. Let
$$
  L_G:=\{v\in V(G): d_G(v)\geq dm^{31}\}.
$$
We divide the proof into two cases according as $|L_G|\geq 11d/50$ or $|L_G|<11d/50$.
  
\subsection{Tools}
We now introduce the definitions of a \emph{unit} and a \emph{web}, which may be viewed as ``tree-like'' structures.
  \begin{definition}[unit; see~\cite{Liu-balanced-sub}]\label{defn:unit}
    For $h_1,h_2,h_3\in \mathbb{N}$, a graph $F$ is an \emph{$(h_1,h_2,h_3)$-unit} if it contains distinct vertices $u$ (the \emph{core} vertex of $F$) and $x_1,\ldots,x_{h_1}$, and $F=\bigcup_{i\in[h_1]}(P_i\cup S_{i})$, where
    \begin{itemize}
      \item $\mathcal{P}=\{P_i:i\in[h_1]\}$ is a collection of pairwise internally vertex-disjoint paths, each of length exactly $s$ for some $s\leq h_3$, such that $P_i$ is a $u,x_i$-path, and
      \item $\mathcal{S}=\{S_i:i\in[h_1]\}$ is a collection of vertex-disjoint $h_2$-stars such that $S_{i}$ has center $x_i$ and $\bigcup_{i\in[h_1]}(V(S_{i})\backslash \{x_i\})$ is disjoint from $V(\mathcal{P})$.
    \end{itemize}
  \end{definition}
  We call $S_{i}$ a \emph{pendant} star in the unit $F$ and every such path $P_i$ a \emph{branch} of $F$. Define the \emph{exterior} $\mathsf{Ext}(F):=\bigcup_{i\in[h_1]}(V(S_{i})\backslash \{x_i\})$ and \emph{interior} $\mathsf{Int}(F):=V(F)\backslash \mathsf{Ext}(F)$.

  \begin{figure}[H]
    \begin{center}
      \includegraphics[scale=0.5]{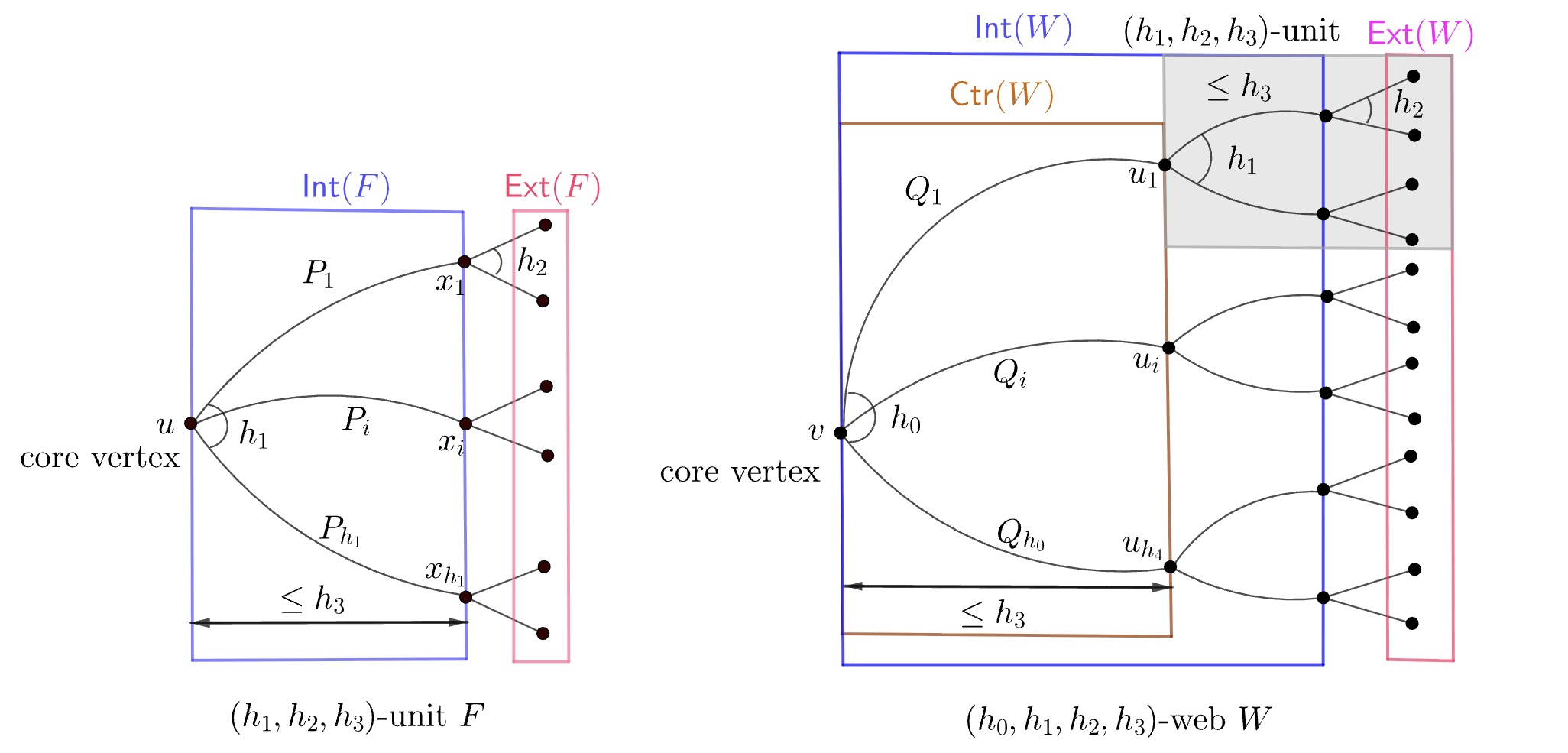}\\
      \caption{A unit and a web.}
      \label{webdraw}
    \end{center}
  \end{figure}

  \begin{definition}[web; see~\cite{Liu-balanced-sub}]\label{defn:web}
    For $h_0,h_1,h_2,h_3\in \mathbb{N}$, a graph $W$ is an \emph{$(h_0,h_1,h_2,h_3)$-web} if it contains distinct vertices $v$ (the \emph{core} vertex of $W$), $u_1,\ldots,u_{h_0}$, and $W=\bigcup_{i\in[h_0]}(Q_i\cup F_{i})$, where
    \begin{itemize}
      \item $\mathcal{Q}=\{Q_i:i\in[h_0]\}$ is a collection of pairwise internally vertex-disjoint paths such that each $Q_i$ is a $v,u_i$-path of length exactly $t$ for some $t\leq h_3$, and
      \item $\mathcal{F}=\{F_i:i\in[h_0]\}$ is a collection of vertex-disjoint $(h_1,h_2,h_3)$-units such that $F_{i}$ has core vertex $u_i$ and $\bigcup_{i\in[h_0]}(V(F_{i})\backslash \{u_i\})$ is vertex-disjoint from $V(\mathcal{Q})$.
    \end{itemize}
  \end{definition}

  We call each $Q_i$ a \emph{branch} and the branches inside each unit $F_i$ the \emph{second-level branches} of $W$. We define the \emph{exterior}, \emph{interior}, and \emph{center} of $W$ by $\mathsf{Ext}(W):=\bigcup_{i\in[h_0]}\mathsf{Ext}(F_i)$, $\mathsf{Int}(W):=V(W)\backslash \mathsf{Ext}(W)$, and $\mathsf{Ctr}(W):=V(\mathcal{Q})$, respectively.

To adjust path lengths as desired, we use the notion of an \emph{adjuster}, introduced by Liu and Montgomery in~\cite{Liu-Mon-JAMS}.

  \begin{definition}[\cite{Liu-Mon-JAMS}]\label{def:adj}
    An $(\ell,k)$-\emph{adjuster} $\mathcal{A}=(v_1, v_2, F_1, F_2, A,\mathcal{P})$ in a graph $G$ consists of
    two vertices $v_1,v_2$ of $G$, two subgraphs $F_1,F_2$ of $G$, a vertex set $A\subseteq V(G)$, and a collection of $v_1,v_2$-paths $\mathcal{P}$ in $G$ satisfying the following properties.
    \stepcounter{propcounter}
    \begin{enumerate}[label = ({\bfseries \Alph{propcounter}\arabic{enumi}})]
        \rm
      \item\label{adj-oddadj-1} $F_1$ and $F_2$ are either both stars or both units, and $F_i$ has core $v_i$ for each $i\in[2]$.
        \rm
      \item\label{adj-oddadj-2} $\mathsf{Int}(F_1)\cap \mathsf{Int}(F_2)=\varnothing$.\rm
      \item\label{adj-oddadj-3} $|A|\le 2\ell$ and $A\subseteq V(G)\backslash\bigcup_{i\in [2]}V(F_i)$.
        \rm
      \item\label{adj-oddadj-4} $\mathcal{P}$ consists of $k+1$ $v_1,v_2$-paths in $G[A\cup \{v_1,v_2\}]$ of lengths $\ell', \ell'+2, \dots, \ell'+2k$ for some $\ell'\leq \ell$.
    \end{enumerate}

    Furthermore, we define the \emph{length} and \emph{perimeter} of $\mathcal{A}$ by $\ell(\mathcal{A}):=\ell'+2k$ and $p(\mathcal{A}):=A\cup \mathsf{Int}(F_1)\cup \mathsf{Int}(F_2)$, respectively.
  \end{definition}

  Definition~\ref{def:adj} changes path lengths only by $2$. Since passing to a bipartite subgraph would lose half the minimum degree, we also use a short odd cycle to change parity. This motivates the following odd simple adjuster.

  If, in Definition~\ref{def:adj}, we replace only~\ref{adj-oddadj-4}, while retaining~\ref{adj-oddadj-1}--\ref{adj-oddadj-3}, with
  \begin{itemize}
    \item $\mathcal{P}$ consists of two $v_1,v_2$-paths in $G[A\cup \{v_1,v_2\}]$ of lengths $\ell', \ell'+1$ for some $\ell'\leq \ell$,
  \end{itemize}
  then $\mathcal{O}=(v_1,v_2,F_1,F_2,A,\mathcal{P})$ is called an $(\ell,1)$-\emph{odd simple adjuster}. We define
  \[
    \ell(\mathcal{O}):=\ell'+1,\qquad
    p(\mathcal{O}):=A\cup\mathsf{Int}(F_1)\cup\mathsf{Int}(F_2).
  \]

  \subsection[Many large-degree vertices]{Many large-degree vertices: $|L_G|\geq \tfrac{11}{50}d$}\label{subsec:52} We prove the following lemma. Recall that $m$ is the smallest even integer larger than $\log^4\tfrac{n}{d}$.

  \begin{lemma}\label{lem:lar-dense}
    Let $C\in\mathbb N$ and
    $0<\tfrac{1}{n},\tfrac{1}{d}\ll \tfrac{1}{K}\ll \eps_1,\eps_2\ll 1/C<1/5$, with $\tfrac{1}{4}\log^{1000}n\leq d\leq \tfrac{n}{K}$ and $Cd\in 2\mathbb{N}$.
    Suppose $G=(V,E)$ is an $n$-vertex $(\eps_1,\eps_2d)$-expander with $\delta(G)\geq \tfrac{d}{2}$. If $|L_G|\ge  \tfrac{11}{50}d$, 
    then $G$ contains a cycle of length $\ell$ for any even integer $\ell$ with
        $10m^3-2m^2+(250C+10)m\leq \ell\leq Cd+10m^3-50Cm$.
  \end{lemma}

  After fixing $G$, we can, since $|L_G|\ge\tfrac{11}{50}d$, choose two disjoint subsets $L_G^0$ and $L_G^1$ of $L_G$ such that $|L_G^0|=\tfrac{d}{50}$ and $|L_G^1|=\tfrac{d}{5}$.
  In the following part of this subsection, we fix our choices of $L_G^0$ and $L_G^1$ when $G$ is given.
  The proof sketch has two steps. First, we use $L_G^0$ to construct two adjusters: one for controlling parity (when needed) and the other for controlling even length changes (see Section~\ref{subsec:adj}). We then use vertices in $L_G^1$ to find a long path greedily while avoiding these adjusters. After truncating the path to the desired length and linking it to the adjusters, we obtain $\tfrac{Cd}{2}$ consecutive even cycle lengths (see Section~\ref{subsec:longp}).

  \subsubsection{Constructing adjusters}\label{subsec:adj}

  We now use two vertices in $L_G^0$ to construct a $(4m,1)$-\emph{odd simple adjuster}; see Lemma~\ref{lem:odd-sim-adj}.
  \begin{lemma}\label{lem:odd-sim-adj}
    Suppose $0<\tfrac{1}{n},\tfrac{1}{d}\ll \tfrac{1}{K}\ll \eps_1,\eps_2<\tfrac{1}{20}$ with $\tfrac{1}{4}\log^{1000}n\leq d\leq \tfrac{n}{K}$. Let $G=(V,E)$ be an $n$-vertex $(\eps_1,\eps_2d)$-expander with $\delta(G)\geq d/2$. If $G$ contains at least one odd cycle and $|L_G|\ge \tfrac{11}{50}d$, then there exists a $(4m,1)$-odd simple adjuster $\mathcal{O}=(v_1,v_2,S_1,S_2,O,\mathcal{P})$ in $G$ such that $v_1,v_2\in L_G^0$, where $S_i$ is a $dm^{28}$-star for each $i\in[2]$.
  \end{lemma}
  \begin{proof}
    By Lemma~\ref{lem:oddcyclelengthinexpander}, there is an odd cycle $C$ with $|V(C)|\leq 2m+5$.
    Let $u_1,u_2$ be two vertices of $C$ at distance $\tfrac{|V(C)|-1}{2}$ on $C$.
    Denote the $u_1,u_2$-path of odd length in $C$ by $P_C^1$ and the other path by $P_C^2$.
    Since $\delta(G)\geq \tfrac{d}{2}$ and $d\geq m^{200}$, we can take two vertex-disjoint sets $N'(u_1)\subseteq N(u_1)\backslash V(C)$ and $N'(u_2)\subseteq N(u_2)\backslash V(C)$, each of size at least $\tfrac{d}{5}$. Choose $v_1,v_2\in L_G^0\backslash V(C)$. We can then take two vertex-disjoint sets $N'(v_1)\subseteq N(v_1)\backslash (V(C)\cup N'(u_1)\cup N'(u_2))$ and $N'(v_2)\subseteq N(v_2)\backslash (V(C)\cup N'(u_1)\cup N'(u_2))$, each of size at least $\tfrac{dm^{31}}{3}$.
    Applying Lemma~\ref{distance} with $(X_1,X_2,W)=(N'(u_1),N'(v_1),V(C))$, there is a path $Q_1$ of length at most $m$ between $N'(u_1)$ and $N'(v_1)$, and $Q_1$ can be extended to a $u_1,v_1$-path $Q_1'$ of length at most $m+2$. Similarly, applying Lemma~\ref{distance} with $(X_1,X_2,W)=(N'(u_2)\backslash V(Q_1'),N'(v_2)\backslash V(Q_1'),V(C)\cup V(Q_1'))$, there is a $u_2,v_2$-path $Q_2'$ of length at most $m+2$. Then $P_1:=Q_1'P_C^1Q_2'$ and $P_2:=Q_1'P_C^2Q_2'$ are two paths whose lengths differ by $1$. Let $\mathcal{P}:=\{P_1,P_2\}$.

    Let $O=(V(Q_1')\cup V(Q_2')\cup V(C))\backslash\{v_1,v_2\}$. Hence $|O|\le 2m+5+2(m+2)\le 2\cdot 3m$. Since $v_1,v_2$ are in $L_G$, both have degree at least $dm^{31}$. By the pigeonhole principle, there exist two vertex-disjoint $dm^{28}$-stars $S_1$ and $S_2$, with centers $v_1$ and $v_2$, respectively, avoiding $O$. The shorter of $P_1$ and $P_2$ has length at most $2(m+2)+m+2=3m+6\leq 4m$.
    Therefore, $(v_1,v_2,S_1,S_2,O,\mathcal{P})$ is the required $(4m,1)$-odd simple adjuster in $G$.
  \end{proof}
  
The next lemma, Lemma~\ref{lem:lar-adjuster}, uses the vertices in $L_G^0$ to construct adjusters that control even length changes. Its proof follows a method similar to that of~\cite{Liu-balanced-sub}. For completeness, we include a proof in Appendix~\ref{proofadj-odd} of the arXiv version.

  \begin{lemma}\label{lem:lar-adjuster}
    Suppose $0<\tfrac{1}{n},\tfrac{1}{d}\ll \tfrac{1}{K}\ll \eps_1,\eps_2<\tfrac{1}{5}$ with $\tfrac{1}{4}\log^{1000}n\leq d\leq \tfrac{n}{K}$. Let $G=(V,E)$ be an $n$-vertex $(\eps_1,\eps_2d)$-expander with $\delta(G)\geq \tfrac{d}{2}$. Let $W\subseteq V(G)$ be a set of size at most $m^{10}$. If $|L_G|\geq \tfrac{11}{50}d$, then there exists a $(100m^3,m^2)$-adjuster $\mathcal{A}=(v_1,v_2,S_1,S_2,A,\mathcal{P})$ in $G-L_G^1-W$ such that $10m^3\leq \ell(\mathcal{A})\leq 10m^3+10m$, where $v_1,v_2\in L_G^0$ and $S_j$ is a $dm^{28}$-star for each $j\in[2]$.
  \end{lemma}

  \subsubsection{Finding a long path}\label{subsec:longp}
  Since $|L_G|$ may be smaller than $d$, a path through all of $L_G$ can still be too short. Lemma~\ref{lem:lar-unit} supplies the vertex-disjoint units used in Lemma~\ref{cla:longpath} to build a path of length in $[Cd,5Cdm]$. We omit its proof, which is almost identical to that of Lemma~4.2 in~\cite{Liu-balanced-sub}.

  \begin{lemma}[\cite{Liu-balanced-sub}]\label{lem:lar-unit}
    For each $0<\varepsilon_1,\varepsilon_2<1$, the following holds for all sufficiently large $K=K(\eps_1,\eps_2)$. Let $G$ be an $n$-vertex $(\varepsilon_1,\varepsilon_2d)$-expander with $\delta(G)\geq d$, $n\geq Kd$ and $d\geq m^{200}$. Then, given any $W\subseteq V(G)$ with $|W|\leq dm^{50}$, the graph
    $G-W$ contains a $(4\sqrt{d}m^6,\sqrt{d}m^{23},10m)$-unit.
  \end{lemma}

  \begin{lemma}\label{cla:longpath}
    Let $C\geq 1$ be an integer and $0<\tfrac{1}{n},\tfrac{1}{d}\ll \tfrac{1}{K}\ll \eps_1,\eps_2<\tfrac{1}{5}$ with $\tfrac{1}{4}\log^{1000}n\leq d\leq \tfrac{n}{K}$. Suppose $G=(V,E)$ is an $n$-vertex $(\eps_1,\eps_2d)$-expander with $\delta(G)\geq \tfrac{d}{2}$. If $|L_G|\ge \tfrac{11}{50}d$, then, for any vertex set $X\subseteq V(G)$ of size at most $\sqrt{d}m$, the graph $G-X$ contains a path $Q$ through a set $V_0\subseteq L_G^1$ with $|V_0|=\tfrac{d}{10}+1$ such that the endvertices of $Q$ lie in $V_0$ and $Cd\leq\ell(Q)\leq5Cdm$. Furthermore, every subpath between two consecutive vertices of $V(Q)\cap V_0$ has length at most $49Cm$.
  \end{lemma}
  \begin{proof}
    Set $t=\tfrac{d}{10}+1$. As $|L_G^1|-|X|\ge \tfrac{d}{5}-\sqrt{d}m\ge \tfrac{d}{10}+1$,
    we can directly choose $V_0\subseteq L_G^1-X$ with $|V_0|= t$.
    Denote the vertices in $V_0$ by $v_1, \ldots, v_t$. 
    Let $\mathcal{P}$ be a maximal family of internally vertex-disjoint paths $P_i$ satisfying the following conditions.
    \stepcounter{propcounter}
    \begin{enumerate}[label = ({\bfseries \Alph{propcounter}\arabic{enumi}})]
        \rm
      \item\label{longpath1} Every path $P_i\in\mathcal{P}$ is a $v_i,v_{i+1}$-path of length in $[10C+3,42Cm+m+4C+3]$.
        \rm
      \item\label{longpath2} No $P_i$ contains a vertex of $V_0\cup X$ as an internal vertex.
    \end{enumerate}
    We shall prove $|\mathcal{P}|=t-1$. Otherwise, there exists $j\in[t-1]$ for which no such path $P_j$ exists. Note that $|N(v_j)|,|N(v_{j+1})|\geq dm^{31}$ and $|V(\mathcal{P})|\leq (t-1)(42Cm+m+4C+2)+t \leq 15Cdm$. %\wx{Why $|N(v_j)|=|N(v_{j+1})|$?}
    We begin with the following claims.
     \begin{figure}[H]
    \begin{center}
      \includegraphics[scale=9]{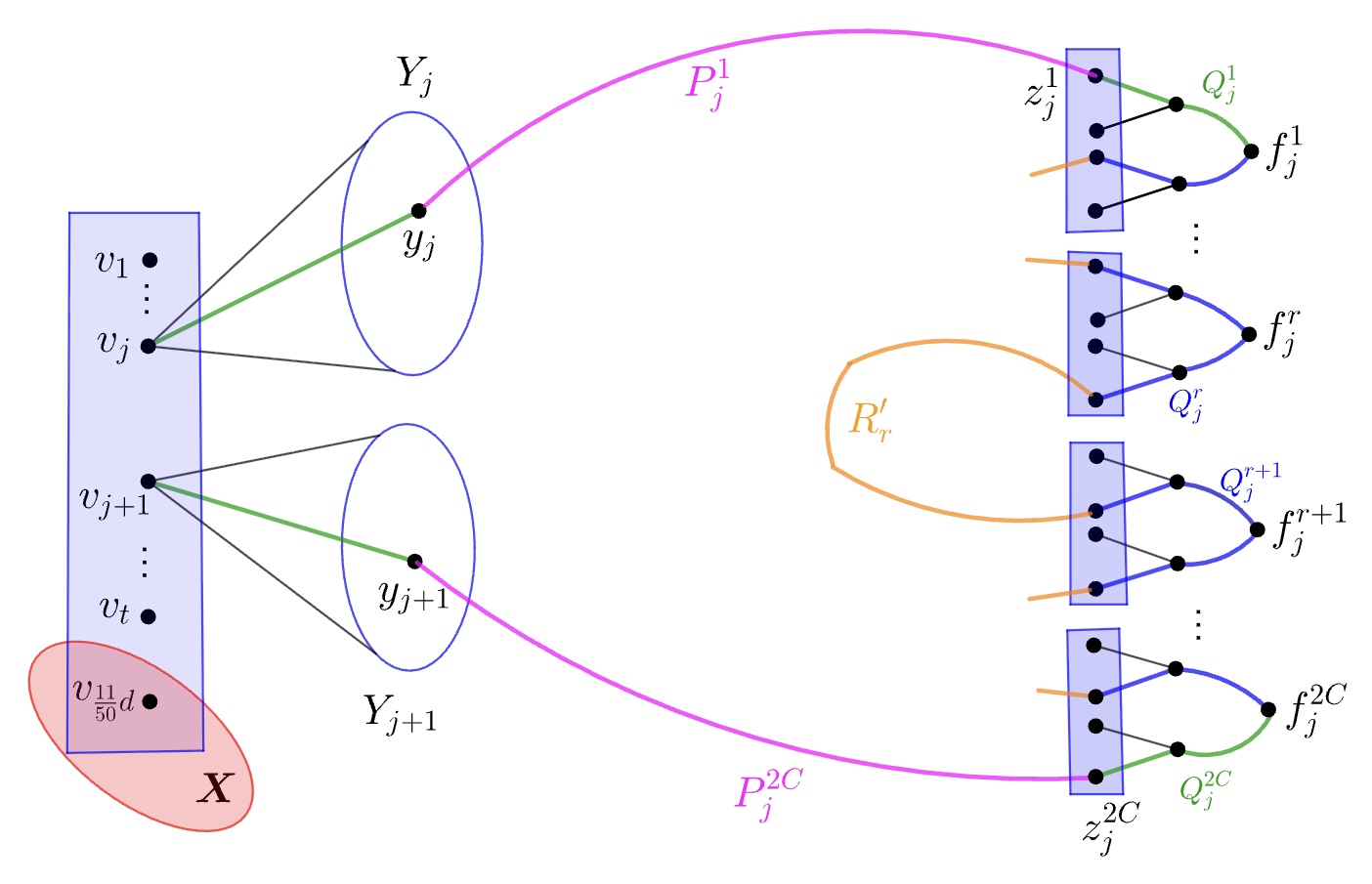}\\
      \caption{Process of finding a long path}
      \label{link}
    \end{center}
  \end{figure}
    \begin{claim}\label{cla:many-unit-lar}
      There are $2C$ vertex-disjoint $(2\sqrt{d}m^6,\tfrac{\sqrt{d}}{2}m^{23},10m)$-units $\mathcal{F}:=\{F_j^1,\ldots,F_j^{2C}\}$ in $G-(V_0\cup X\cup V(\mathcal{P}))$, with respective core vertices $f_j^1,\ldots,f_j^{2C}$.
    \end{claim}
    \begin{poc}
      Suppose that a maximal such family $\mathcal F$ has $|\mathcal{F}|<2C$. Then
      \begin{equation*}
        |V(\mathcal{F})|\leq 2C\cdot 2\sqrt{d}m^6\left(10m+\tfrac{\sqrt{d}m^{23}}{2}\right)\leq 4Cdm^{29}.
      \end{equation*}
      Note that $|V_0\cup X\cup V(\mathcal{P})\cup V(\mathcal{F})|\leq t+\sqrt{d}m+15Cdm+4Cdm^{29}\leq dm^{50}/2$. 
      Applying Lemma~\ref{lem:lar-unit} with $(W,d,\eps_2)=(V_0\cup X\cup V(\mathcal{P})\cup V(\mathcal{F}),\tfrac{d}{2},2\eps_2)$,
      there exists a $(2\sqrt{d}m^6,\tfrac{\sqrt{d}}{2}m^{23},10m)$-unit in $G-(V_0\cup X\cup V(\mathcal{P})\cup V(\mathcal{F}))$, contradicting the maximality of $\mathcal{F}$ as claimed.
    \end{poc}
    Let $\mathcal{R}$ be a maximal family of internally vertex-disjoint paths $R_k$, for $1\leq k\leq 2C-1$, satisfying the following conditions.
    \stepcounter{propcounter}
    \begin{enumerate}[label = ({\bfseries \Alph{propcounter}\arabic{enumi}})]
        \rm
      \item\label{pair-path1} $R_k$ is a $f_j^{k},f_j^{k+1}$-path of length in $[5,21m+2]$ that avoids $L_G^0\cup L_G^1\cup X\cup V(\mathcal{P})$ internally.
        \rm
      \item\label{pair-path2} No $R_k$ contains a vertex of $\bigcup_{s\in [2C]\setminus \{k,k+1\}}\mathsf{Int}(F_j^s)$ as an internal vertex.
    \end{enumerate}
    \begin{claim}\label{cla:units-link}
      $|\mathcal{R}|=2C-1$.
    \end{claim}
    \begin{poc}
      Suppose $|\mathcal{R}|<2C-1$. Then there exists some $r$ for which no $f_j^{r},f_j^{r+1}$-path satisfying the conditions exists. Note that
      \begin{equation*}
        |V(\mathcal{R})|\leq (2C-1)(21m+3)\leq 42Cm.
      \end{equation*}
      A star is called \emph{bad} if at least one leaf lies in $V(\mathcal R)$, and a branch is called \emph{bad} if its pendant star is bad.
      The paths in $\mathcal R$ use at most $42Cm$ vertices of each of $\mathsf{Ext}(F_j^r)$ and $\mathsf{Ext}(F_j^{r+1})$, and hence destroy at most $42Cm$ branches in each unit. Let $\overline{F_j^{r}}$ and $\overline{F_j^{r+1}}$ be the units obtained from $F_j^{r}$ and $F_j^{r+1}$, respectively, by removing the bad branches. Then
      \begin{equation*}
        |\mathsf{Ext}(\overline{F_j^{r}})|, |\mathsf{Ext}(\overline{F_j^{r+1}})|\geq(2\sqrt{d}m^6-42Cm)\cdot \tfrac{\sqrt{d}}{2}m^{22} \geq \sqrt{d}m^6\cdot \tfrac{\sqrt{d}}{2}m^{22}\geq \tfrac{d}{2}m^{28},
      \end{equation*}
      and we can find disjoint subsets $S\subseteq \mathsf{Ext}(\overline{F_j^{r}})$ and $S'\subseteq \mathsf{Ext}(\overline{F_j^{r+1}})$, each of size at least $\tfrac{d}{4}m^{28}$.
      Let $W_1=V(\mathcal{R})\cup L_G^0\cup L_G^1\cup X\cup V(\mathcal{P})\cup (\cup_{k=1}^{2C}\mathsf{Int}(F_j^k))$. Then
      \begin{equation*}
        |W_1|\leq 42Cm+\tfrac{11}{50}d+\sqrt{d}m+15Cdm+2C\cdot 2\sqrt{d}m^6\cdot 10m\leq 16Cdm.
      \end{equation*}
      Applying Lemma~\ref{distance} with $(X_1,X_2,W)=(S,S',W_1)$, we obtain distinct vertices $z_j^r\in \mathsf{Ext}(\overline{F_j^{r}})$ and $z_{j}^{r+1}\in \mathsf{Ext}(\overline{F_j^{r+1}})$ and a $z_j^r,z_j^{r+1}$-path $R_r'$ of length at most $m$ that avoids $W_1$. Let $Q_j^r$ be the $z_j^r,f_j^r$-path in $\overline{F_j^{r}}$, and let $Q_j^{r+1}$ be the $z_j^{r+1},f_j^{r+1}$-path in $\overline{F_j^{r+1}}$. Finally, set $R_r:=Q_j^rR_r'Q_j^{r+1}$. Since the lengths of $Q_j^r$ and $Q_j^{r+1}$ lie in $[2,10m+1]$ and that of $R_r'$ lies in $[1,m]$, the length of $R_r$ lies in $[5,21m+2]$, contradicting the maximality of $\mathcal{R}$.
    \end{poc}
    Let $W_0:=V(F_j^1)\cup V(F_j^{2C})\cup X\cup L_G^0\cup L_G^1\cup V(\mathcal{P})\cup V(\mathcal{R})$. Then
    \begin{equation*}
      |W_0|\leq 4dm^{29}+\sqrt{d}m+\tfrac{11}{50}d+15Cdm+42Cm\leq 5dm^{29}.
    \end{equation*}
    For each $k\in\{j,j+1\}$, there is a set $X_k\subseteq N(v_k)\setminus W_0$ with $|X_k|\geq\tfrac{dm^{31}}{2}$. %\wx{I see $v_k\in L_G^1$, so $|N(v_k)|\ge dm^{31}$. I do not know why we take it half.}
    Choose vertex-disjoint sets $Y_j\subseteq X_j$ and $Y_{j+1}\subseteq X_{j+1}$, each of size at least $\tfrac{dm^{31}}{4}$. Let $W':=V_0\cup V(\mathcal{P})\cup V(\mathcal{R})\cup X\cup \mathsf{Int}(F_j^1)\cup \mathsf{Int}(F_j^{2C})$. Then
    \begin{equation*}
      |W'|\leq t+15Cdm+42Cm+\sqrt{d}m+2\cdot 2\sqrt{d}m^6\cdot 10m\leq 16Cdm.
    \end{equation*}
    The paths in $\mathcal R$ use at most $42Cm$ vertices of $\mathsf{Ext}(F_j^1)$, and hence destroy at most $42Cm$ branches. Let $\widetilde{F_j^1}$ be the unit obtained from $F_j^1$ by removing these bad branches. Then
    \begin{equation*}
      |\mathsf{Ext}(\widetilde{F_j^1})|\geq(2\sqrt{d}m^6-42Cm)\cdot \tfrac{\sqrt{d}}{2}m^{22} \geq \sqrt{d}m^6\cdot \tfrac{\sqrt{d}}{2}m^{22}\geq \tfrac{d}{2}m^{28}.
    \end{equation*}
    Applying Lemma~\ref{distance} with $(X_1,X_2,W)=(Y_j,\mathsf{Ext}(\widetilde{F_j^1}),W')$, we obtain vertices $y_j\in Y_j$ and $z_j^1\in\mathsf{Ext}(\widetilde{F_j^1})$ and a $y_j,z_j^1$-path $P_j^1$ of length at most $m$. Let $Q_j^1$ be the $z_j^1,f_j^1$-path in $F_j^1$. The path $P_j^1$ uses at most $m$ vertices of $\mathsf{Ext}(F_j^{2C})$, and hence destroys at most $42Cm+m$ branches. Let $\widetilde{F_j^{2C}}$ be the unit obtained from $F_j^{2C}$ by removing these bad branches. Then
    \begin{equation*}
      |\mathsf{Ext}(\widetilde{F_j^{2C}})|\geq(2\sqrt{d}m^6-42Cm-m)\cdot \tfrac{\sqrt{d}}{2}m^{22} \geq \sqrt{d}m^6\cdot \tfrac{\sqrt{d}}{2}m^{22}\geq \tfrac{d}{2}m^{28}.
    \end{equation*}
    Let $W'':=V(P_j^1)\cup W'$. Then $|W''|\leq m+16Cdm\leq 17Cdm$. Applying Lemma~\ref{distance} again with $(X_1,X_2,W)=(Y_{j+1},\mathsf{Ext}(\widetilde{F_j^{2C}}),W'')$, we obtain vertices $y_{j+1}\in Y_{j+1}$ and $z_j^{2C}\in \mathsf{Ext}(\widetilde{F_j^{2C}})$ and a $y_{j+1},z_j^{2C}$-path $P_j^{2C}$ of length at most $m$. Let $Q_j^{2C}$ be the $z_j^{2C},f_j^{2C}$-path in $F_j^{2C}$. Let
    
    $P_j:=v_jy_jP_j^1Q_j^1R_1R_2\cdots R_{2C-1}Q_j^{2C}P_j^{2C}y_{j+1}v_{j+1}$.
    Then
    \begin{equation}\label{eq:pair-length}
      10C+3\leq \ell(P_j)\leq 42Cm+m+4C+2.
    \end{equation}
    Thus, $P_j$ satisfies~\ref{longpath1}--\ref{longpath2}, contradicting the maximality of $\mathcal{P}$. Hence, there is a path $Q$ through a subset of $\{v_1,\ldots,v_t\}$ of length in $[Cd+\tfrac{3d}{10},\tfrac{21}{5}Cdm+\tfrac{d}{10}m+\tfrac{2}{5}Cd+\tfrac{d}{5}]\subseteq [Cd,5Cdm]$, as claimed. By the preceding arguments and~\eqref{eq:pair-length}, every subpath between two consecutive vertices of $V(Q)\cap V_0$ has length at most $42Cm+m+4C+2\leq 49Cm$.
  \end{proof}

  \begin{proof}[Proof of Lemma~\ref{lem:lar-dense}]
    We divide the proof into two cases according to whether $G$ is bipartite.

    \textbf{Case 1: Suppose that $G$ contains an odd cycle}. By Lemma~\ref{lem:odd-sim-adj}, there exists a $(4m,1)$-odd simple adjuster $\mathcal{O}=(v_1,v_2,S_1,S_2,O,\{P_1,P_2\})$ such that $v_1,v_2\in L_G^0$ and $S_i$ is a $dm^{28}$-star for each $i\in[2]$. Note that $|p(\mathcal{O})|\leq 6m\leq m^{10}$. Applying Lemma~\ref{lem:lar-adjuster} with $W:=p(\mathcal{O})$, there exists a $(100m^3,m^2)$-adjuster $\mathcal{A}=(v_3,v_4,S_3,S_4,A,\mathcal{P})$ such that $v_3,v_4\in L_G^0$ and $S_i$ is a $dm^{28}$-star for each $i\in\{3,4\}$.
    Note that $|p(\mathcal{A})|\leq 210m^3$ and $|p(\mathcal{O})|+|p(\mathcal{A})|\leq 240m^3$. Applying Lemma~\ref{distance} with $(X_1,X_2,W)=(S_2,S_3,p(\mathcal{O})\cup p(\mathcal{A}))$, we obtain a path $Q_0$ of length at most $m$ between $S_2$ and $S_3$. Extend it to a $v_2,v_3$-path $Q_0'$ of length at most $m+2$. Since $|p(\mathcal{O})|+|p(\mathcal{A})|+|V(Q_0')|\leq 300m^3\leq \sqrt{d}m$, Lemma~\ref{cla:longpath} gives a path $Q$ in $G-p(\mathcal{O})-p(\mathcal{A})-V(Q_0')$ of length in $[Cd,5Cdm]$ whose endvertices lie in $V_0\subseteq L_G^1$. As the distance along $Q$ between consecutive vertices of $V(Q)\cap V_0$ is at most $49Cm<50Cm$, set
    \[
      s:=\left\lfloor \ell(Q)/(50Cm)\right\rfloor-3
      \qquad\text{and}\qquad
      L_i:=\ell(Q)-50Cmi \quad (0\le i\le s).
    \]
    Then $L_i\in[100Cm,\ell(Q)]$, and we can take a subpath $\widetilde{Q}_i$ of $Q$ whose length lies in $[L_i-50Cm,L_i]$ and whose endvertices $v^i_1,v^i_2$ lie in $V_0$. Recall that $\ell(Q)\in[Cd,5Cdm]$.
    Let $W'=p(\mathcal{A})\cup p(\mathcal{O})\cup V(\widetilde{Q}_i)\cup V(Q_0')$. Then
    \begin{equation*}
      |W'|\leq 200m^5+6m+ 5Cdm  \leq 6Cdm.
    \end{equation*}
    For each $j\in[2]$, let $N'(v^i_j):=N(v^i_j)\setminus (\bigcup_{k\in[4]}V(S_k)\cup W')$. Then
    \begin{equation*}
      |N'(v^i_{j})|\geq dm^{31}-4dm^{28}-6Cdm\geq \tfrac{dm^{31}}{2}.
    \end{equation*}
    Choose vertex-disjoint sets $N''(v^i_1)\subseteq N'(v^i_1)$ and $N''(v^i_2)\subseteq N'(v^i_2)$, each of size at least $\tfrac{dm^{31}}{4}$. Applying Lemma~\ref{distance} with $(X_1,X_2,W)=(N''(v^i_1),N(v_1),W')$, we obtain a path $Q^i_1$ of length at most $m$ from $N''(v^i_1)$ to $N(v_1)$. Extend it to a $v^i_1,v_1$-path $\widetilde Q^i_1$ of length at most $m+2$. Let $W'':=W'\cup V(\widetilde Q^i_1)$; then $|W''|\leq 6Cdm+(m+2)\leq 7Cdm$. Applying Lemma~\ref{distance} again with $(X_1,X_2,W)=(N''(v^i_2),N(v_4),W'')$, we obtain a path $Q^i_2$ of length at most $m$ from $N''(v^i_2)$ to $N(v_4)$. Extend it to a $v^i_2,v_4$-path $\widetilde Q^i_2$ of length at most $m+2$. Hence, $\widetilde Q^i_1\widetilde Q_i\widetilde Q^i_2$ is a $v_1,v_4$-path of length $\ell$ satisfying
    \begin{equation*}
      \ell \in [L_i-50Cm, L_i+30Cm].
    \end{equation*}
    This path avoids $O\cup V(Q_0')\cup A$. Let $\ell_0:=\ell(Q_0')\le m+2$. Since the two $v_1,v_2$-paths belonging to $\mathcal O$ have lengths at most $3m+7$ and differ in length by $1$, one of the paths $Q_0'P_j\widetilde Q^i_1\widetilde Q_i\widetilde Q^i_2$, $j\in[2]$, has the same parity as $\ell(\mathcal A)$. Denote its length by $\ell_i'$; then $\ell_i'\in[0,m+2]+[0,3m+7]+[L_i-50Cm,L_i+30Cm]\subseteq[L_i-50Cm,L_i+50Cm]$.
    By the definition of $\mathcal A$, there are $m^2+1$ paths with lengths
    $\ell(\mathcal{A}),\ell(\mathcal{A})-2,\ldots,\ell(\mathcal{A})-2m^2$.
    Thus there is a cycle of every length in the following interval:
    \begin{equation*}
      I_i:=\ell_i'+\{0,1\}+\{\ell(\mathcal{A})-2j:0\le j\le m^2\}
      =[\ell_i'+\ell(\mathcal{A})-2m^2,\ell_i'+\ell(\mathcal{A})+1].
    \end{equation*}
    Combining these intervals, there is a cycle of every length in
    $\bigcup_{i=0}^{s}I_i$. For every $j\in[s]$,
    \begin{equation*}
      |\ell_j'-\ell_{j-1}'|\le 150Cm<2m^2,
    \end{equation*}
    the consecutive intervals $I_{j-1}$ and $I_j$ overlap. Hence
    $\bigcup_{i=0}^{s}I_i$ is an interval of integers. Moreover,
    \begin{equation*}
      150Cm\leq L_s<200Cm,
    \end{equation*}
    and therefore $\ell_s'<250Cm$ and $\ell_0'\geq \ell(Q)-50Cm\geq Cd-50Cm$. Hence the lower endpoint of $I_s$ is at most $10m^3-2m^2+(250C+10)m$, while the upper endpoint of $I_0$ is at least $Cd+10m^3-50Cm$. In addition,
    $\ell_s'+\ell(\mathcal{A})\geq 10m^3+100Cm>10m^3-2m^2+(250C+10)m$.
    Consequently, $\bigcup_{i=0}^{s}I_i$ contains every integer in
    $[10m^3-2m^2+(250C+10)m,Cd+10m^3-50Cm]$.

    \textbf{Case 2: Suppose that $G$ contains no odd cycle}; equivalently, $G$ is bipartite. By arguments analogous to those above, but without the odd simple adjuster, we obtain the following three structures:
    \begin{itemize}
      \item[$(1)$] A $(100m^3,m^2)$-adjuster $\mathcal{A}=(u_1,u_2,S_1,S_2,A,\mathcal{P})$ in $G$.
    \item[$(2)$] A path $Q$ in $G-p(\mathcal{A})$ of length in $[Cd,5Cdm]$ whose endvertices lie in $V_0\subseteq L_G^1$.
      \item[$(3)$] For every integer $i$ with $0\le i\le\lfloor \ell(Q)/(50Cm)\rfloor-3$, setting $L_i:=\ell(Q)-50Cmi$, a subpath $\widetilde Q_i$ of length in $[L_i-50Cm,L_i]$ whose endvertices $v^i_1$ and $v^i_2$ lie in $V_0\cap V(Q)$.
    \end{itemize}
    By Lemma~\ref{distance}, we can find a $v^i_1,u_1$-path $\widetilde Q^i_3$ and a $v^i_2,u_2$-path $\widetilde Q^i_4$, each of length at most $m+2$, such that they are vertex-disjoint. Hence, $\widetilde Q^i_3\widetilde Q_i\widetilde Q^i_4$ is a path of length $\ell_i''$ in $[L_i-50Cm,L_i+50Cm]$. By the definition of $\mathcal{A}$, the following set consists of attainable even cycle lengths:
    \begin{equation*}
      I_i':=\ell_i''+\{\ell(\mathcal{A})-2j:0\le j\le m^2\}.
    \end{equation*}
    All $u_1,u_2$-paths in a bipartite graph have the same parity, so every element of $I_i'$ is even. Furthermore, for every integer $i$ with
    $1\le i\le\lfloor\ell(Q)/(50Cm)\rfloor-3$, the difference $\ell_i''-\ell_{i-1}''$ is even and
    \begin{equation*}
      |\ell_i''-\ell_{i-1}^{''}|\leq 150Cm<2m^2.
    \end{equation*}
    By the same analysis, $\bigcup_{0\le i\le \lfloor \ell(Q)/(50Cm)\rfloor-3} I_i'$ contains every even integer in $[10m^3-2m^2+(250C+10)m,Cd+10m^3-50Cm]$.
  \end{proof}

  \subsection[Few large-degree vertices]{Few large-degree vertices: $|L_G|<\tfrac{11}{50}d$}\label{subsec:53}
  
    In this subsection, we deal with the case where there are few large-degree vertices. Instead of using stars centered in $L_G$, we construct many ``tree-like'' structures, such as units and webs, to play the role of stars. Recall that $m$ is the smallest even integer larger than $\log^4\tfrac{n}{d}$. Hence $d\geq m^{200}$ and $n\geq dm^{200}$.

    \subsubsection{Constructing units and webs}\label{web-unit}

    In the following lemmas, we first show that small $|L_G|$ yields many pairwise vertex-disjoint $\Omega(d)$-stars. We then use these stars as building blocks to create many pairwise vertex-disjoint units with large exteriors. Finally, we repeat the same procedure with units in place of stars to construct many internally vertex-disjoint webs. Since the proof ideas are similar to those in~\cite{Liu-balanced-sub,Yang1}, we include the proofs of Lemmas~\ref{claim:unit} and~\ref{lem:web} in Appendix~\ref{secapp:webs} of the arXiv version.

    \begin{lemma}\label{claim:unit}
      Let $0<\tfrac{1}{n},\tfrac{1}{d}\ll \tfrac{1}{K}\ll \eps_1,\eps_2<\tfrac{1}{5}$ with $\tfrac{1}{4}\log^{1000}n\leq d\leq \tfrac{n}{K}$. Suppose that $G=(V,E)$ is an $n$-vertex $(\eps_1,\eps_2d)$-expander with $\delta(G)\geq \tfrac{d}{2}$. If $|L_G|< \tfrac{11}{50}d$, then, for every set $X'$ of size at most $dm^{110}$, the graph $G-X'$ contains a collection of $m^{109}$ pairwise vertex-disjoint $(2m^{28},\tfrac{d}{100},m+2)$-units.
    \end{lemma}

    \begin{lemma}\label{lem:web}
      Let $0<\tfrac{1}{n},\tfrac{1}{d}\ll \tfrac{1}{K}\ll \eps_1,\eps_2\ll\tfrac{1}{C}<\tfrac{1}{5}$ with $\tfrac{1}{4}\log^{1000}n\leq d\leq \tfrac{n}{K}$. Suppose that $G=(V,E)$ is an $n$-vertex $(\eps_1,\eps_2d)$-expander with $\delta(G)\geq \tfrac{d}{2}$. If $|L_G|< \tfrac{11}{50}d$, then $G$ contains $200Cdm^{70}$ internally vertex-disjoint $(m^3,m^{28},\tfrac{d}{200},4m)$-webs.
    \end{lemma}

  \subsubsection{Constructing adjusters}
  We now use two units supplied by Lemma~\ref{claim:unit} to construct a $(6m,1)$-\emph{odd simple adjuster} that controls path parity.
  \begin{lemma}\label{lem:odd-sim-adj-less}
    Suppose $0<\tfrac{1}{n},\tfrac{1}{d}\ll \tfrac{1}{K}\ll \eps_1,\eps_2<\tfrac{1}{20}$ with $\tfrac{1}{4}\log^{1000}n\leq d\leq \tfrac{n}{K}$. Let $G=(V,E)$ be an $n$-vertex $(\eps_1,\eps_2d)$-expander with $\delta(G)\geq \tfrac{d}{2}$. If $G$ contains an odd cycle and $|L_G|<\tfrac{11}{50}d$, then there exists a $(6m,1)$-odd simple adjuster $\mathcal{O}=(v_1,v_2,F_1,F_2,O,\{P_1,P_2\})$ in $G$ such that $F_i$ is a $(m^{28},\tfrac{d}{100},m+2)$-unit for each $i\in[2]$.
  \end{lemma}
  \begin{proof}
    By Lemma~\ref{lem:oddcyclelengthinexpander}, there is an odd cycle $\Gamma$ with $|V(\Gamma)|\leq {2m+5}$. Let $u_1,u_2$ be two vertices of $\Gamma$ such that the two $u_1,u_2$-paths on $\Gamma$ have lengths differing by $1$. Denote these two paths by $P_\Gamma^1$ and $P_\Gamma^2$, and assume that $P_\Gamma^1$ has odd length. Since $\delta (G)\geq \tfrac{d}{2}$ and $d\geq m^{200}$, we can take two vertex-disjoint sets
    $N'(u_1)\subseteq N(u_1)\backslash V(\Gamma)$ and $N'(u_2)\subseteq N(u_2)\backslash V(\Gamma)$ of size $\tfrac{d}{5}$. Note that
    $|V(\Gamma)|+|N'(u_1)|+|N'(u_2)|\leq 2m+5+\tfrac{2d}{5}\leq dm^{40}$.
    Then by Lemma~\ref{claim:unit}, there are two vertex-disjoint $(2m^{28},\tfrac{d}{100},m+2)$-units $F'_1$ and $F'_2$, with core vertices $v_1,v_2$, respectively, in $G-(V(\Gamma)\cup N'(u_1)\cup N'(u_2))$. For each $i\in[2]$, we have $|\mathsf{Ext}(F'_i)|\geq \tfrac{dm^{28}}{50}$. Applying Lemma~\ref{distance} with $(X_1,X_2,W)=(N'(u_1),\mathsf{Ext}(F'_1),V(\Gamma)\cup\mathsf{Int}(F'_2))$, there is a path $Q_1$ of length at most $m$ between $N'(u_1)$ and $\mathsf{Ext}(F'_1)$. Extending $Q_1$ through the edge to $u_1$ and through the corresponding branch of $F'_1$ to the core $v_1$, and then removing loops if necessary, we obtain a $u_1,v_1$-path $Q'_1$ of length at most $2m+4$. Similarly, applying Lemma~\ref{distance} with
    $$
      (X_1,X_2,W)=(N'(u_2)\backslash V(Q'_1),\mathsf{Ext}(F'_2)\backslash V(Q'_1),V(\Gamma)\cup V(Q'_1)),
    $$
    there is a path $Q_2$ which can be extended, after removing loops if necessary, to a $u_2,v_2$-path $Q'_2$ of length at most $2m+4$, and $Q'_2$ is internally disjoint from $V(\Gamma)\cup V(Q'_1)$.

    Set $P_1:=Q'_1 P_\Gamma^1 Q'_2$ and $P_2:=Q'_1P_\Gamma^2Q'_2$. Then $P_1$ and $P_2$ are two $v_1,v_2$-paths whose lengths differ by $1$. Moreover, $\max\{|P_1|,|P_2|\}\leq 6m$. Let
    $O:=(V(P_1)\cup V(P_2))\backslash\{v_1,v_2\}$.
    Since $|V(\Gamma)|\leq 2m+5$ and $\ell(Q'_1),\ell(Q'_2)\leq 2m+4$, we have
    $$
      |O|\leq |V(\Gamma)|+|V(Q'_1)|+|V(Q'_2)|< 6m+20\leq 12m=2\cdot 6m.
    $$

    A branch in $F'_1$ or $F'_2$ is called \emph{bad} if some vertex of this branch other than its core $v_i$, or some vertex of its pendant star, lies in $V(P_1)\cup V(P_2)$. Since $\max\{|P_1|,|P_2|\}\leq 6m$, each unit has at least $2m^{28}-12m-2\geq m^{28}$ branches which are not bad. For each $i\in[2]$, take a subunit $F_i$ of $F'_i$ using only these non-bad branches. Then $F_i$ is a $(m^{28},\tfrac{d}{100},m+2)$-unit, and $O$ is disjoint from the interiors of $F_1$ and $F_2$. Therefore, $(v_1,v_2,F_1,F_2,O,\{P_1,P_2\})$ is a $(6m,1)$-odd simple adjuster in $G$.
  \end{proof}

  The analogous Lemma~\ref{lem:adjuster}, proved in Appendix~\ref{proofadj-odd} of the arXiv version, controls even path lengths when the ends are units.

  \begin{lemma}\label{lem:adjuster}
    Suppose $0<\tfrac{1}{n},\tfrac{1}{d}\ll \tfrac{1}{K}\ll \eps_1,\eps_2<\tfrac{1}{5}$ with $\tfrac{1}{4}\log^{1000}n\leq d\leq \tfrac{n}{K}$. Let $G$ be an $n$-vertex $(\eps_1,\eps_2d)$-expander with $\delta(G)\geq \tfrac{d}{2}$. Let $W\subseteq V(G)$ with $|W|\leq m^{50}$. Then there exists a $(100m^3,m^2)$-adjuster $\mathcal{A}=(v_1,v_2,F_1,F_2,A,\mathcal{P})$ in $G-W$ such that $60m^3\leq \ell(\mathcal{A})\leq 60m^3+60m$, where $F_i$ is a $(m^{28},\tfrac{d}{100},m+2)$-unit for each $i\in[2]$.
  \end{lemma}

  \subsubsection{Putting things together: proof of Lemma~\ref{lem:dense}}\label{subsubsec:533}
  \begin{proof}[Proof of Lemma~\ref{lem:dense}]
    If $|L_G|\ge \tfrac{11}{50}d$, Lemma~\ref{lem:dense} follows directly
    from Lemma~\ref{lem:lar-dense}. So we only consider the case where
    $|L_G|< \tfrac{11}{50}d$.
    By Lemma~\ref{lem:web}, we can find
    $200Cdm^{70}$
    internally vertex-disjoint $(m^3,m^{28},\tfrac{d}{200},4m)$-webs.
    Choose $200Cdm^{30}$
    of them, denoted by $W_1, \ldots, W_{200Cdm^{30}}$, where $W_i$ has core vertex $v_i$, and let $V_1':=
    \{v_1,\ldots, v_{200Cdm^{30}}\}$.
    Since the exterior of each $W_i$ is fairly large, the vertices in
    $V_1'$ 
     behave like the large stars in Section~\ref{subsec:52}, yielding an analogue of Lemma~\ref{cla:longpath} for long paths, as stated in the following claim.

    \begin{claim}\label{lem_long_path_web}
      For any subset $W\subseteq V(G)$ with $|W|\le 50m^{50}$ and any
      integer $L\in [200m,Cdm^{30}]$, we can always find a path $P=P(L)$ in $G-W$ satisfying the following
      properties.
      \stepcounter{propcounter}
      \begin{enumerate}[label = ({\bfseries \Alph{propcounter}\arabic{enumi}})]
          \rm
        \item\label{num-518(1)}
          There exist $i\ne j\in 
          [200Cdm^{30}]$ such that $P$ is a $v_iv_j$-path.

          \rm
        \item\label{num-518(2)}
          The interiors of $W_i$ and $W_j$ each intersect $W\cup V(P)$ in at
          most $m^2$ vertices.

          \rm
        \item\label{num-518(3)}
          The length of $P$ is in the range $[L-200m, L-100m]$.
      \end{enumerate}
    \end{claim}

    \begin{poc}
      Choose a longest path $P$ in $G-W$ satisfying~\ref{num-518(1)} and~\ref{num-518(2)} and having length at most $L-100m$. The existence of such a path follows from Lemma~\ref{distance}.
      We only need to prove that $\ell(P)$ is at least $L-200m$. Suppose not. Then $\ell(P)<L-200m$. Denote
      $W'=W\cup V(P)\cup \mathsf{Int}(W_i)\cup\mathsf{Int}(W_j)$.
      Then
      \begin{equation*}
        |W'|
        \le
        50m^{50}+Cdm^{30}+20m^{32}
        \le
        (C+1)dm^{30}.
      \end{equation*}
      By the pigeonhole principle, there are at least $200Cdm^{30}-(C+1)dm^{30}
          -\tfrac{2(C+1)dm^{30}}{m^2}
        \ge
        190Cdm^{30}$
      webs from
      $\{W_1,\ldots,W_{200Cdm^{30}}\}$
      whose cores do not lie in $W'$ and whose interiors intersect $W'$ in at most $m^2/2$ vertices.

      Choose two such webs, say $W_1$ and $W_2$. By our choice, they contain internally vertex-disjoint $(m^3-\tfrac{m^2}{2},m^{28},\tfrac{d}{200},4m)$-subwebs $W_1'$ and $W_2'$, respectively, whose interiors avoid $W'$ except at their cores. 
      Note that, for each $k\in[2]$, 
$        |\mathsf{Ext}(W_k')|\ge \tfrac{dm^{31}}{400}$.
      By~\ref{num-518(2)}, there are two internally vertex-disjoint
      $(m^3-m^2,m^{28},\tfrac{d}{200},4m)$-subwebs $W_i'$ and $W_j'$ of $W_i$ and $W_j$, respectively, whose interiors avoid $W\cup V(P)$ except at their cores. Moreover, $|\mathsf{Ext}(W_i')|,
      |\mathsf{Ext}(W_j')|
      \ge \tfrac{dm^{31}}{400}$. Let $W''=
      W\cup V(P)\cup V_1'
      \cup\mathsf{Int}(W_2')
      \cup\mathsf{Int}(W_j')$.
      Then
      \begin{equation*}
        |W''|
        \le
          50m^{50}+Cdm^{30}+200Cdm^{30}+20m^{32}
        \le
        210Cdm^{30}\le
        \tfrac14
        \rho\left(\tfrac{dm^{31}}{800}\right)
        \tfrac{dm^{31}}{800},
      \end{equation*}
      where the last inequality holds for sufficiently large $m$.
      Note that $|\mathsf{Ext}(W_1')\setminus W''|,|\mathsf{Ext}(W_2')\setminus W''|\geq \tfrac{dm^{31}}{800}$. By Lemma~\ref{distance}, there exists a path of length at most $2m$ between $\mathsf{Ext}(W_1')$ and $\mathsf{Ext}(W_i')$ avoiding $W''$. Extending it through the two webs gives a $v_1,v_i$-path $P_1$ of length at most $20m$ avoiding $W''$.

      By the same argument and the small-diameter lemma, there exists a path of length at most $2m$ between $\mathsf{Ext}(W_2')$ and $\mathsf{Ext}(W_j')$ whose internal vertices avoid $W\cup V(P)\cup V(P_1)$. Extending it through the webs gives a $v_2,v_j$-path $P_2$ of length at most $20m$ avoiding $W\cup V(P)\cup V(P_1)$. We now check that $P':=P_1PP_2$
      is a path satisfying both~\ref{num-518(1)} and~\ref{num-518(2)}, of length at most $L-100m$, but longer than $P$, contradicting the maximality of $P$.

      Property~\ref{num-518(1)} holds because $P'$ is a $v_1,v_2$-path. For~\ref{num-518(2)},
      \[
        |(W\cup V(P'))\cap\mathsf{Int}(W_1)|
        \le
        |(W\cup V(P))\cap\mathsf{Int}(W_1)|
        +|V(P_1)|+|V(P_2)|
        \le
        \tfrac{m^2}{2}+40m+2<m^2.
      \]
      The same argument applies to $W_2$. Finally,
      \[
        \ell(P') = \ell(P)+\ell(P_1)+\ell(P_2) \ge \ell(P)+2>\ell(P),
      \]
      while
      $\ell(P') < L-200m+40m < L-100m$.
      This contradicts the maximality of $P$.
    \end{poc}

    Next, we use Claim~\ref{lem_long_path_web} to prove the existence of cycles of consecutive even lengths from $\Theta(m^3)$ to $D:=Cdm^{27}$ in two cases, depending on whether $G$ contains an odd cycle.

    \textbf{Case 1: Suppose that $G$ contains an odd cycle}. By Lemma~\ref{lem:odd-sim-adj-less}, there exists in $G$ a $(6m,1)$-odd simple adjuster $\mathcal{O}=(v_1,v_2,F_1,F_2,O,\{P_1,P_2\})$ such that, for each $i\in[2]$, $F_i$ is a $(m^{28},\tfrac{d}{100},m+2)$-unit with core $v_i$. Note that $|O|\leq 12m$ and $|\mathsf{Int}(F_i)|\le m^{40}$. Hence the perimeter satisfies $|p(\mathcal{O})|\le |O|+|\mathsf{Int}(F_1)|+|\mathsf{Int}(F_2)|<m^{50}$.
    By Lemma~\ref{lem:adjuster} with $W=p(\mathcal{O})$, there exists in $G-p(\mathcal{O})$ a $(100m^3,m^2)$-adjuster $\mathcal A=(v_3,v_4,F_3,F_4,A,\mathcal P)$ such that $\ell(\mathcal A)\in[60m^3,60m^3+60m]$, where each $F_i$, for $i\in\{3,4\}$, is a $(m^{28},\tfrac{d}{100},m+2)$-unit with core $v_i$.
    Since $\mathsf{Ext}(F_2)$ and $\mathsf{Ext}(F_3)$ both have size at least $\tfrac{dm^{28}}{100}$, there are disjoint subsets $D_2\subseteq \mathsf{Ext}(F_2)$ and $D_3\subseteq \mathsf{Ext}(F_3)$, each of size at least $\tfrac{dm^{28}}{300}$ and disjoint from $p(\mathcal{O})\cup p(\mathcal{A})$.
     Note that $|A|\leq 210m^3$, so
    \[
      |p(\mathcal{O})|+|p(\mathcal{A})|\le |O|+|A|+\sum_{i\in [4]}|\mathsf{Int}(F_i)|\leq 240m^3+4m^{40}<m^{50}.
    \]
    Applying Lemma~\ref{distance} with
    $(X_1,X_2,W)=(D_2,D_3,p(\mathcal{O})\cup p(\mathcal{A}))$,
    we obtain a path $Q_0$ of length at most $m$ between $D_2$ and $D_3$. Extend it to a $v_2,v_3$-path $Q_0'$ of length at most $3m+6$.

    Recall that $D=Cdm^{27}\le Cdm^{30}$, so Claim~\ref{lem_long_path_web} applies throughout the range below. Since $|p(\mathcal{O})|+|p(\mathcal{A})|+|V(Q_0')|\leq 240m^3+5m^{40}<m^{50}$, for every integer $i$ with $0\le i\le\left\lfloor\tfrac{D}{100m}\right\rfloor-3$, setting $L_i:=D-100mi$, Claim~\ref{lem_long_path_web} gives a $v^i_5,v^i_6$-path $\widetilde Q_i$ of length in $[L_i-200m,L_i-100m]$ that avoids $p(\mathcal{O})\cup p(\mathcal{A})\cup V(Q_0')$. Moreover, the webs $W^i_5$ and $W^i_6$ with cores $v^i_5$ and $v^i_6$ contain $(m^3-m^2,m^{28},\tfrac{d}{200},4m)$-subwebs $\widetilde W^i_5$ and $\widetilde W^i_6$, respectively, whose interiors avoid $V(\widetilde Q_i)\cup p(\mathcal{O})\cup p(\mathcal{A})\cup V(Q_0')$ except at their cores. Let
      \[
        W':=p(\mathcal{A})\cup p(\mathcal{O})\cup V(Q_0')\cup V(\widetilde Q_i)\cup\mathsf{Int}(W^i_5)\cup\mathsf{Int}(W^i_6).
      \]
    Then
     $ |W'|\leq 50m^{50}+D+2m^{40}\leq 2D$.
    For each $j\in\{5,6\}$,
    \begin{equation*}
      |\mathsf{Ext}(\widetilde{W}^i_j)|\geq (m^3-m^2)m^{28}\cdot\tfrac{d}{200}\geq \tfrac{dm^{31}}{300}.
    \end{equation*}
    Hence, we can choose vertex-disjoint subsets $D_1\subseteq \mathsf{Ext}(F_1)$ and $D_4\subseteq \mathsf{Ext}(F_4)$, each of size at least $\tfrac{dm^{28}}{700}$ and disjoint from $W'$.
      Furthermore, we can choose vertex-disjoint subsets $D^i_5\subseteq \mathsf{Ext}(\widetilde W^i_5)$ and $D^i_6\subseteq \mathsf{Ext}(\widetilde W^i_6)$, each of size at least $\tfrac{dm^{31}}{700}$ and disjoint from $W'\cup D_1\cup D_4$. Moreover,
      \[
        |W'|\leq 2D=2Cdm^{27}
        \leq \tfrac14\rho\left(\tfrac{dm^{28}}{700}\right)\tfrac{dm^{28}}{700}.
      \]
      Indeed, set $\Lambda:=\log(3m^{28}/(140\eps_2))$. By the definition of $\rho$, after cancelling $dm^{27}$ it suffices to verify $2C\leq \eps_1m/(2800\Lambda^2)$, which holds for sufficiently large $m$. Applying Lemma~\ref{distance} with $(X_1,X_2,W)=(D_1,D^i_5,W')$, we obtain a path $Q^i_1$ of length at most $m$ between $D_1$ and $D^i_5$. Extend it to a $v_1,v^i_5$-path $\widetilde Q^i_1$ of length at most $(m+3)+m+(8m+1)=10m+4<11m$. Let $W'':=W'\cup V(\widetilde Q^i_1)$. Then
      \[
        |W''|\leq 2D+11m\leq 3D
        \leq \tfrac14\rho\left(\tfrac{dm^{28}}{700}\right)\tfrac{dm^{28}}{700}.
      \]
      Applying Lemma~\ref{distance} again with $(X_1,X_2,W)=(D_4,D^i_6,W'')$, we obtain a path $Q^i_2$ of length at most $m$ between $D_4$ and $D^i_6$. Extend it to a $v_4,v^i_6$-path $\widetilde Q^i_2$ of length less than $11m$. Hence, $\widetilde Q^i_1\widetilde Q_i\widetilde Q^i_2$ is a $v_1,v_4$-path of length
    \[
      \ell_i\in L_i+[-200m,-100m]+[2,22m]\subseteq L_i+[-200m,-78m].
    \]
    This path avoids $O\cup V(Q_0')\cup A$. Let $q_0:=\ell(Q_0')\le 3m+6$. Since the two $v_1,v_2$-paths belonging to $\mathcal O$ have lengths at most $6m$ and differ in length by $1$, we obtain two $v_3,v_4$-paths of lengths $\ell'_i$ and $\ell'_i+1$ that avoid $A$, where
   $$
     \ell'_i\in \ell_i+[0,q_0]+[0,6m]\subseteq \ell_i+[0,3m+6]+[0,6m]\subseteq L_i+[-200m,-68m].
   $$
    By the definition of $\mathcal A$, for each $j\in\{0,1,\ldots,m^2\}$, the graph $G[A\cup\{v_3,v_4\}]$ contains a $v_3,v_4$-path of length $\ell(\mathcal A)-2j$.
    Thus every integer in the following interval is an attainable cycle length:
    $$
      I_i:=\ell'_i+\{0,1\}+\{\ell(\mathcal A)-2j:0\le j\le m^2\}
      =[\ell'_i+\ell(\mathcal{A})-2m^2,\ell'_i+\ell(\mathcal{A})+1].
    $$
    For every $j\in\left[\left\lfloor D/(100m)\right\rfloor-3\right]$, $|\ell_j'-\ell_{j-1}'|\le (L_{j-1}-68m)-(L_j-200m)<400m<2m^2$. Thus, combining the intervals for all integers $i$ with $0\le i\le\left\lfloor D/(100m)\right\rfloor-3$, we obtain
    $$
      \bigcup_{0\le i\le \lfloor D/(100m)\rfloor-3}I_i\supseteq [\ell'_{ \lfloor D/(100m)\rfloor-3}+\ell(\mathcal{A})-2m^2,\ell_0'+\ell(\mathcal{A})].
    $$
      This interval has length at least $D+2m^2-600m>D$, and its left endpoint is at most $\ell'_{ \lfloor D/(100m)\rfloor-3}+\ell(\mathcal{A})-2m^2\le 800m-2m^2+\ell(\mathcal{A})=\Theta(m^3)$. Hence, $G$ contains cycles of consecutive even lengths from $\Theta(m^3)$ to $Cdm^{27}$.

    \textbf{Case 2: Suppose that $G$ contains no odd cycle}; equivalently, $G$ is bipartite. By the preceding arguments, we have the following two structures:
    \begin{itemize}
      \item[$(1)$] A $(100m^3,m^2)$-adjuster $\mathcal{A}=(v_1,v_2,F_1,F_2,A,\mathcal{P})$ such that $F_1$ and $F_2$ are internally vertex-disjoint $(m^{28},\tfrac{d}{100},m+2)$-units.
      \item[$(2)$] For every integer $i$ with $0\le i\le\lfloor D/(100m)\rfloor-3$, setting $L_i:=D-100mi$, a path $\widetilde Q_i$ in $G-p(\mathcal A)$ of length in $[L_i-200m,L_i-100m]$ whose endvertices $v^i_5$ and $v^i_6$ are equipped with the $(m^3-m^2,m^{28},\tfrac{d}{200},4m)$-webs $\widetilde W^i_5$ and $\widetilde W^i_6$, respectively; their interiors avoid $p(\mathcal A)\cup V(\widetilde Q_i)$ except at their cores.
    \end{itemize}
    By Lemma~\ref{distance}, we can find a $v_1,v^i_5$-path $\widetilde Q^i_1$ and a $v_2,v^i_6$-path $\widetilde Q^i_2$, each of length at most $11m$, such that they are vertex-disjoint. Hence, $\widetilde Q^i_1\widetilde Q_i\widetilde Q^i_2$ is a path of length $\ell''_i\in[L_i-200m,L_i-78m]$. By the definition of $\mathcal{A}$, the following set consists of attainable even cycle lengths:
    \[
      I_i':=\ell_i''+\{\ell(\mathcal{A})-2j:0\le j\le m^2\}.
    \]
    Considering all integers $i$ with $0\le i\le\lfloor D/(100m)\rfloor-3$, we obtain cycles of every even length in $\bigcup_{0\le i\le\lfloor D/(100m)\rfloor-3} I_i'$. By the same analysis, $G$ contains cycles of consecutive even lengths from $\Theta(m^3)$ to $Cdm^{27}$.
  \end{proof}

  \section{Sparse case}\label{section-spar}

  We divide the sparse argument according to whether a large clique subdivision is present.

 \subsection{The subdivision-free case}
  In this subsection, we prove the subdivision-free case.
  \begin{lemma}\label{lem:sparse}
    Let $0<\eps_0<1/20$. 
    Let
    \(
    0<1/n,1/d\ll\eps_1,\eps_2\ll\eps_0,1/C,1/5,
    \)
    where $d\le\log^{1000}n$ and $Cd\in2\mathbb N$. If $G$ is an $n$-vertex $TK_{(1-\eps_0)d/2}^{(1)\ast d^8}$-free $(\eps_1,\eps_2d)$-expander with $\delta(G)\ge d/2$, then $G$ contains $Cd/2$ consecutive even cycle lengths.
  \end{lemma}

  A graph $F$ is a \emph{$(D,m)$-expansion centered at $v$} if $|F|=D$ and every vertex of $F$ is at distance at most $m$ from $v$.

  \begin{proposition}[\cite{Liu-Mon-JAMS}]\label{lem:smallexpansion}
    Let $D,m\in\mathbb{N}$ and let $1\le D' \le D$. Then any graph $F$ which is a
    $(D,m)$-expansion of $v$ contains a subgraph which is a $(D',m)$-expansion of $v$.
  \end{proposition}

  We use two types of adjusters: the ordinary version from~\cite{Liu-Mon-JAMS} varies length by $2$, and an odd version varies length by $1$ and changes parity.

  \begin{definition}[\cite{Liu-Mon-JAMS}]\label{defn:evenadjuster}
    A \emph{$(D,m,k)$-adjuster} $\mathcal{A}=(v_1,F_1,v_2,F_2,A)$ in a graph $G$ consists of vertices $v_1,v_2\in V(G)$, graphs $F_1,F_2\subseteq G$, and a set $A\subseteq V(G)$ such that the following conditions hold for some $\ell\in\mathbb{N}$.
    \stepcounter{propcounter}
    \begin{enumerate}[label = ({\bfseries \Alph{propcounter}\arabic{enumi}})]
        \rm
      \item\label{adjla1} $A$, $V(F_1)$ and $V(F_2)$ are pairwise disjoint.
        \rm
      \item\label{adjla2} For each $i\in[2]$, $F_i$ is a $(D,m)$-expansion centered at $v_i$.
        \rm
      \item\label{adjla3} $|A|\leq 10mk$.
        \rm
      \item\label{adjla4} For each $i\in\{0,1,\ldots,k\}$, there is a $v_1,v_2$-path in $G[A\cup\{v_1,v_2\}]$ of length $\ell+2i$.
    \end{enumerate}
    We call it a \emph{$(D,m,k)$-odd-adjuster} if it satisfies \ref{adjla1}-\ref{adjla3} and \ref{adjlaodd4} below.
    \stepcounter{propcounter}
    \begin{enumerate}[label = ({\bfseries \Alph{propcounter}\arabic{enumi}})]
        \rm
      \item\label{adjlaodd4} For each $i\in\{0,1,\ldots,k\}$, there is a $v_1,v_2$-path in $G[A\cup\{v_1,v_2\}]$ of length $\ell+i$.
    \end{enumerate}
  \end{definition}

  We denote by $\ell(\mathcal{A})$ the smallest integer $\ell$ for which condition~\ref{adjla4} (or condition~\ref{adjlaodd4}) holds. Note that $\ell(\mathcal{A})\leq |A|+1\leq 10mk+1$. We call $F_1,F_2$ the \emph{ends}, set $V(\mathcal{A})=V(F_1)\cup V(F_2)\cup A$, and call $B_{\mathcal{A}}:=G[A\cup \{v_1,v_2\}]$ the \emph{box}, with core vertices $v_1,v_2$.

The following lemma gives a simple odd-adjuster with prescribed core vertices, which will be used to handle the parity issue in the non-bipartite case.

\begin{lemma}\label{lem:cycleadjuster}
For any $0<\varepsilon_1<1$ and $0<\varepsilon_2<1/20$, there exists $d_0=d_0(\varepsilon_1,\varepsilon_2)$ such that the following holds for each $d_0\le d\le n$. Suppose that $G$ is an $n$-vertex $(\varepsilon_1,\varepsilon_2 d)$-expander with $\delta(G)\ge d-1$.
Let $m=\tfrac{800}{\varepsilon_1}\log^3 n$ and let $D\in\mathbb N$ satisfy $m^3D\le e^{5(\log\log n)^{200}}$. Let $\Gamma$ be a shortest odd cycle in $G$ if $G$ is non-bipartite and a shortest cycle when $G$ is bipartite, and let $x_1,x_2$ be distinct vertices in $V(G)\setminus V(\Gamma)$. If $G$ is non-bipartite, then $G$ contains a $(D,m/4,1)$-odd-adjuster $(v_1,F_1,v_2,F_2,A)$ with $v_1=x_1$, $v_2=x_2$, and
$V(\Gamma)\subseteq A$. Furthermore, if $G$ is bipartite, then $G$ contains a $(D,m/4,1)$-adjuster $(v_1,F_1,v_2,F_2,A)$ with $v_1=x_1$, $v_2=x_2$, and
$V(\Gamma)\subseteq A$.
\end{lemma}

  We use Lemma \ref{vesp:radj} to robustly find an adjuster in a sparse expander. %{\color{blue}(XWEI: In the following lemma some $m$ is undefined. Please make sure if all $m$ can be chosen as in \ref{lem:cycleadjuster}. If so please write this out. Please also make sure what $m$ is in the statement of former lemma 6.10)}

  \begin{lemma}\label{vesp:radj}
    Let $0<\eps_0\le 1/20$.
    There exists $\eps_1>0$ such that, for any $0<\eps_2\ll\eps_0$, there exists $d_0=d_0(\eps_0,\eps_1,\eps_2)$ such that the following holds for each $d_0\le d\le\log^{1000}n$.
    Let  $G$ be an $n$-vertex $TK_{(1-\eps_0)d}^{(1)\ast d^{9}}$-free $(\eps_1,\eps_2d)$-expander with  $\delta(G)\geq d$.
    Let $m=\tfrac{800}{\eps_1}\log^3n$.
    Suppose $\log^{10}n\leq D\leq e^{(\log\log n)^{100}}$, $1\leq r\leq 22m$ and $W\subseteq V(G)$ with $|W|\leq D$. Then $G-W$ contains either a $(D,m,r)$-adjuster or a
$(D,m,r)$-odd-adjuster.
  \end{lemma}
Since the proofs of the above two lemmas follow arguments similar to those in~\cite{Liu-Mon-JAMS}, with an additional treatment of the non-bipartite case, we defer the proofs of Lemmas~\ref{lem:cycleadjuster} and \ref{vesp:radj} to Appendix \ref{appen:proofofdmradjuster}.
The following Lemma \ref{cor:3.15} is a refinement of Corollary~3.15 in \cite{Liu-Mon-JAMS}; inspection of its proof shows that the bipartiteness assumption on $G$ is unnecessary.

  \begin{lemma}[\cite{Liu-Mon-JAMS}]\label{cor:3.15}
    For any $0<\varepsilon_1,\varepsilon_2<1$,
    there exists $d_0=d_0(\varepsilon_1,\varepsilon_2)$ such that the following holds for each $n\geq d\geq d_0$.
    Suppose that $G$ is an $n$-vertex $(\varepsilon_1,\varepsilon_2 d)$-expander with $\delta(G)\geq d$.
    Let $\log^{10} n \leq D \leq n/\log^{10} n$, $\tfrac{100}{\varepsilon_1}\log^3 n \leq m \leq \log^4 n$ and $\ell \leq n/\log^{12} n$.
    Let $A\subseteq V(G)$ satisfy $|A|\leq D/\log^3 n$.
    Let $F_1,\ldots,F_4\subseteq G-A$ be vertex-disjoint subgraphs and let $v_1,\ldots,v_4$ be vertices such that, for each $i\in[4]$, $F_i$ is a $(D,m)$-expansion of $v_i$.
    Then $G-A$ contains vertex-disjoint paths $P$ and $Q$ with
    \(
      \ell \leq \ell(P)+\ell(Q)\leq \ell+22m
    \)
    such that $P$ and $Q$ form a linkage from $\{v_1,v_2\}$ to $\{v_3,v_4\}$.
  \end{lemma}

Lemma \ref{lem:spa-max-deg} is inspired by Theorem~2.7 in \cite{Liu-Mon-JAMS}. We adapt its path-length adjustment argument to our setting by modifying the subdivision-free assumption and allowing a small forbidden set. Moreover, we remove the bipartiteness assumption, so that the path lengths form an arithmetic progression with common difference $\tau\in\{1,2\}$. In particular, in the bipartite case, we may take $\tau=2$.
\begin{lemma}\label{lem:spa-max-deg}%6.6. 6.3. 6.7.
    Let $0<\eps_0\le 1/20$.
    There exists $\eps_1>0$ such that, for any $0<\eps_2\ll\eps_0$, there exists $d_0=d_0(\eps_0,\eps_1,\eps_2)$ such that the following holds for each $d_0\leq d\leq \log^{1000}n$.
    Suppose that $G$ is an $n$-vertex $TK^{(1)\ast d^9}_{(1-\eps_0)d}$-free $(\eps_1,\eps_2d)$-expander with $\delta(G)\geq d$. 
    Let $m=\tfrac{800}{\eps_1}\log^3n$.
    Let $D$ be any integer such that $\log^{10}n\leq D\leq e^{(\log\log n)^{100}}/10$.
    Let $U\subseteq V(G)$ satisfy $|U|\le D/(2\log^3n)$, and let $F_1,F_2\subseteq G-U$ be two vertex-disjoint $(D,m)$-expansions centered at $v_1$ and $v_2$, respectively.
    \begin{itemize}
      \item[$(1)$]\label{spa-max-deg1} There exists $\tau\in\{1,2\}$ such that, for every integer $\ell\in[\log^7 n,n/\log^{12}n]$, there is an integer $\ell'$ with $\ell-22m\leq \ell'\leq\ell$ such that $G-U$ contains $v_1,v_2$-paths of lengths $\ell',\ell'+\tau,\ldots,\ell'+22m\tau$.
      \item[$(2)$]\label{spa-max-deg2} If $G-F_1-F_2-U$ is a bipartite graph, then we may take $\tau=2$ in~$(1)$.
    \end{itemize}
  \end{lemma}

  \begin{proof}
    Let
    \(
      D_A=3D,
      ~
      W_0=U\cup V(F_1)\cup V(F_2).
    \)
    Then $|W_0|\le D_A$ and $D_A\le e^{(\log\log n)^{100}}$. Lemma~\ref{vesp:radj}, applied with scale $D_A$, gives a $(D_A,m,22m)$-adjuster or a $(D_A,m,22m)$-odd-adjuster $\mathcal{A}=(v_3,F_3,v_4,F_4,A)$ in $G-W_0$, with $\ell(\mathcal{A})\le |A|+1\le400m^2$. Let $\tau=2$ if $\mathcal A$ is an adjuster, and let $\tau=1$ if $\mathcal A$ is an odd-adjuster. By Proposition~\ref{lem:smallexpansion}, choose $(D,m)$-subexpansions $F'_3\subseteq F_3$ and $F'_4\subseteq F_4$ with the same centers.

    Fix an integer $\ell\in[\log^7 n,n/\log^{12}n]$ and let $\bar{\ell}=\ell-22m-\ell(\mathcal{A})$. For sufficiently large $n$, we have $0\le\bar{\ell}\le n/\log^{12}n$. As $|A\cup U|\leq400m^2+D/(2\log^3n)\leq D/\log^3n$, Lemma~\ref{cor:3.15}, applied to $F_1,F_2,F'_3,F'_4$, gives vertex-disjoint paths $P$ and $Q$ in $G-U-A$ connecting $\{v_1,v_2\}$ to $\{v_3,v_4\}$, respectively, such that $\bar{\ell}\leq\ell(P)+\ell(Q)\leq\bar{\ell}+22m$. Without loss of generality, we assume that $P$ is a $v_1,v_3$-path and $Q$ is a $v_2,v_4$-path.

    Let $\ell'=\ell(P)+\ell(Q)+\ell(\mathcal A)$. Then $\ell-22m\leq\ell'\leq\ell$. By the definition of an adjuster or odd-adjuster, for every $i\in\{0,1,\ldots,22m\}$, there is a $v_3,v_4$-path $R_i$ in $G[A\cup\{v_3,v_4\}]$ of length $\ell(\mathcal A)+\tau i$. Hence $P\cup R_i\cup Q$ is a $v_1,v_2$-path of length $\ell'+\tau i$ in $G-U$. This proves~$(1)$.

    Finally, if $G-W_0$ is bipartite, Lemma~\ref{vesp:radj} gives an adjuster. Thus $\tau=2$, proving~$(2)$.
  \end{proof}

We are now ready to prove~Lemma~\ref{lem:sparse}.
%6.5. 6.8. 
\begin{proof}[Proof of Lemma~\ref{lem:sparse}]
    For sufficiently large $d$, we have $(d/2)^9\ge d^8$, so $G$ is $TK_{(1-\eps_0)d/2}^{(1)\ast(d/2)^9}$-free. Thus we may apply Lemmas~\ref{lem:cycleadjuster} and~\ref{lem:spa-max-deg} with $d/2$ in place of $d$ and $2\eps_2$ in place of $\eps_2$.
    
    Here, we only consider the case that $G$ is non-bipartite. When $G$ is bipartite, the argument is almost the same. 
    Let $D=\log^{100}n$, and let $\Gamma$ be a shortest odd cycle in $G$. Choose distinct vertices $v_1,v_2\in V(G)\setminus V(\Gamma)$. By Lemma~\ref{lem:cycleadjuster}, there is a $(D,m/4,1)$-odd-adjuster $\mathcal A=(v_1,F_1,v_2,F_2,A)$. 
 By~\ref{adjla3}, $|A|\le10m\le D/(2\log^3n)$, and $F_1,F_2$ are also $(D,m)$-expansions. Lemma~\ref{lem:spa-max-deg}, applied with $U=A$, gives some $\tau\in\{1,2\}$ such that, for every integer $\ell\in[\log^7n,n/\log^{12}n]$, there is an integer $\ell'$ with $\ell-22m\le\ell'\le\ell$ for which $G-A$ contains $v_1,v_2$-paths of lengths $\ell',\ell'+\tau,\ldots,\ell'+22m\tau$.
    For each such $\ell$, fix the corresponding $\ell'_\ell$.  Let $Q_0,Q_1$ be the two $v_1,v_2$-paths in $\mathcal A$ of lengths $\ell(\mathcal{A})$ and $\ell(\mathcal{A})+1$, respectively. For a fixed $\ell$, closing the paths with $Q_0$ or $Q_1$ gives cycles of every length in
    \[
      [\ell(\mathcal{A})+\ell'_\ell,\ell(\mathcal{A})+\ell'_\ell+22m+1]\quad\text{if }\tau=1,
      \qquad
      [\ell(\mathcal{A})+\ell'_\ell,\ell(\mathcal{A})+\ell'_\ell+44m+1]\quad\text{if }\tau=2.
    \]
    In the second case, the two closing paths give the two parities since $Q_0$ and $Q_1$ have different parities.  
    %Since $\ell-22m\le\ell'_\ell\le\ell$, we have $|\ell'_{\ell+1}-\ell'_\ell|\le22m+1$. Thus the intervals corresponding to $\ell,~\ell+1$ have nonempty intersection. Also,
    Notice that for any $\ell\in [\log^7n, n/\log^{12}n]$, $\ell-22m\le\ell'_\ell\le\ell$. So   $\ell(\mathcal{A})+\ell$ lies in the interval corresponding to $\ell$. Hence every integer in $[\ell(\mathcal{A})+\log^7n,\ell(\mathcal{A})+n/\log^{12}n]$ is a cycle length of $G$.

    Since $\ell(\mathcal{A})\le10m+1$ and $d\le\log^{1000}n$, the interval $[2\log^7n,2\log^7n+Cd]$ is contained in $[\ell(\mathcal{A})+\log^7n,\ell(\mathcal{A})+n/\log^{12}n]$ for sufficiently large $n$. It follows that $G$ contains at least $Cd/2$ consecutive even cycle lengths.
\end{proof}

 \subsection{The subdivision-present case}\label{sec:sparse-expander}
In this case, we first show that a large subdivision yields either cycles of consecutive even lengths or a dense subgraph.
\begin{lemma}\label{lem:new}
  Let \(0<1/n,1/d\ll \eps_1,\eps_2\ll \eps\ll 1/5\) with \(d\leq \log^{1000} n\). Let \(G\) be an \(n\)-vertex \((\eps_1,\eps_2d)\)-expander with \(\delta(G)\geq (1-\eps)d/2\) and \(d(G)\geq (1-\eps)d\). Suppose that \(G\) contains a copy \(T\) of \(TK_t^{(1)\ast d^7}\), where \(t=\lfloor (1-2\eps)d/2\rfloor\). Then at least one of the following holds:
  \begin{enumerate}[label=\textup{(\roman*)}]
    \item\label{sub-pre-lem1} \(G\) contains \(d/2\) consecutive even cycle lengths,
    \item\label{sub-pre-lem2} there exists a subgraph \(H\subseteq G-\mathsf{Br}(T)\) with average degree at least \(8\eps d\).
  \end{enumerate}
\end{lemma}

\begin{proof}
  Suppose, for contradiction, that neither~\ref{sub-pre-lem1} nor~\ref{sub-pre-lem2} holds.
  For simplicity, let $R=\mathsf{Br}(T)$ and $\Gamma=G-\mathsf{Br}(T)$.
  Since condition~\ref{sub-pre-lem2} does not hold, we have in particular
  \begin{equation}\label{ine:newl}
    e(\Gamma)<4\eps d |V(\Gamma)|\leq 4\eps dn.
  \end{equation}
  Let
  \(
    A=\{v\in V(\Gamma):d_R(v)\geq16\eps d\}\) and 
\(    B_1=\{v\in V(\Gamma)\setminus A:d_A(v)\geq16\eps d\}.
  \)
  We first show that \(A\) contains almost all vertices of $\Gamma$.

  \begin{claim}\label{claim:A-large}
    \(|A|\geq (1-\sqrt{\eps})n\).
  \end{claim}

  \begin{poc}
    Suppose, to the contrary, that \(|A|<(1-\sqrt{\eps})n\). By the definition of $A$, every vertex in \(V(\Gamma)\setminus A\) has fewer than \(16\eps d\) neighbors in \(R\). Then
    \(
      e_G(R,V(\Gamma)\setminus A)\leq16\eps d|V(\Gamma)\setminus A|\leq16\eps dn.
    \)
    In addition,
    \(e_G(R,A)\leq t|A|\leq\tfrac{(1-2\eps)d}{2}|A|\) and \(e(G[R])\leq d^2\).
    Therefore
    \[
      \begin{aligned}
        e(\Gamma)
        &=e(G)-e_G(R,V(\Gamma))-e(G[R])\geq \tfrac{(1-\eps)d}{2}n-\tfrac{(1-2\eps)d}{2}|A|-16\eps dn-d^2\\
        &> \tfrac{(1-\eps)d}{2}n-\tfrac{(1-2\eps)d}{2}(1-\sqrt{\eps})n-16\eps dn-d^2\geq \left(\tfrac{\sqrt{\eps}}{3}-16\eps\right)dn-d^2
        \geq4\eps dn,
      \end{aligned}
    \]
    where the last inequality holds since $d\leq\log^{1000}n$ and $\eps\ll1/5$. This contradicts~\eqref{ine:newl}. Hence \(|A|\geq(1-\sqrt{\eps})n\).
  \end{poc}

  We proceed with the following observation. Let \(T_0\subseteq T\) be a copy of \(TK_t^{(1)}\) with branch vertex set $\mathsf{Br}(T)$. Since \(K_t\) contains cycles of all lengths \(3,4,\ldots,t\), the graph \(T_0\) contains cycles of all even lengths \(6,8,\ldots,2t\). Moreover, suppose that there are \(m\) internally vertex-disjoint paths \(P_1,\ldots,P_m\), denoted by $\mathcal P$, satisfying the following properties:
  \begin{enumerate}[label=\textup{(\arabic*)}]
    \item\label{path-sec6-1} each $P_i\in\mathcal P$ has length $4$,
    \item\label{path-sec6-2} the endvertices of all $P_i\in\mathcal P$ lie in \(R\),
    \item\label{path-sec6-3} for every $i\in[m]$, $\mathsf{Int}(P_i)\cap V(T_0)=\varnothing$,
    \item\label{path-sec6-4} the edges of the paths in $\mathcal P$ incident with \(R\) form a matching.
  \end{enumerate}
  Then \(T_0\cup P_1\cup\cdots\cup P_m\) contains cycles of all even lengths
  \(
    6,8,\ldots,2t+2m.
  \)
  Similarly, if there are \(m\) internally vertex-disjoint paths \(P_1,\ldots,P_m\) of length \(3\) satisfying conditions~\ref{path-sec6-2}--\ref{path-sec6-4}, then we obtain cycles of all even lengths
  \(
    6,8,\ldots,2t+m.
  \)
  Indeed, by condition~\ref{path-sec6-4}, the pairs of endvertices of these paths form a matching in \(K_t\), and every matching in \(K_t\) is contained in a Hamilton cycle. For each \(1\leq q\leq m\), choose \(q\) of the paths of length four and a Hamilton cycle containing their corresponding pairs of endvertices. Replacing the corresponding two-edge paths in \(T_0\) by these paths gives a cycle of length \(2t+2q\). For paths of length three, use \(2q\) of them for each \(1\leq q\leq m/2\); the same replacement gives a cycle of length \(2t+2q\). Together with the cycles of lengths \(6,8,\ldots,2t\) in \(T_0\), this proves the observation.

  Based on the observation, we next show that \(B_1\) is small.

  \begin{claim}\label{claim:B1-small}
    \(|B_1|<3\eps d\).
  \end{claim}
  \begin{poc}
    Suppose that \(|B_1|\geq3\eps d\), and choose \(m=3\eps d\) distinct vertices \(v_1,\ldots,v_m\in B_1\). Let $\mathcal P$ be a maximal family of internally vertex-disjoint paths, each of the form
    \(
      P_i=x_i a_i v_i b_i y_i
    \)
    for some \(i\in[m]\), such that \(x_i,y_i\in R\), \(a_i,b_i\in A\), and all vertices \(x_i,y_i\) are distinct. We first show that $|\mathcal P|=m$. Indeed, suppose to the contrary that $|\mathcal P|<m$, and choose \(j\in[m]\) such that no path in $\mathcal P$ contains \(v_j\). By the definition of $B_1$, \(d_A(v_j)\geq16\eps d\). Since $|V(\mathcal P)\cap A|<2m\leq7\eps d$, it is possible to choose two distinct vertices
    \(a_j,b_j\in N_A(v_j)\setminus V(\mathcal P)\).
    Similarly, since \(d_R(a_j),d_R(b_j)\geq16\eps d\) and
    $|V(\mathcal P)\cap R|<2m\leq7\eps d$, we may greedily choose distinct vertices
    \(x_j\in N_R(a_j)\setminus V(\mathcal P)\) and
    \(y_j\in N_R(b_j)\setminus(V(\mathcal P)\cup\{x_j\})\).
    Therefore, we find another desired path
    \(P_j=x_ja_jv_jb_jy_j\), contradicting the maximality of $\mathcal P$. Hence $|\mathcal P|=m$.

    Recall that \(T=TK_t^{(1)\ast d^7}\) has \(d^7\) subdividing vertices for every pair of branch vertices. Since the total number of internal vertices used by the paths in $\mathcal P$ is \(O(\eps d)\), we may choose a copy \(T_0\subseteq T\) of \(TK_t^{(1)}\) with branch set \(R\) whose subdividing vertices avoid all internal vertices of the paths in $\mathcal P$. By the observation above, \(T_0\cup P_1\cup\cdots\cup P_m\) contains at least
    \[
      t+m-2\geq\tfrac{(1-2\eps)d}{2}+3\eps d-3\geq\tfrac d2
    \]
    consecutive even cycle lengths, provided \(d\) is sufficiently large. This contradicts the assumption that condition~\ref{sub-pre-lem1} does not hold. Hence \(|B_1|<3\eps d\).
  \end{poc}

  Let
  \(
    C_0=V(\Gamma)\setminus(A\cup B_1).
  \)
  We claim that \(C_0=\varnothing\). Otherwise, for every vertex \(v\in C_0\), since \(v\notin A\cup B_1\), we have \(d_R(v)<16\eps d\) and \(d_A(v)<16\eps d\). Together with \(|B_1|<3\eps d\), this gives
  \[
    d_{C_0}(v)\geq\delta(G)-d_R(v)-d_A(v)-|B_1|
    \geq\tfrac{(1-\eps)d}{2}-16\eps d-16\eps d-3\eps d
    \geq\tfrac d3,
  \]
  where the last inequality holds since \(\eps\ll1/5\). Thus \(\Gamma[C_0]\) has average degree at least \(d/3\), which is larger than \(8\eps d\). This contradicts the assumption that condition~\ref{sub-pre-lem2} does not hold. Hence
  \(
    V(\Gamma)=A\cup B_1.
  \)

  We next prove that \(G[A]\) contains a large matching.

  \begin{claim}\label{claim:A-matching}
    There exists a matching of size at least \(8\eps d\) in \(G[A]\).
  \end{claim}

  \begin{poc}
    Let \(M\) be a maximal matching in \(G[A]\). Suppose that $|M|<8\eps d$. Then \(|V(M)|<16\eps d\), and \(A\setminus V(M)\) is an independent set. Let $W=R\cup B_1\cup V(M)$. Then
    \begin{equation*}
      |W|\leq|R|+|B_1|+|V(M)|
      \leq t+3\eps d+16\eps d\leq d.
    \end{equation*}
    By Claim~\ref{claim:A-large}, \(|A|\geq(1-\sqrt{\eps})n\). Since \(d\leq\log^{1000}n\), it is possible to choose a set \(A_0\subseteq A\setminus V(M)\) of size \(n/5\). As \(A\setminus V(M)\) is independent and \(V(\Gamma)=A\cup B_1\), every neighbor of \(A_0\) in \(G\) lies in \(W\). Hence
    $|N_G(A_0)|\leq|W|\leq d$.
    On the other hand, for sufficiently large \(n\), we have
    \(\eps_2d/2\leq|A_0|=n/5\leq n/2\), and hence the expansion property gives $|N_G(A_0)|\geq\rho(n/5)\cdot n/5\gg d$, a contradiction. Therefore \(G[A]\) contains a matching of size at least \(8\eps d\).
  \end{poc}

  By Claim~\ref{claim:A-matching}, choose \(r=5\eps d\) disjoint edges \(e_i=a_ib_i\), \(1\leq i\leq r\), from the matching above. It is possible to greedily choose distinct vertices \(x_i,y_i\in R\) such that
  \(
    x_ia_i,\ b_iy_i\in E(G)
  \)
  for every \(i\in[r]\). Indeed, each \(a_i,b_i\in A\) has at least \(16\eps d\) neighbors in \(R\), while fewer than \(2r\leq11\eps d\) vertices have been used. Thus we obtain internally vertex-disjoint paths
  \(
    Q_i=x_ia_ib_iy_i
  \)
  of length \(3\), whose edges incident with the vertices in
  \(\bigcup_{i\in[r]}V(Q_i)\cap R\) form a matching. As above, choose a copy \(T_0\subseteq T\) of \(TK_t^{(1)}\) with \(\mathsf{Br}(T_0)=R\) whose subdividing vertices avoid all internal vertices of the paths \(Q_1,\ldots,Q_r\).

  Using an even number of the paths \(Q_i\), the observation above gives cycles of all even lengths
  \(
    6,8,\ldots,2t+r.
  \)
  Therefore, the number of consecutive even cycle lengths obtained is at least
  \[
    t+r/2-2
    \geq\tfrac{(1-2\eps)d}{2}+\tfrac{5\eps d}{2}-5
    \geq\tfrac d2
  \]
  for \(d\) sufficiently large. This gives \(d/2\) consecutive even cycle lengths, contradicting the assumption that condition~\ref{sub-pre-lem1} does not hold. Thus at least one of conditions~\ref{sub-pre-lem1} and~\ref{sub-pre-lem2} holds.
\end{proof}
\subsection{Putting things together: proof of Lemma~\ref{lem:den-exp-reduced}}

\begin{proof}[Proof of Lemma~\ref{lem:den-exp-reduced}]
We divide the proof according to whether $H$ contains a sufficiently large clique subdivision. Let
\[
d_0:=(1-\eps)d,\qquad
\eps_0:=2\eps,\qquad
t:=\tfrac{(1-\eps_0)d}{2}.
\] 
We first consider the subdivision-free case. Let
\[
\widetilde{\eps}_0
:=
\tfrac{\eps_0-\eps}{1-\eps}
=
\tfrac{\eps}{1-\eps},
\qquad
\widetilde{\eps}_2
:=
\tfrac{\eps_2}{1-\eps},
\qquad
\widetilde C
:=
\tfrac{3}{1-\eps}.
\]
Then $(1-\widetilde{\eps}_0)d_0 = (1-\eps_0)d = 2t$, $\widetilde{\eps}_2d_0=\eps_2d$, and $\widetilde C d_0=3d$.
Suppose that $H$ is $TK_t^{(1)\ast d_0^8}$-free. Since $H$ is an
$(\eps_1,\widetilde{\eps}_2d_0)$-expander and
$\delta(H)\ge \tfrac{d_0}{2}$,
Lemma~\ref{lem:sparse}, applied with 
$(G,d,\eps_0,C,\eps_1,\eps_2)
=
(H,d_0,\widetilde{\eps}_0,\widetilde C,
 \eps_1,\widetilde{\eps}_2)$, shows that $H$ contains at least $\tfrac{\widetilde C d_0}{2}=\tfrac{3}{2}d>\tfrac d2$ consecutive even cycle lengths.

Therefore, from now on, we need only consider the case in which $H$ contains a copy $T$ of $TK_t^{(1)\ast d_0^8}$.
Since $d_0=(1-\eps)d$, we have $d_0^8\ge d^7$ for sufficiently large $d$. Thus $T$ contains a copy of $TK_t^{(1)\ast d^7}$.
By Lemma~\ref{lem:new}, $H$ contains $d/2$ consecutive even cycle lengths unless there exists a subgraph $G'\subseteq H-\mathsf{Br}(T)$ with average degree at least $8\eps d$.
Denote $d_2:=8\eps(1-\eps)d$ and $c':=\tfrac{\eps_2d_2}{d(G')}$.  Note that $0<c'\leq \eps_2(1-\eps)<1/2$. Applying Lemma~\ref{lem:sublinear} to $G'$ with parameter $c'$, we get that $G'$ contains an $h'$-vertex $(\eps_1,c'd(G')=\eps_2d_2)$-expander $H'$ satisfying  $d(H')\ge (1-\eps)d(G')\geq (1-\eps)8\eps d=d_2$ and 
$\delta(H')\ge 4\eps(1-\eps)d=\tfrac{1}{2}d_2$.
Moreover, $H'$ is also $\nu'd_2$-connected for some $\nu'=\nu'(\eps_1,\eps_2,\eps)>0$.

Choose a constant $K$ such that $\tfrac1d\ll\tfrac1K\ll\nu, \nu'$.
Next, we divide the proof into  sparse, intermediate, and dense cases, respectively:
\[
d_2\le\log^{1000}h',
\qquad
\log^{1000}h'\le d_2\le\tfrac{h'}K,
\qquad
d_2>\tfrac{h'}K.
\]

\medskip

\noindent\textbf{Sparse Case.}
Suppose that $d_2\le\log^{1000}h'$. If $H'$ is $TK_{(1-\eps_0)d_2/2}^{(1)\ast d_2^8}$-free, 
then Lemma~\ref{lem:sparse}, applied with
$(G,d,\eps_0,C,\eps_1,\eps_2)
=
\left(
H',d_2,\eps_0,\tfrac{8d}{d_2}=\tfrac{1}{\eps(1-\eps)},\eps_1,\eps_2
\right)$,
shows that $H'$ contains $\tfrac{Cd_2}{2} = 4d$
consecutive even cycle lengths. In particular, $G$ contains
$d/2$ consecutive even cycle lengths.

We may therefore suppose that $H'$ contains a copy $T'$ of $TK_{(1-\eps_0)d_2/2}^{(1)\ast d_2^8}$. Choose a copy $T_1'$ of $TK_{(1-\eps_0)d_2/2}^{(1)}$ as a subgraph of $T'$. As $|V(T_1')| \le \tfrac{(1-\eps_0)d_2}{2}+d_2^2 < d_0^3$,
and every pair of branch vertices of $T$ has $d_0^8$ available subdividing vertices, there exists a copy $T_1\cong TK_t^{(1)}$ as a subgraph of $T$ such that $V(T_1)\cap V(T_1')=\varnothing$. Let
$B:=\mathsf{Br}(T_1)$ and $B':=\mathsf{Br}(T_1')$.
As $\min\{|V(T_1)|,|V(T_1')|\}>\nu d$ and $H$ is $\nu d$-connected, Menger's theorem gives at least $18$ vertex-disjoint paths between $V(T_1)$ and $V(T_1')$.
Classify each of these paths according to the parity of its length, whether its endvertex in $T_1$ lies in $B$, and whether its endvertex in $T_1'$ lies in $B'$. There are at most eight classes. Hence, by the pigeonhole principle, there exist two vertex-disjoint paths $P_1, P_2$ such that, after writing $P_i$ as a $u_i,v_i$-path and setting $\ell_i:=\ell(P_i)$, the following holds.
\begin{itemize}
    \item[(1)] For every $i\in\{1,2\}$, $u_i\in V(T_1)$ and $v_i\in V(T_1')$; moreover, $\ell_1\equiv\ell_2\pmod 2$.
    \item[(2)] $u_1,u_2\in B$ or $u_1,u_2\in V(T_1)\setminus B$.
    \item[(3)]  $v_1,v_2\in B'$ or $v_1,v_2\in V(T_1')\setminus B'$.
\end{itemize}

If $u_1,u_2\in B$, then $T_1$ contains $u_1u_2$-paths of length $2,4,\ldots,2(t-1)$.
If $u_1,u_2\in V(T_1)\setminus B$, then $T_1$ contains
$u_1,u_2$-paths of lengths 
$4,6,\ldots,2t$. 
Consequently, for every even integer 
$\ell\in[4,2t-2]$, 
there exists a $u_1,u_2$-path $P$ of length $\ell$ in $T_1$.

By the same argument, for every even integer  $\ell' \in [4,(1-\eps_0)d_2-2]$, there exists a $v_1,v_2$-path $P'$ of length $\ell'$ in $T_1'$.
Therefore, $P\,P_1\,P'\,P_2$ forms a cycle in $H$. Varying $P$ and $P'$ gives cycles of every even length in
$\left[
\ell_1+\ell_2+8,\,
\ell_1+\ell_2+2t+(1-\eps_0)d_2-4
\right]$.
The number of consecutive even cycle lengths in this interval is
$t+\tfrac{(1-\eps_0)d_2}{2}-5 =
\tfrac{(1-2\eps)d}{2}+4\eps(1-\eps)(1-2\eps)d-5
>\tfrac d2$
for sufficiently large $d$. Hence $G$ contains $d/2$ consecutive even cycle lengths.

\medskip

\noindent\textbf{Intermediate Case.}
Suppose that 
$\log^{1000}h'\le d_2\le\tfrac{h'}K$. 
By Lemma~\ref{lem:dense}, applied with 
$(G,d,C)
=
\left(
H',d_2,\tfrac{1}{\eps(1-\eps)}
\right)$, 
the graph $H'$ contains
$\tfrac{Cd_2}{2}
=
\tfrac{1}{2\eps(1-\eps)}
\cdot 8\eps(1-\eps)d
=
4d$
consecutive even cycle lengths. Thus $G$ contains $d/2$
consecutive even cycle lengths.

\medskip

\noindent\textbf{Dense Case.}
In this case, $d_2>h'/K$. Apply Lemma~\ref{lem:verydense-path} with a fixed constant $\alpha\le 1/22$ and  $(G,d,K,\nu)=(H',d_2,K,\nu')$. There exist a partition  $V(H')=X_0\mathbin{\dot\cup}X_1$ and a constant $M=M(K,\nu')$ with the following property: For distinct vertices $u,v\in V(H')$, there exists an integer $s=s(u,v)\in[1,M/10]$ satisfying
$s\equiv\pi(u,v)\pmod 2$
such that, for every integer $0\le i\le\tfrac5{11}d_2$, the graph $H'$ contains a $u,v$-path of length $s+2i$.

Since $d_2>h'/K$ and $d$ is large, we have
$|H'|=h'<Kd_2<d_0^3$. Notice that each pair of branch vertices of $T$ has $d_0^8$ available subdividing vertices. Hence we can find a copy $T_1$ of $TK_t^{(1)}$ inside $T$ such that $V(T_1)\cap V(H')=\varnothing$.
As $\min\{|V(T_1)|, |V(H')|\}\ge \nu d$ and $H$ is $\nu d$-connected, Menger's theorem gives at least
$18$ vertex-disjoint paths between $V(T_1)$ and $V(H')$ whose internal vertices lie outside $V(T_1)\cup V(H')$.
After classifying each such path according to the parity of its length, whether its endvertex in $T_1$ is a branch vertex, and whether
its endvertex in $H'$ lies in $X_0$ or $X_1$, we obtain at least two paths, say a $u_0,u$-path
$P_1$ and a $v_0,v$-path $P_2$, for which the following conditions hold.
\begin{itemize}
    \item[(1)] $\pi(u,v)=0$ and $\ell_1\equiv\ell_2\pmod 2$. Here $\ell_1:=\ell(P_1)$ and $\ell_2:=\ell(P_2)$.
    \item[(2)] Either $u_0,v_0\in\mathsf{Br}(T_1)$ 
or $u_0,v_0\in V(T_1)\setminus\mathsf{Br}(T_1)$.
\end{itemize}

Recall that the graph $H'$ contains a $u,v$-path $Q_i$ of length $s+2i$ for every $0\le i\le\tfrac5{11}d_2$, with $s\equiv \pi(u,v)\equiv0\pmod{2}$, and $T_1$ contains a $u_0,v_0$-path $Q_j'$ of length $2j$ for every integer $j$ with $2\le j\le t-1$.
Thus $Q_i\, P_1\, Q_j'\, P_2$ forms a cycle in $H$ of length $s+\ell_1+\ell_2+2i+2j$. 
As both $s$ and $\ell_1+\ell_2$ are even, varying $i$ and $j$ gives at least
$t+\tfrac5{11}d_2-3>\tfrac{d}{2}$
consecutive even cycle lengths.
\end{proof}

\section{Concluding remarks}\label{sec:con}

Theorem~\ref{thm:main} and its modular applications point to a more general
question about forbidden cycle lengths. For a set
$\mathcal L\subseteq\{3,4,\ldots\}$, let $\operatorname{ex}(n,\mathcal L)$
denote the maximum number of edges in an $n$-vertex graph containing no cycle
whose length belongs to $\mathcal L$. Define
\[
  \ell_{\rm odd}:=\min\bigl(\mathcal L\cap(2\mathbb N+1)\bigr),\qquad
  \ell_{\rm even}:=\min\bigl(\mathcal L\cap2\mathbb N\bigr),\qquad
  \ell:=\min\{\ell_{\rm odd},\ell_{\rm even}\},
\]
where the minimum of the empty set is taken to be infinity. We assume that $\ell_{\rm even}<\infty$.

Two familiar constructions give the natural lower bounds. For the first, take
a connected graph whose blocks are complete graphs of order 
$\ell-1$, thus having $\tfrac{\ell-1}{2}n-O_{\mathcal L}(1)$ edges. Since every cycle is contained in a single block, the resulting
graph is $\mathcal L$-free. For the second, consider the
complete bipartite graph
\(
  K_{\ell_{\rm even}/2-1,\,n-\ell_{\rm even}/2+1}
\).
It has no odd cycles, while every even cycle has length at most
$\ell_{\rm even}-2$, and hence it is also $\mathcal L$-free. Its number of
edges is $\tfrac{\ell_{\rm even}-2}{2}n-O_{\mathcal L}(1)$. Consequently,
$\operatorname{ex}(n,\mathcal L)\geq
\max\{(\ell-1)/2,(\ell_{\rm even}-2)/2\}n-O_{\mathcal L}(1)$.

The relative strengths of the two constructions are determined by the first
odd and even members of $\mathcal L$. If $\ell=\ell_{\rm even}$, the
clique-block construction has the larger density. If
$\ell=\ell_{\rm odd}$, the complete bipartite construction has the larger
density, except when $\ell_{\rm even}=\ell_{\rm odd}+1$, in which case the
two leading coefficients agree.

\begin{problem}\label{prob:forbidden-cycle-spectrum}
Determine natural arithmetic conditions on $\mathcal L$ under which
\[
  \operatorname{ex}(n,\mathcal L)
  =
  \left(
    \max\left\{
      \tfrac{\ell-1}{2},
      \tfrac{\ell_{\rm even}-2}{2}
    \right\}
    +o(1)
  \right)n,
\]
and characterize the corresponding extremal and near-extremal graphs. In
particular, identify those $\mathcal L$ for which the least odd and least even
members determine whether the extremal graphs have clique-block or complete
bipartite structure.
\end{problem}

Both alternatives already occur in divisible-cycle problems. If
$\mathcal L=\{k,2k,3k,\ldots\}$ and $k$ is odd, then
$\ell_{\rm odd}=k$ and $\ell_{\rm even}=2k$, and the complete bipartite
construction has the larger density. When $k$ is even, the first forbidden
length is even and the clique-block construction is extremal. The same
construction is extremal for cycles of length $2\pmod k$ with $k$ even,
where the first possible forbidden cycle length is $k+2$.

Some restriction on the size or arithmetic structure of $\mathcal L$ is
necessary, since the extremal number need not be linear for an arbitrary
set of forbidden lengths. Thus the least odd and least even members of
$\mathcal L$ alone do not determine the answer, and further assumptions
are needed to rule out other extremal constructions.
Theorem~\ref{thm:main} concerns a related but different question.
Rather than forbidding cycles of prescribed lengths, we forbid any
$t$ consecutive even cycle lengths, without fixing the first length.
This amounts to forbidding every even translate of
$\{0,2,\ldots,2(t-1)\}$ in the set of cycle lengths.
It would be interesting to extend
Problem~\ref{prob:forbidden-cycle-spectrum} to include such patterns,
and to determine when the clique-block and complete bipartite
constructions remain extremal.

\vspace{0.3cm}
\noindent
\textbf{Acknowledgments.} This project was initiated at the ``Communications in Combinatorics'' workshop in Shanghai, China, in December 2025. The second author thanks Yandong Bai for bringing this problem to his attention during the workshop.

All mathematical ideas, problem formulations, proof strategies, and
arguments in this paper were developed by the authors. ChatGPT was used
solely as a checking tool: to identify typographical and linguistic errors
and to flag possible gaps or inconsistencies in the written proofs.

    \begin{appendix}
      \small
\section*{Appendix}
    We include here the missing proofs of lemmas used in Sections~\ref{section-med} and~\ref{section-spar}. For some of them, we provide only proof sketches as they are straightforward adaptions of previous ones. For proofs with full details, we refer the readers to the appendix of arXiv version 1.

      \section{Proofs of Lemmas~\ref{lem:lar-adjuster} and~\ref{lem:adjuster}}\label{proofadj-odd}

It suffices to prove Lemmas~\ref{lem:lar-simple-adj} and~\ref{lem:simple-adj}. 
Lemmas~\ref{lem:lar-adjuster} and~\ref{lem:adjuster} can be obtained by iterating Lemmas~\ref{lem:lar-simple-adj} and~\ref{lem:simple-adj}, respectively; see, e.g.,~\cite{Liu-balanced-sub}.
Recall that $m$ is the smallest even integer larger than $\log^4(n/d)$, $L_G:=\{v\in V(G):d_G(v)\ge dm^{31}\}$, and, in the large-degree case, $L_G^0,L_G^1\subseteq L_G$ are the fixed disjoint sets from Section~\ref{subsec:52} of sizes $d/50$ and $d/5$, respectively.

 \begin{lemma}\label{lem:lar-simple-adj}
    Suppose $0<\tfrac{1}{n},\tfrac{1}{d}\ll \tfrac{1}{K}\ll \eps_1,\eps_2<\tfrac{1}{5}$ with $\tfrac{1}{4}\log^{1000}n\leq d\leq \tfrac{n}{K}$. Let $G=(V,E)$ be an $n$-vertex $(\eps_1,\eps_2d)$-expander with $\delta(G)\geq \tfrac{d}{2}$. Let $W\subseteq V(G)$ be a set of size at most $m^{10}$. If $|L_G|\ge \tfrac{11}{50}d$, then there exists a $(2m,1)$-adjuster $\mathcal{A}=(v_1,v_2,S_1,S_2,A,\mathcal{P})$ in $G-L_G^1-W$ such that $v_1,v_2\in L_G^0$, where $S_i$ is a $dm^{28}$-star for each $i\in[2]$. Moreover, $|A|\leq 2\ell(\mathcal{A})$.
  \end{lemma}

      \begin{proof}
        By Proposition~\ref{fact:bipartition}, we can take a partition $V(G)=V_1\mathbin{\dot\cup}V_2$ such that, for every $i\in[2]$ and every $v\in V_i$, we have $d_{V_{3-i}}(v)\geq \tfrac{d_G(v)}{2}$. By the pigeonhole principle, either $|V_1\cap L_G^0|\geq \tfrac{d}{100}$ or $|V_2\cap L_G^0|\geq\tfrac{d}{100}$. Without loss of generality, assume $|V_1\cap L_G^0|\geq \tfrac{d}{100}$ and set $\widetilde{L_G^0}:=V_1\cap L_G^0$.
        As $|\widetilde{L_G^0}|-|W|\ge \tfrac{d}{100}- m^{10}>m$,
        we can choose two distinct vertices $v_1,v_2\in \widetilde{L_G^0}\backslash W$.
        As $N_G(v_1)\cap V_2$ and $N_G(v_2)\cap V_2$ each have  size at least $\tfrac{dm^{31}}{2}$, while $|L_G^0\cup L_G^1\cup W|\le \tfrac{11}{50}d+m^{10}$,
        we can choose disjoint subsets $N_1\subseteq (N_G(v_1)\cap V_2)\setminus (L_G^0\cup L_G^1\cup W)$ and $N_2\subseteq (N_G(v_2)\cap V_2)\setminus (L_G^0\cup L_G^1\cup W)$, each of size at least $\tfrac{dm^{31}}{10}$. By the same argument, choose another vertex $v_3\in \widetilde{L_G^0}\setminus(W\cup \{v_1,v_2\})$ and a set $N_3\subseteq (N_G(v_3)\cap V_2)\setminus (L_G^0\cup L_G^1\cup W\cup N_1\cup N_2)$ of size at least $\tfrac{dm^{31}}{10}$. Observe that $N_1,N_2,N_3\subseteq V_2$. Let $W':=W\cup \{v_1,v_2,v_3\}\cup L_G^0\cup L_G^1$. Then $|W'|\leq m^{10}+3+\tfrac{11d}{50}\leq \tfrac{3d}{10}$.

        First, we claim that there is a family $\mathcal{P}_0$ of $dm^{26}$ pairwise vertex-disjoint paths of length at most $m/2$ that connect $N_1$ and $N_2$ and avoid $W'$. Consider a maximal family $\mathcal{P}_0$ of such paths. If $|\mathcal{P}_0|<dm^{26}$, then $|V(\mathcal{P}_0)|<dm^{27}$. For each $i\in[2]$, we have
        \begin{equation*}
          |N_i-V(\mathcal{P}_0)-W'|\geq \tfrac{dm^{31}}{10}-dm^{27}-\tfrac{3d}{10}\geq \tfrac{dm^{31}}{20}.
        \end{equation*}
        %\wx{$\tfrac{dm^{31}}{10}-dm^{27}-m^{10}-3-11d/50$?}
        Applying Lemma~\ref{distance} with $(X_1,X_2,W)=(N_1\setminus(V(\mathcal{P}_0)\cup W'),N_2\setminus(V(\mathcal{P}_0)\cup W'),V(\mathcal{P}_0)\cup W')$, we obtain a path of length at most $m/2$ between $N_1$ and $N_2$ that avoids $V(\mathcal{P}_0)\cup W'$, contradicting the maximality of $\mathcal{P}_0$.
        Thus $|\mathcal{P}_0|\geq dm^{26}$, as claimed. By the pigeonhole principle, there exists a subfamily $\mathcal{P}\subseteq \mathcal{P}_0$ of $dm^{25}$ paths of equal length. Let $B:=V(\mathcal{P})\cap N_1$; then $|B|=dm^{25}$. For every $v\in B$, let $P_v\in \mathcal{P}$ be the path with endvertex $v$. There are two cases.

        \textbf{Case $1$}: There is a vertex $w\notin W'$ with two neighbors $w_1,w_2\in B$.
        Let $P_{w_1},P_{w_2}\in \mathcal{P}$ be the paths containing $w_1$ and $w_2$, respectively, and let $V(P_{w_i})\cap N_2=\{z_i\}$ for each $i\in[2]$.
        Without loss of generality, assume that $w\notin V(P_{w_2})$; otherwise, swap the labels of $w_i$, $P_{w_i}$, and $z_i$ for $i=1,2$.
        Set $R_1:=v_1w_1P_{w_1}z_1v_2$ and $R_2:=v_1w_1ww_2P_{w_2}z_2v_2$.
        These are two $v_1,v_2$-paths whose lengths differ by $2$, and the shorter has length at most $m/2+2\leq2m$.
        For each $i\in[2]$, choose a $dm^{28}$-star $S_i$ with core $v_i$ and leaves in $N_i\setminus(V(R_1)\cup V(R_2))$. Let $\mathcal P':=\{R_1,R_2\}$ and $A:=\bigl(V(R_1)\cup V(R_2)\bigr)\setminus\{v_1,v_2\}$.
        Then $\mathcal{A}:=(v_1,v_2,S_1,S_2,A,\mathcal P')$ is a $(2m,1)$-adjuster in $G-L_G^1-W$. Moreover, $|A|\leq2\ell(\mathcal A)$.

        \textbf{Case $2$}: No vertex $w\notin W'$ has two neighbors in $B$.
    Note that, for each $v\in B$, $|N_G(v)\cap W'|\leq m^{10}+\tfrac{11}{50}d+3$. Since every vertex outside $W'$ has at most one neighbor in $B$, the set $N_G(B)\setminus W'$ has size at least $|B|(\tfrac{d}{4}-m^{10}-\tfrac{11}{50}d-3)$. Hence
        \[
          |N_G(B)\backslash (W'\cup V(\mathcal P))|\ge |B|(\tfrac{d}{4}-m^{10}-\tfrac{11}{50}d-3)-|B|m \ge \tfrac{d|B|}{50} \ge \tfrac{d^2m^{25}}{50}.
        \]
        Applying Lemma~\ref{distance} with $(X_1,X_2,W)=(N_G(B)\setminus (W'\cup V(\mathcal P)),N_3,W'\cup V(\mathcal P))$, we obtain a path $P$ of length at most $m/2$ joining a vertex $u\in N_G(B)\setminus (W'\cup V(\mathcal P))$ to a vertex $y\in N_3$ and avoiding $W'\cup V(\mathcal P)$ internally. Choose $w_1\in B$ such that $u\in N_G(w_1)$.
        Choose $w_2\in B\setminus\{w_1\}$, and let $V(P_{w_i})\cap N_2=\{z_i\}$ for each $i\in[2]$. Set $R_1:=v_3yPuw_1P_{w_1}z_1v_2$ and $R_2:=v_3yPuw_1v_1w_2P_{w_2}z_2v_2$. These are two $v_3,v_2$-paths whose lengths differ by $2$, and the shorter has length at most $\ell(P)+\ell(P_{w_1})+3\leq m+3\leq2m$. Let $\mathcal P^*:=\{R_1,R_2\}$ and $A^*:=\bigl(V(R_1)\cup V(R_2)\bigr)\setminus\{v_2,v_3\}$.
        Choose a $dm^{28}$-star $S_3'$ with core $v_3$ and leaves in $N_3\setminus(V(R_1)\cup V(R_2))$, and a $dm^{28}$-star $S_2'$ with core $v_2$ and leaves in $N_2\setminus(V(R_1)\cup V(R_2))$. Then
        $\mathcal{A}^{*}:=(v_3,v_2,S_3',S_2',A^{*},\mathcal P^{*})$
        is a $(2m,1)$-adjuster in $G-L_G^1-W$. Moreover, $|A^{*}|\leq 2\ell(\mathcal{A}^{*})$, as desired.
      \end{proof}

\begin{lemma}\label{lem:simple-adj}
    Suppose $0<\tfrac{1}{n},\tfrac{1}{d}\ll \tfrac{1}{K}\ll \eps_1,\eps_2<\tfrac{1}{5}$ with $\tfrac{1}{4}\log^{1000}n\leq d\leq \tfrac{n}{K}$. Let $G$ be an $n$-vertex $(\eps_1,\eps_2d)$-expander with $\delta(G)\geq \tfrac{d}{2}$. Let $W\subseteq V(G)$ be a set with $|W|\leq 50m^{50}$. There exists a $(5m,1)$-adjuster $\mathcal{A}=(v_1,v_2,F_1,F_2,A,\mathcal{P})$ in $G-W$, where $F_i$ is a $(\tfrac{3}{2}m^{28},\tfrac{d}{100}, m+2)$-unit for each $i\in[2]$. Moreover, $|A|\leq 2\ell(\mathcal{A})$.
  \end{lemma}

      \begin{proof}
        Since $|W|\leq 50m^{50}$, Lemma~\ref{claim:unit} gives three vertex-disjoint units $F_1,F_2,F_3$ in $G-W$, each with parameters $(2m^{28},d/100,m+2)$ and with core vertices $v_1,v_2,v_3$, respectively. Set $W':=W\cup \mathsf{Int}(F_1)\cup \mathsf{Int}(F_2)\cup \mathsf{Int}(F_3)$. Then $|W'|\leq 50m^{50}+3\cdot 2m^{28}(m+2)\le 60m^{50}$.

        Next, we claim that there is a family $\mathcal{P}_0$ of $dm^{26}$ pairwise vertex-disjoint paths in $G-W'$ joining $\mathsf{Ext}(F_1)$ to $\mathsf{Ext}(F_2)$, each of length at most $m$. Consider a maximal family $\mathcal{P}_0$ of such paths. If $|\mathcal{P}_0|<dm^{26}$, then $|V(\mathcal{P}_0)|<dm^{27}$. Let $W'':=W'\cup V(\mathcal{P}_0)$; then $|W''|\leq 60m^{50}+dm^{27}\leq 2dm^{27}$. For each $i\in[2]$, we have $|\mathsf{Ext}(F_i)\setminus V(\mathcal{P}_0)|\geq \tfrac{dm^{28}}{50}-2dm^{27}\geq \tfrac{dm^{28}}{100}$. Applying Lemma~\ref{distance} with $\mathsf{Ext}(F_1)\setminus V(\mathcal{P}_0)$, $\mathsf{Ext}(F_2)\setminus V(\mathcal{P}_0)$, and $W''$ playing the roles of $X_1,X_2,W$, respectively, we obtain another path of length at most $m$, contradicting the maximality of $\mathcal{P}_0$. Thus $|\mathcal{P}_0|\geq dm^{26}$, as claimed. By averaging, there exists a subfamily $\mathcal{P}\subseteq \mathcal{P}_0$ of $dm^{25}$ paths of equal length. Let $B:=V(\mathcal{P})\cap \mathsf{Ext}(F_1)$; then $|B|=dm^{25}$.
        For any $v\in B$, denote by $P_v\in \mathcal{P}$ the path with $v$ as an endvertex.
        We now consider two cases.

        \textbf{Case $1$}: There is a vertex $w\notin W'$ with two neighbors $w_1,w_2\in B$.
        Let $Q_1$ and $Q_2$ be the paths in $F_1$ joining $v_1$ to $w_1$ and $w_2$, respectively. Let $R_1$ and $R_2$ be the paths in $F_2$ joining the respective endvertices of $P_{w_1}$ and $P_{w_2}$ in $\mathsf{Ext}(F_2)$ to $v_2$. Without loss of generality, assume that $w\notin V(P_{w_2})$. Set $T_1:=v_1Q_1P_{w_1}R_1v_2$ and $T_2:=v_1Q_1w_1ww_2P_{w_2}R_2v_2$. These are two $v_1,v_2$-paths whose lengths differ by $2$, and the shorter has length at most $(m+3)+m+(m+3)=3m+6\leq5m$. Let $F_1'$ be the subgraph obtained from $F_1$ by removing the paths $Q_1,Q_2$ and the leaves of their pendant stars, and define $F_2'$ analogously using $R_1,R_2$. For each $i\in[2]$, the graph $F_i'$ contains a $(\tfrac{3}{2}m^{28},\tfrac{d}{100},m+2)$-subunit $F_i''$ with core $v_i$. Let $\mathcal P':=\{T_1,T_2\}$ and $A':=\bigl(V(T_1)\cup V(T_2)\bigr)\setminus\{v_1,v_2\}$. Therefore, $\mathcal A':=(v_1,v_2,F_1'',F_2'',A',\mathcal P')$
        is a $(5m,1)$-adjuster in $G-W$. Moreover, $|A'|\leq2\ell(\mathcal A')$.

        \textbf{Case $2$}: No vertex $w\notin W'$ has two neighbors in $B$.
 Let $B'\subseteq B$ consist of the vertices $v$ for which there
  exist $v'\in B\setminus\{v\}$ and $x\in\mathsf{Int}(F_1)$ such
  that $xv,xv'\in E(G)$. 
  Since every vertex of $B\setminus B'$ has a distinct neighbor
  in $\mathsf{Int}(F_1)$, we have
  \[
    |B'|
    \geq |B|-|\mathsf{Int}(F_1)|
    \geq dm^{25}-m^{50}
    \geq \tfrac{dm^{25}}{2}.
  \]

  We claim that there is a set $B_1\subseteq B'$ satisfying
  $m^{25}\leq |B_1|\leq m^{25}+1$
  such that every vertex of $B_1$ has a partner in $B_1$ with
  which it has a common neighbor in $\mathsf{Int}(F_1)$. To see this, define an auxiliary graph $H$ on $B'$ by joining
  two distinct vertices precisely when they have a common
  neighbor in $\mathsf{Int}(F_1)$. By the definition of $B'$,
  the graph $H$ has no isolated vertices. Consider the connected
  components of $H$ in an arbitrary order and add whole components
  until their union first has at least $m^{25}$ vertices. If the
  last component causes the union to have more than $m^{25}+1$
  vertices, take a connected subset of the required size from that
  component. Such a subset can be obtained by repeatedly deleting
  leaves from a spanning tree. In this way, we obtain a set $B_1$
  of size $m^{25}$ or $m^{25}+1$ such that $H[B_1]$ has no isolated
  vertices.

  Let $\mathcal P_1$ be the collection of paths in $\mathcal P$
  whose endvertices in $\mathsf{Ext}(F_1)$ belong to $B_1$. We
  count the edges between $B_1$ and $V(G)\setminus
    \bigl(B_1\cup W'\cup V(\mathcal P_1)\bigr)$. 
  Since $\delta(G)\geq d/2$, for each $v\in B_1$, there are at most
  \[
    |B_1|+|W'|+|V(\mathcal P_1)|
    \leq m^{25}+1+60m^{50}+2m^{26}
    <\tfrac d4
  \]
  edges incident with $v$ having their other endpoint in
  $B_1\cup W'\cup V(\mathcal P_1)$. Therefore, there are at least $\tfrac d4|B_1|
    \geq \tfrac{dm^{25}}4$ edges from $B_1$ to $V(G)\setminus
    \bigl(B_1\cup W'\cup V(\mathcal P_1)\bigr)$. 
  By the assumption of Case 2, every vertex outside
  $B_1\cup W'\cup V(\mathcal P_1)$ has at most one neighbor in
  $B_1$. Hence the set $C:=
    N_G(B_1)\setminus
    \bigl(B_1\cup W'\cup V(\mathcal P_1)\bigr)$
  satisfies $|C|\geq \tfrac{dm^{25}}4$.
  
  Let $\widetilde W:=B_1\cup W'\cup V(\mathcal P_1)$. 
  Then $|\widetilde W|
    \leq m^{25}+1+60m^{50}+2m^{26}
    \leq 70m^{50}$.
  Applying Lemma~\ref{distance}, with
  $C,\mathsf{Ext}(F_3),\widetilde W$ playing the roles of
  $X_1,X_2,W$, respectively, we obtain a path $P$ of length at
  most $m$ connecting some $u\in C$ to some
  $y\in\mathsf{Ext}(F_3)$ while avoiding $\widetilde W$.
  Choose $w_1\in N_G(u)\cap B_1$. By the defining property of
  $B_1$, there exist $w_2\in B_1\setminus\{w_1\}$ and
  $x\in\mathsf{Int}(F_1)$ such that $xw_1,xw_2\in E(G)$.
  Let $Q_3$ be the $v_3,y$-path in $F_3$, and let $Q_4,Q_5$ be
  the paths in $F_2$ joining the respective endpoints of
  $P_{w_1},P_{w_2}$ in $\mathsf{Ext}(F_2)$ to $v_2$. Set
  \begin{equation*}
      R_1:=v_3Q_3yPu w_1P_{w_1}Q_4v_2, \ \ \ \ R_2:=v_3Q_3yPu w_1xw_2P_{w_2}Q_5v_2.
  \end{equation*}
  These are two $v_3,v_2$-paths whose lengths differ by $2$, and the shorter has length at most $\ell(Q_3)+\ell(P)+3+\ell(P_{w_1})+\ell(Q_4)\leq4m+9\leq5m$.

  To keep the end units disjoint from these paths, delete from $F_3$ every branch meeting $V(P_{w_1})\cup V(P_{w_2})\cup V(Q_3)$, and delete from $F_2$ every branch meeting $V(Q_4)\cup V(Q_5)$. At most $O(m)$ branches are deleted. Since each original unit has $2m^{28}$ branches, the remaining graphs contain subunits $F_3''$ and $F_2''$, with cores $v_3$ and $v_2$, respectively, both having parameters $(\tfrac32m^{28},\tfrac{d}{100},m+2)$.
  Let
  \[
    A^*
    :=\left(
      V(P)\cup V(P_{w_1})\cup\{x\}\cup V(P_{w_2})
      \cup V(Q_3)\cup V(Q_4)\cup V(Q_5)
    \right)\setminus\{v_2,v_3\}.
  \]
  Then
  $|A^*|\leq 8m\leq 2\cdot 5m$.
  Let $\mathcal P^*:=\{R_1,R_2\}$.
  Thus, $\mathcal A^*
    :=(v_3,v_2,F_3'',F_2'',A^*,\mathcal P^*)$
  is a $(5m,1)$-adjuster in $G-W$. Moreover, $|A^*|\leq 2\ell(\mathcal A^*)$,
  as desired.
\end{proof}

\section{Proofs of Lemmas~\ref{claim:unit} and~\ref{lem:web}}\label{secapp:webs}

\begin{lemma}\label{claim:star}
      Let $0<\tfrac{1}{n},\tfrac{1}{d}\ll \tfrac{1}{K}\ll \eps_1,\eps_2<\tfrac{1}{5}$ with $\tfrac{1}{4}\log^{1000}n\leq d\leq \tfrac{n}{K}$. Suppose that $G=(V,E)$ is an $n$-vertex $(\eps_1,\eps_2d)$-expander with $\delta(G)\geq \tfrac{d}{2}$. If $|L_G|< \tfrac{11}{50}d$, then for every set $X$ of size at most $dm^{150}$, the graph $G-X$ contains a collection of $m^{168}$ pairwise vertex-disjoint $\tfrac{d}{50}$-stars.
    \end{lemma}

\begin{proof}[Proof of Lemma~\ref{claim:star}]
      Consider a maximal family $\mathcal{S}$ of vertex-disjoint $\tfrac{d}{50}$-stars in $G-X$. Suppose that $|\mathcal{S}|<m^{168}$. Then $|V(\mathcal{S})|\leq \tfrac{dm^{168}}{4}$.
      Since $\delta(G)\geq \tfrac{d}{2}$ and every vertex outside $L_G$ has degree less than $dm^{31}$, a simple computation shows that 
$        d(G-X-V(\mathcal{S}))
        \geq
        \tfrac{n\delta(G)-2(|L_G|n+|X\cup V(\mathcal{S})|dm^{31})}{n}\geq \tfrac{d}{50}$.
      Thus $G-X-V(\mathcal{S})$ contains a $\tfrac{d}{50}$-star, contradicting the maximality of $\mathcal{S}$.
    \end{proof}
    
 \begin{proof}[Proof of Lemma~\ref{claim:unit}]
      Consider a maximal family $\mathcal{F}$ of vertex-disjoint $(2m^{28},\tfrac{d}{100},m+2)$-units in $G-X'$. Suppose that $|\mathcal{F}|<m^{109}$. Then $|V(\mathcal{F})|\leq \tfrac{dm^{137}}{25}$. Since $m^{138}+\tfrac{m^{168}}{2}<m^{168}$ for sufficiently large $m$, Lemma~\ref{claim:star} yields pairwise vertex-disjoint stars $S_1,\dots,S_{m^{138}}$ and
        $T_1,\dots,T_{ m^{168}/2}$
      in $G-X'-V(\mathcal{F})$, with respective centers $v_1,\dots,v_{m^{138}},u_1,\dots,u_{ m^{168}/2}$.
      Indeed,
$        |X'\cup V(\mathcal{F})|
        \leq dm^{110}+\tfrac{dm^{137}}{25}
        <dm^{150}$.
      Let $Z:=\{v_1, \ldots, v_{m^{138}}\}\cup\{u_1, \ldots, u_{ m^{168}/2}\}$
      and let $\mathcal P$ be a maximal family of internally vertex-disjoint paths $P_{ij}$ in
        $G-X'-V(\mathcal F)$
      satisfying the following conditions.
      \stepcounter{propcounter}
      \begin{enumerate}[label = ({\bfseries \Alph{propcounter}\arabic{enumi}})]
          \rm
        \item\label{num-cla1} For every pair $(i,j)$, the collection $\mathcal{P}$ contains at most one $v_i,u_j$-path, and every such path has length at most $m+2$.
          \rm
        \item\label{num-cla2} No path in $\mathcal{P}$ contains a vertex of $Z$ as an internal vertex.
      \end{enumerate}

      We claim that some center $v_i$ is connected by paths in $\mathcal{P}$ to at least $4m^{29}$ distinct centers $u_j$. Suppose otherwise. Then $|\mathcal{P}|\leq 4m^{29}\cdot m^{138}=4m^{167}$
      and $|V(\mathcal{P})|\leq 4m^{167}(m+2)\leq 5m^{168}$.
      Let
      \[
        V_0:=
        \left(\cup_{i\in[m^{138}]}(V(S_i)\setminus\{v_i\})\right)
        \setminus V(\mathcal{P}),
      \]
      and let $U_0$ be the set of leaves of those stars $T_j$ whose centers are not endpoints of paths in $\mathcal{P}$. Then
      $|V_0|\geq \tfrac{d}{50}m^{138}-5m^{168}
        >\tfrac{dm^{138}}{100}$
      and $
        |U_0|
        \geq \tfrac{d}{50}
        \left(m^{168}/2-4m^{167}\right)
        >\tfrac{dm^{138}}{100}.
      $
      Moreover,
      \begin{equation*}         |X'|+|V(\mathcal{F})|+|Z|+|\mathsf{Int}(\mathcal{P})|
        \leq
dm^{110}+\tfrac{dm^{137}}{25}+m^{138}+\tfrac{m^{168}}{2}+5m^{168}\leq
\tfrac14\rho\left(\tfrac{dm^{138}}{100}\right)
        \tfrac{dm^{138}}{100},
      \end{equation*}
      where the last inequality holds for sufficiently large $n$. Applying Lemma~\ref{distance} with $(X_1,X_2,W)=(V_0,U_0,X'\cup V(\mathcal{F})\cup Z\cup\mathsf{Int}(\mathcal{P}))$,
      we obtain a path of length at most $m$ joining some vertex of $V_0$ to some vertex of $U_0$ while avoiding $X'\cup V(\mathcal{F})\cup Z\cup\mathsf{Int}(\mathcal{P})$. Extending this path by the corresponding two star edges gives a new admissible $v_i,u_j$-path, contradicting the maximality of $\mathcal{P}$.

      Therefore, some center $v_i$ is connected by paths in $\mathcal{P}$ to at least $4m^{29}$ distinct centers $u_j$. By the pigeonhole principle, we can choose $2m^{28}$ distinct centers, say $u_1,\ldots,u_{2m^{28}}$, such that the corresponding paths $P_{i,j}$ all have the same length. The total number of vertices used by these paths is at most
        $2m^{28}(m+2)\leq \tfrac{d}{100}$.
      Hence every corresponding star $T_j$ retains at least $\tfrac{d}{100}$ leaves not used by these paths. These stars, together with the corresponding paths to $v_i$, form a $(2m^{28},\tfrac{d}{100},m+2)$-unit in $G-X'-V(\mathcal{F})$, contradicting the maximality of $\mathcal{F}$.
    \end{proof}

     \begin{proof}[Proof sketch of Lemma~\ref{lem:web}]
     This is a special case of Lemma~4.6 in~\cite{Yang1}. We highlight only the differences. First, the definition of a web in~\cite{Yang1} does not require the internal paths to have the same length, whereas ours does. This additional requirement can be met by starting with a larger collection of paths and applying the pigeonhole principle. Second, Lemma~4.6 in~\cite{Yang1} assumes a local density condition that robustly guarantees the existence of stars. Here we use Lemma~\ref{claim:star} instead.
\end{proof}

      \section{Proofs of Lemmas~\ref{lem:cycleadjuster} and~\ref{vesp:radj}}\label{appen:proofofdmradjuster}
      Before proving Lemmas~\ref{lem:cycleadjuster} and~\ref{vesp:radj}, we introduce some necessary tools. Here we replace the subdivision-free condition with the condition of being $TK_{d}^{(1)\ast c}$-free, so we invoke Lemma~\ref{lem:densebipartite}. The proof of Lemma~\ref{lem:manyexpansion} below follows that of Lemma~3.11 in~\cite{Liu-Mon-JAMS}; another difference is that, if its shortest cycle is odd, we invoke Lemma~\ref{lem:oddcyclelengthinexpander}.

\begin{lemma}\label{lem:densebipartite}
    Let $d,c\in\mathbb{N}$, and let $G$ be a graph containing vertex-disjoint sets $U$ and $W$ such that $|U|\ge c|W|^2$ and every vertex of $U$ has at least $d$ neighbors in $W$. Then $G$ contains a $TK_d^{(1)\ast c}$.
\end{lemma}

\begin{proof}
    Construct a maximal multigraph $H$ on $W$ whose edges have distinct labels in $U$, where an edge $xy$ labelled by $u$ requires $x,y\in N_W(u)$ and every pair has multiplicity at most $c$. If $H$ is the $c$-fold complete multigraph, any $d$ vertices of $W$ give the required subdivision. Otherwise,
    \(
      e(H)<c\binom{|W|}{2}\le |U|,
    \)
    so some $u\in U$ is unused. By maximality, every pair in $\binom{N_W(u)}2$ has multiplicity $c$. Choose $S\subseteq N_W(u)$ with $|S|=d$; the $c$ distinct labels on each pair of $S$ are common neighbors of that pair, and hence form a copy of $TK_d^{(1)\ast c}$ with branch set $S$.
\end{proof}   
We also need the following version of connecting lemma. 
      \begin{lemma}[\cite{Liu-Mon-JAMS}]\label{lem:smalldiameterlemma2}%3.4
        For each $0<\varepsilon_1,\varepsilon_2<1$, there exists
        $d_0=d_0(\varepsilon_1,\varepsilon_2)$ such that the following holds for each $d_0\le d\le n$ and $x\ge1$.
        Let $G$ be an $n$-vertex $(\varepsilon_1,\varepsilon_2 d)$-expander with $\delta(G)\ge d-1$.
        Let $A,B\subseteq V(G)$ with
        $|A|,|B|\ge x$,
        and let
        $W\subseteq V(G)\setminus (A\cup B)$
        satisfy
        $|W|\log^3 n \le 10x$.
        Then there is a path from $A$ to $B$ in $G-W$ of length at most
        $\tfrac{40}{\varepsilon_1}\log^3 n$.
      \end{lemma}

The proof of the lemma below follows that of Lemma 3.11 in~\cite{Liu-Mon-JAMS}; the only difference is that if the shortest cycle in Lemma~\ref{lem:manyexpansion} is an \emph{odd} cycle, then we invoke Lemma \ref{lem:oddcyclelengthinexpander}.
\begin{lemma}\label{lem:manyexpansion}%3.11.  
For each $k\in\mathbb{N}$ and any $0<\varepsilon_1,\varepsilon_2<1/20$, there exists $d_0=d_0(\varepsilon_1,\varepsilon_2,k)$ such that the following holds for each $d_0\le d \le n$. Suppose that $G$ is an $n$-vertex $(\varepsilon_1,\varepsilon_2 d)$-expander with $\delta(G)\ge d-1$. Let $m=\tfrac{800}{\varepsilon_1}\log^3 n$ and $m'=\tfrac{m}{20}=\tfrac{40}{\varepsilon_1}\log^3 n$.
Let $C$ be a shortest odd cycle in $G$ if one exists, and a shortest cycle in $G$ otherwise. Let $x_1,\dots,x_k$ be distinct vertices of $G$.
For each $i,j\in [k]$, let $D_{i,j}\in [1,e^{5(\log\log n)^{200}}]$.
Then there are graphs $F_{i,j}\subseteq G$, $i,j\in [k]$, such that the following hold.
\begin{itemize}
\item For each $i,j\in [k]$, $F_{i,j}$ is a $(D_{i,j},5m')$-expansion centered at $x_i$ that contains no vertex other than $x_i$ in $V(C)\cup \{x_1,\dots,x_k\}$.
\item The sets $V(F_{i,j})\setminus \{x_i\}$, $i,j\in [k]$, are pairwise disjoint.
\end{itemize}
\end{lemma}

We are now ready to prove Lemma~\ref{lem:cycleadjuster}.

\begin{proof}[Proof of Lemma~\ref{lem:cycleadjuster}]
We prove only the case in which an odd cycle exists. When all cycles are even, the argument is similar.
Let $|V(\Gamma)|=2\ell_0+1$.
Since $\delta(G)\ge d-1$, Lemma~\ref{lem:oddcyclelengthinexpander} gives $2\ell_0+1\le m/20$, provided $d\ge d_0(\varepsilon_1,\varepsilon_2)$ is sufficiently large. Pick vertices $x_3,x_4\in V(\Gamma)$ that are at distance $\ell_0$ from each other on $\Gamma$, and let the two $x_3,x_4$-paths on $\Gamma$ be $R_1$ and $R_2$, where $R_1$ is the shorter path. Let $D_{1,1}=D_{2,1}=D$, $D_{1,2}=D_{3,1}=m^3D$, and $D_{2,2}=D_{4,1}=m^2D$.
For all $(i,j)\in[4]\times[4]$ not specified above, set $D_{i,j}=1$.
Applying Lemma~\ref{lem:manyexpansion} with $k=4$, $C=\Gamma$, and the vertices
$x_1,x_2,x_3,x_4$, we obtain graphs $F_{i,j}$, $i,j\in[4]$.
In particular, retaining only $F_{1,1},F_{1,2},F_{2,1},F_{2,2},F_{3,1},F_{4,1}$, we have the following properties:
\begin{itemize}
    \item for $(i,j)\in[2]\times[2]$ or $(i,j)\in\{3,4\}\times\{1\}$,
    the graph $F_{i,j}$ is a $(D_{i,j},m/4)$-expansion centered at $x_i$
    in $G$ which contains no vertices other than $x_i$ in
    $\{x_1,x_2,x_3,x_4\}\cup V(\Gamma)$;
    \item the sets $V(F_{i,j})\setminus\{x_i\}$,
    for $(i,j)\in[2]\times[2]$ or $(i,j)\in\{3,4\}\times\{1\}$, are pairwise disjoint.
\end{itemize}

Now, \[|V(\Gamma)\cup V(F_{1,1}\cup F_{2,1}\cup F_{2,2}\cup F_{4,1})|\le m+2D+2m^2D\le \tfrac{10m^3D}{\log^3 n}.\]
Since $|F_{1,2}|=|F_{3,1}|=m^3D$, Lemma~\ref{lem:smalldiameterlemma2} yields a path $P$ of length at most $m/20$ from $V(F_{1,2})$ to $V(F_{3,1})$ with no vertices in $\bigl(V(\Gamma)\cup V(F_{1,1}\cup F_{2,1}\cup F_{2,2}\cup F_{4,1})\bigr)\setminus\{x_1,x_3\}$. As $F_{1,2}$ is a $(D_{1,2},m/4)$-expansion of $x_1$ and $F_{3,1}$ is a $(D_{3,1},m/4)$-expansion of $x_3$, we can extend $P$ using vertices from $V(F_{1,2}\cup F_{3,1})$ to obtain an $x_1,x_3$-path, say $P'$, of length at most $11m/20$, with no vertices in $\bigl(V(\Gamma)\cup V(F_{1,1}\cup F_{2,1}\cup F_{2,2}\cup F_{4,1})\bigr)\setminus\{x_1,x_3\}$.

Next, observe that \[|V(\Gamma)\cup V(P')\cup V(F_{1,1})\cup V(F_{2,1})|\le m+3m+1+2D\le \tfrac{10m^2D}{\log^3 n}. \]
Since $|F_{2,2}|=|F_{4,1}|=m^2D$, Lemma~\ref{lem:smalldiameterlemma2} gives a path $Q$ of length at most $m/20$ from $V(F_{2,2})$ to $V(F_{4,1})$ with no vertices in $\bigl(V(\Gamma)\cup V(P')\cup V(F_{1,1})\cup V(F_{2,1})\bigr)\setminus\{x_2,x_4\}$. As $F_{2,2}$ is a $(D_{2,2},m/4)$-expansion of $x_2$ and $F_{4,1}$ is a $(D_{4,1},m/4)$-expansion of $x_4$, we can extend $Q$ using vertices from $V(F_{2,2}\cup F_{4,1})$ to obtain an $x_2,x_4$-path, say $Q'$, of length at most $11m/20$, with no vertices in $\bigl(V(\Gamma)\cup V(P')\cup V(F_{1,1})\cup V(F_{2,1})\bigr)\setminus\{x_2,x_4\}$.

Now set $v_1=x_1, v_2=x_2, F_1=F_{1,1}, F_2=F_{2,1}$, and $A=V(P'\cup Q'\cup R_1\cup R_2)\setminus\{v_1,v_2\}$. Then $|A|\le 2(11m/20+1)+2\ell_0+1 \le 5m/2$, and $A$ is disjoint from $V(F_1)\cup V(F_2)$. Let $\ell = \ell(P'\cup R_1\cup Q')$. Then $P'\cup R_1\cup Q'$ and $P'\cup R_2\cup Q'$ are $v_1,v_2$-paths in $G[A\cup\{v_1,v_2\}]$ of lengths $\ell$ and $\ell+1$, respectively. Thus, $(v_1,F_1,v_2,F_2,A)$ is a $(D,m/4,1)$-odd-adjuster, as required.
\end{proof}

              Next, we show the existence of adjusters. (See Definition~\ref{defn:evenadjuster} for the definition of even and odd adjusters.)

\begin{lemma}\label{vesp:1adj}
Let $0<\eps_0\le 1/20$. There exists $\eps_1>0$ such that, for any $0<\eps_2\ll\eps_0$, there exists $d_0=d_0(\eps_0,\eps_1,\eps_2)$ such that the following holds for each $d_0\le d\le\log^{1000}n$. Suppose that $G$ is an $n$-vertex $TK_{(1-\eps_0)d}^{(1)\ast d^9}$-free $(\eps_1,\eps_2d)$-expander with $\delta(G)\ge d$, and let $m=\tfrac{800}{\eps_1}\log^3n$. If $\log^{10}n\le D\le e^{2(\log\log n)^{100}}$ and $W\subseteq V(G)$ satisfies $|W|\le4D$, then $G-W$ contains a $(D,m/2,1)$-adjuster or a $(D,m/2,1)$-odd-adjuster.
\end{lemma}

\begin{proof}[Proof sketch]
We follow the proof of Lemma~4.3 in~\cite{Liu-Mon-JAMS}. There are two differences. First, we remove the bipartiteness condition on $G$; consequently, we may obtain odd adjusters supplied by Lemma~\ref{lem:cycleadjuster}. Second, to ensure that the average degree remains $\Omega(d)$ after removing $W$, we require only the weaker hypothesis of $TK_{(1-\eps_0)d}^{(1)\ast d^9}$-freeness. Indeed, let $W_0=\{v\in V(G)\setminus W:d_G(v,W)\ge (1-\eps_0)d\}$. Note that $|W_0|\leq d^9|W|^2\le d^9e^{10(\log\log n)^{200}}\le e^{20(\log\log n)^{200}}$. Otherwise, Lemma~\ref{lem:densebipartite}, applied with $U=W_0$ and $W=W$, implies that $G$ contains a $TK_{(1-\eps_0)d}^{(1)\ast d^9}$, a contradiction.
Therefore, we know that the number of edges in $G-W$ is at least \[(n-|W|-|W_0|)\cdot\eps_0d\cdot\tfrac{1}{2}\ge \bigl(n-e^{(\log\log n)^{200}}-e^{20(\log\log n)^{200}}\bigr)\cdot\eps_0 d\cdot\tfrac{1}{2}\ge \tfrac{\eps_0 dn}{4},\]
as desired. Apart from these two differences, the rest of the proof of Lemma 4.3 in~\cite{Liu-Mon-JAMS} applies.
\end{proof}

Iterative application of Lemma~\ref{vesp:1adj} then yields Lemma~\ref{vesp:radj}.

            \end{appendix}

            \end{document}